\documentclass[a4paper,11pt]{article}

\usepackage{amsmath,amsthm,amssymb,mathrsfs,tikz,subfigure}
\usepackage[english]{babel}
\usepackage[utf8]{inputenc}
\usepackage{fancyhdr,lmodern,indentfirst,graphicx,newlfont}

\numberwithin{equation}{section}

\newcommand{\ot}{\otimes}
\newcommand{\id}{\text{id}}
\newcommand{\Z}{\mathbb{Z}}
\newcommand{\R}{\mathbb{R}}
\newcommand{\N}{\mathbb{N}}
\newcommand{\C}{\mathbb{C}}
\newcommand{\G}{\mathbb{G}}
\newcommand{\HH}{\mathbb{H}}

\newcommand{\aut}{\mathbb{G}^{aut}(B,\psi)}
\newcommand{\autd}{\mathbb{G}^{aut}(B,\psi,d)}

\newcommand{\um}{\frac{1}{2}}

\newcommand{\wprthg}{H_{(B,\tpsi)}^+(\G)}
\newcommand{\wprhg}{H_{(B,\psi)}^+(\G)}
\newcommand{\wprhgd}{H_{(B,\psi,d)}^+(\G)}
\newcommand{\wprd}{\widehat{\Gamma}\wr_{*} \mathbb{G}^{aut}(B,\psi)}

\newcommand{\wprg}{\G\wr_{*} \mathbb{G}^{aut}(B,\psi)}

\newcommand{\wprgd}{\G\wr_{*} \mathbb{G}^{aut}(B,\psi,d)}
\newcommand{\wprgdc}{C(\G)\ast_w C(\G^{aut}(B,\psi,d))}
\newcommand{\Tr}{\text{Tr}}

\newcommand{\irr}{\text{Irr}}
\newcommand{\rep}{\text{Rep}}
\newcommand{\Hom}{\text{Hom}}
\newcommand{\tpsi}{\tilde{\psi}}

\newcommand{\End}{\text{End}}

\newcommand{\cg}{C(\mathbb{G})}
\newcommand{\caut}{C(\mathbb{G}^{aut}(B,\psi))}
\newcommand{\cautd}{C(\mathbb{G}^{aut}(B,\psi,d))}
\newcommand{\Gq}{\mathbb{G}}

\newcommand{\idbb}{\id_{B\ot B'}}
\newcommand{\uq}{\ot 1_{M/I}}

\newcommand{\cc}{\mathscr{C}}
\newcommand{\dd}{\mathscr{D}}
\newcommand{\cct}{\tilde{\mathscr{C}}}
\newcommand{\ddt}{\tilde{\mathscr{D}}}

\theoremstyle{plain}
\newtheorem{theorem}{Theorem}[section]

\newtheorem{lemma}[theorem]{Lemma}
\newtheorem{prop}[theorem]{Proposition}
\newtheorem{cor}[theorem]{Corollary}
\theoremstyle{definition}
\newtheorem{deff}[theorem]{Definition}
\newtheorem{example}[theorem]{Example}
\theoremstyle{remark}
\newtheorem{notaz}{Notation}
\newtheorem{remark}{Remark}

\begin{document}

\begin{center}
{\LARGE\bf  The free wreath product of a compact quantum group by a quantum automorphism group}
\bigskip

{\sc Pierre Fima and Lorenzo Pittau}
\end{center}

\begin{abstract}
\noindent Let $\G$ be a compact quantum group and $\aut$ be the quantum automorphism group of a finite dimensional C*-algebra $(B,\psi)$. In this paper, we study the free wreath product $\wprg$. First of all, we describe its space of intertwiners and find its fusion semiring. Then, we prove some stability properties of the free wreath product operation. In particular, we find under which conditions two free wreath products are monoidally equivalent or have isomorphic fusion semirings. We also establish some analytic and algebraic properties of $\wprg$. As a last result, we prove that the free wreath product of two quantum automorphism groups can be seen as the quotient of a suitable quantum automorphism group.
\end{abstract}

%%%%%%%%%%%%%%%%%%%%%%%%%%%%%%%%%%%%%%%%%%%%%%%%%%
\section*{Introduction}
%%%%%%%%%%%%%%%%%%%%%%%%%%%%%%%%%%%%%%%%%%%%%%%%%%

The theory of compact quantum groups was introduced by Woronowicz in \cite{wor87} and further developed in \cite{wor91,wor98}. This framework includes many of the well known and deeply investigated quantum groups constructed by Drinfeld and Jimbo some years before as a $q$-deformation of the universal enveloping algebras of classical Lie algebras (see \cite{dri86,dri87,jim85}). New examples of compact quantum groups, which are not contained in previous theories, were found by Wang in \cite{wan93,wan98} and are the so-called free quantum groups. In the first of these two papers, he defines the quantum analogue of the classical orthogonal and unitary groups as the noncommutative versions of the C*-algebras $C(O_n)$ and $C(U_n)$ respectively.

In \cite{wan98}, Wang studies the classical notion of automorphism group of a finite space and introduces its quantum counterpart. In particular, in the framework of noncommutative geometry, the role of the finite measured space is played by a finite dimensional C*-algebra $B$ endowed with a faithful state $\psi$. The quantum automorphism group $\aut$ is then defined as the universal object in the category of compact quantum groups acting on $B$ and leaving $\psi$ invariant. An important quantum group of this family is the quantum symmetric group $S_n^+$; it is obtained choosing as space $\C^n$ endowed with the usual trace. 

The first relevant results concerning these families of compact quantum groups introduced by Wang are due to Banica. In \cite{ban96,ban97,ban99a,ban02}, he classified their irreducible representations and found the fusion rules. An intermediate essential step towards these results is the understanding of the spaces of intertwiners between tensor products of the fundamental representation. One of the possible approaches to this problem is the combinatorial one. More precisely, it consists in describing the intertwiners by making use of special classes of noncrossing partitions. This technique was used for the first time in \cite{ban99a,ban02}, where the intertwiners are described by using Temperley-Lieb diagrams, but its systematic use has begun only some years later by Banica and Speicher \cite{bs09}. They proved that many quantum groups allow such a combinatorial description and the notion of easy quantum group was introduced to denote the class of all compact quantum groups with this property. 
A complete classification of the (orthogonal) easy quantum groups was later done by Weber and Raum and this allowed also to discover new compact quantum groups (see \cite{web13,rw13}).

As for the classical groups, another way to find new quantum groups is through different types of product operations. In \cite{wan95}, Wang introduced the notion of free product and showed that the irreducible representations of the quantum group so constructed are the alternate tensor products of the irreducible representations of the factors. 
A second important operation to consider in this framework is the free wreath product by the quantum symmetric group. In the classical case the wreath product of a group $G$ by $S_n$, denoted $G\wr S_n$, is defined thanks to the natural action of $S_n$ on a set of $n$ copies of $G$. Bichon in \cite{bic04} introduced the free quantum version $\G\wr_* S_n^+$ by using an action of $S_n^+$ on $n$ copies of $\G$. 
Moreover, always in analogy with the classical case, the free wreath product, from a geometric point of view, allows to describe the quantum symmetry group of $n$ copies of a finite graph in terms of the symmetry group of the graph and of $S_n^+$. 
In \cite{bic04} a first easy example was analysed in detail and its representation theory was described: the free wreath product $\widehat{\Z}_2\wr_* S_n^+$. A more general analysis of the compact quantum groups so obtained was done in three successive steps. In \cite{bv09}, Banica and Vergnioux studied $\widehat{\Z}_s\wr_* S_n^+$ and found its irreducible representations and fusion rules in the case $n\geq 4$. This investigation was continued by Lemeux in \cite{lem14}, where he considered the free wreath product of a discrete group $\Gamma$ by $S_n^+$, $n\geq 4$. Finally, in \cite{lt14}, Lemeux and Tarrago presented an even more general result; they analysed the case of the free wreath product of a compact matrix quantum group of Kac type $\G$ by $S_n^+$ and found its representation category by using an argument of monoidal equivalence.

In this last paper it is also possible to find many results concerning the properties of the algebras associated to a free wreath product. In particular, by using a result from \cite{dcfy14}, it has been proved that, if $\G$ has the Haagerup property, also the von Neumann algebra $L^\infty(\G\wr_*S_n^+)$ has this property. Moreover, the reduced C*-algebra $C_r(\G\wr_*S_n^+)$ is exact if $C_r(\G)$ is exact. In \cite{lem14}, Lemeux proved the simplicity and uniqueness of the trace for the reduced C*-algebra in the discrete case. His argument, based on the so-called Powers method and on the simplicity and uniqueness of the trace of $S_n^+$ ($n\geq 8$) proved by Brannan in \cite{bra13}, was extended by Wahl in \cite{wah14} to the general case of a matrix pseudogroup of Kac type.

In \cite{pit14}, the second author defined and analysed the free wreath product of a discrete group by a quantum automorphism group, giving a partial generalization of Bichon's definition. First of all, he gave a new description of the spaces of intertwiners of $\aut$ which makes use of non crossing partitions instead of Temperley-Lieb diagrams. Thanks to this result, he was then able to find the irreducible representations and the fusion rules of $\wprd$, when $\psi$ is a $\delta$-form and $\dim(B)\geq 4$. Moreover, he proved some properties of the associated operator algebras, extending the results of \cite{lem14}.

We recall that $G(X)$, the group of symmetries of a graph $X$ with $n$ vertices, can be seen as a quotient of the symmetric group $S_n$. Moreover, when dealing with the usual notion of wreath product, there is a sort of geometric interpretation thanks to formulas such as
$$G(X\ast Y)\cong G(X)\wr G(Y)$$
for a suitable notion of product $\ast$ and only for graphs satisfying certain conditions.

The main motivation behind the definition of free wreath product by a quantum permutation group given by Bichon is to find a quantum analogue of these results. More precisely, if we denote by $G^+(X)$ the group of the quantum symmetries of a finite graph $X$, we have formulas such as
$$G^+(X\ast Y)\cong G^+(X)\wr_* G^+(Y)$$
for a suitable notion of $\ast$ and some assumptions on the graphs (see \cite{bic04,bb07,cha15}).\\

In this paper, we introduce and study the free wreath product construction by a quantum automorphism group of a \textit{finite quantum graph}, generalizing all the previous construction (free wreath product by the quantum permutation group or the quantum automorphism group of a finite metric space or graph etc). Our main motivation behind such a construction is to generalize, unify and simplify 	all the known formulas of the type $G^+(X\ast Y)\cong G^+(X)\wr_* G^+(Y)$.

After a preliminary section (section $1$), We introduce in section $2$ the notion of finite quantum graph $(B,\psi,d)$, where $B$ is a finite dimensional C*-algebra, $\psi$ is a faithful state on $B$ and $d\in\mathcal{L}(B)$ and its quantum automorphism group $\G^{aut}(B,\psi,d)$ which generalizes and unifies the notions of quantum permutation group, quantum autmorphism group of a finite dimensional C*-algebra as well as quantum automorphism group of a finite graph or more generally of a finite metric space. Then we introduce the free wreath product $\G\wr_*\aut$ of a compact quantum group $\G$ by $\G^{aut}(B,\psi,d)$ and prove that we have a compact quantum group.

When $d=1$, we describe in section $3$ its spaces of intertwiners by means of specially decorated noncrossing partitions, when $\psi$ is a $\delta$-form and $\dim(B)\geq 4$. From this, by adapting (in section $4$) a monoidal equivalence argument used in \cite{lt14}, we can deduce the irreducible representations and the fusion rules in section $5$. We also prove in section $5$ a decomposition formula of any free wreath product $\G\wr_*\aut$ with a general faithful state $\psi$ in a free product of free wreath wreath products $\G_i\wr_* \G^{aut}(B_i,\psi_i)$ where the state $\psi_i$ is a $\delta_i$-form for all $i$. Such a formula allows to describe the irreducible representations and the fusion rules of any feee wreath product, not necessarily one with a $\delta$-form.

The description of the intertwiners is also used to prove some stability properties of the free wreath product with a $\delta$-form in section $6$. Firstly, we show that the relation of monoidal equivalence between two compact quantum group is preserved by the free wreath product operation. More precisely, if $\G_1\simeq_{mon}\G_2$ and $\aut\simeq_{mon} \G^{aut}(B',\psi')$ then $\G_1\wr_*\aut\simeq_{mon}\G_2\wr_* \G^{aut}(B',\psi')$.
A second result concerns the fusion semiring of the free wreath product $\wprg$; we prove that it is completely determined by the fusion semiring of $\G$.

We analyse some properties of the associated operator algebras in section $7$. Namely, it will be proved by using some results from \cite{dcfy14} that the dual of $\wprg$ has the central ACPAP if $\widehat{\G}$ has the central ACPAP; it follows that in this case the corresponding von Neumann algebra has the Haagerup property. Similarly, a result from \cite{vv07} allows us to show that the exactness of $\widehat{\G}$ implies the exactness of the dual of $\wprg$. If $\psi$ is moreover a trace, we can generalize an argument of Lemeux \cite{lem14}, based on a result of Brannan \cite{bra13} and on the Powers method adapted by Banica \cite{ban97}, in order to show the simplicity and uniqueness of the trace for the reduced C*-algebra of $\wprg$.

In the last section, we finally obtain the main result of this paper: we will show that the C*-algebra of the free wreath product $\G^{aut}(B',\psi')\wr_{*}\aut$ is isomorphic to a suitable explicit and very simple quotient of $C(\G^{aut}(B\otimes B',\psi\otimes \psi'))$. This generalizes the results of \cite{bb07,cha15}.

%%%%%%%%%%%%%%%%%%%%%%%%%%%%%%%%%%%%%%%%%%%%%%%%%%
\subsubsection*{Acknowledgement}The authors are grateful to T. Banica for stimulating discussions. P.F. is supported by ANR grants NEUMANN and OSQPI.
%%%%%%%%%%%%%%%%%%%%%%%%%%%%%%%%%%%%%%%%%%%%%%%%%%

%%%%%%%%%%%%%%%%%%%%%%%%%%%%%%%%%%%%%%%%%%%%%%%%%%
\section{Preliminary results}
%%%%%%%%%%%%%%%%%%%%%%%%%%%%%%%%%%%%%%%%%%%%%%%%%%

This first section contains important definitions and results which will be used throughout all the paper. In the first part, we recall the basic definitions of the theory of compact quantum groups developed by Woronowicz in \cite{wor87,wor91,wor98}. Then, we introduce the language of noncrossing partitions and we present some known results about the quantum automorphism group of a finite dimensional C*-algebra endowed with a state. In particular, we recall the description discovered in \cite{pit14} of the spaces of intertwiners by making use of noncrossing partitions instead of the original description \cite{ban99a} in terms of Temperley-Lieb diagrams. First of all, we fix some notations.

\begin{notaz}
Given a Hilbert space $H$, we write $\mathcal{L}(H)$ the C*-algebra of bounded linear operators on $H$. The symbol $\otimes$ will be used to denote the tensor product of Hilbert spaces, the minimal tensor product of C*-algebras or the tensor product of von Neumann algebras, depending on the context. The symbol $*$ denotes the maximal free product of unital C*-algebras and the symbol $\widehat{*}$ denotes the free product of compact quantum groups (\cite{wan95}). We use the symbol $\underset{\text{red}}{*}$ for the reduced free product of unital C*-algebras and we usually make reference to the state in the notation.
\end{notaz}

Given a unital C*-algebra $A$ and elements $v\in \mathcal{L}(H)\ot A$, $w\in\mathcal{L}(K)\ot A$, where $H$ and $K$ are  finite dimensional Hilbert spaces, we define the vector space of intertwiners
$$\Hom(v,w)=\{T\in \mathcal{L}(H,K)\,|\,(T\otimes 1)v=w(T\otimes 1)\}.$$
Note that if $T\in \Hom(u,v)$ and $S\in \Hom(v,w)$ then $ST\in \Hom(u,w)$. Also, if $v,w$ are unitaries and $T\in \Hom(v,w)$ then $T^*\in \Hom(w,v)$. We define the tensor product of such elements:
$$v\ot w:=v_{13}w_{23}\in \mathcal{L}(H)\ot \mathcal{L}(K)\ot A\simeq \mathcal{L}(H\ot K)\ot A,$$
where we use the leg-numbering notation. Let $(e_i)_i$ be an orthonormal basis of $H$ with associated matrix units $e_{ij}\in\mathcal{L}(H)$ and $u\in\mathcal{L}(H)\ot A$. We write $u=\sum e_{ij}\ot u_{ij}$ and we call $u_{ij}\in A$ the \textit{coefficients} of $u$.

Given $u=\sum e_{ij}\ot u_{ij}\in\mathcal{L}(H)\ot A$, we define $\widetilde{u}=\sum e_{ij}^*\ot u_{ij}^*\in\mathcal{L}(\overline{H})\ot A$, where $e_{ij}$ are the matrix units associated to the dual basis $(e_i^*)_i$ of the orthonormal basis $(e_i)_i$ of $H$ and $\overline{H}$ is the dual of $H$. Observe that the linear maps $s_u\,:\,\C\rightarrow H\ot \overline{H}$, $z\mapsto z\Omega$ and $\widetilde{s}_u\,:\,\C\rightarrow \overline{H}\ot H$, $z\mapsto z\widetilde{\Omega}$, where $\Omega=\frac{1}{\dim(H)}\sum_ie_i\ot e_i^*$ and $\widetilde{\Omega}=\frac{1}{\dim(H)}\sum_ie_i^*\ot e_i$ are isometries in $\Hom(1,u\ot\widetilde{u})$ and  $\Hom(1,\widetilde{u}\ot u)$ respectively, where $1$ denotes the unit of $A$. Even starting with a unitary $u$, the element $\widetilde{u}$ is non necessarily invertible. However, the following well known remark will be useful to check that $\widetilde{u}$ is indeed invertible.
 
 \begin{remark}\label{RmkInvertible}
For any unitary $u\in\mathcal{L}(H)\ot A$ the element $\widetilde{u}$ is invertible whenever there exists a unitary $v\in\mathcal{L}(K)\ot A$, with $K$ a Hilbert space such that $\dim(K)=\dim(H)$ and a non-degenerate\footnote{$t\in\mathcal{L}(\C, K\ot H)$ is \textit{non-degnerate} if $\{(\id\ot\xi^*)(t(1))\,:\,\xi\in H\}=K$.} $t\in \Hom(1,v\ot u)$. Indeed, if such $v$ and $t$ exist we may consider $Q:=(t^*\ot\id_{\overline{H}})\circ(\id_{K}\ot s_u)\in \Hom(v,\widetilde{u})$. Note that, for all $\xi\in H$, one has $Q^*\xi^*=\frac{1}{\dim(H)}(\xi^*\ot\id)(t(1))$ hence $Q^*$ is surjective, since $t$ is non-degenerate. Moreover, since $\dim(\overline{H})=\dim(K)$, $Q^*\in\mathcal{L}(\overline{H},K)$ is invertible so $Q$ also and we have $\widetilde{u}=(Q^{-1}\ot 1)v(Q\ot 1)$, which implies that $\widetilde{u}$ is invertible, since $v$ is unitary.
 \end{remark}

There are several equivalent definitions of a Woronowicz's compact quantum group. It can be defined, for example, as a pair $\G=(\cg,\Delta)$ where $\cg$ is a unital C*-algebra and $\Delta:\cg\longrightarrow \cg\otimes \cg$ a coassociative comultiplication such that $\Delta(\cg)(\cg\otimes 1)$ and $\Delta(\cg)(1\otimes \cg)$ are dense in $\cg\otimes \cg$. For our purpose we will adopt another definition.
\begin{deff}\label{cqg2}
A compact quantum group $\G$ is a pair $(\cg,\Delta)$, where $\cg$ is a unital C*-algebra and $\Delta:\cg\longrightarrow \cg\otimes \cg$ a $\ast$-homomorphism together with a family of unitaries $u^{\alpha}\in \mathcal{L}(H_\alpha)\ot\cg$ for $\alpha\in I$, where $H_\alpha$ is a finite dimensional Hilbert space, such that:
\begin{itemize}
\item the $*$-subalgebra generated by the coefficients of $u^\alpha$, for $\alpha\in I$, is dense in $\cg$;
\item $(\id\ot\Delta)(u^\alpha)=(u^\alpha)_{12}(u^\alpha)_{13}$ for all $\alpha\in I$;
\item $\widetilde{u^\alpha}$ is invertible for all $\alpha\in I$.
\end{itemize}
\end{deff}

A (finite dimensional and unitary) representation of the compact quantum group $(\cg,\Delta)$ is a unitary $v\in \mathcal{L}(H)\ot \cg$, where $H$ is a finite dimensional Hilbert space, such that $(\id\ot\Delta)(v)=v_{12}v_{13}$. Two representations $v$ and $w$ are called (unitarily) equivalent if there exists a unitary in $\Hom(v,w)$ and a representation $u$ is called irreducible if $\Hom(u,u)=\C\id$. We write $1_\G$ the trivial representation of $\G$, it is the unit of $C(\G)$ viewed as a one-dimensional representation.

We denote by $\irr(\G)$ the set of equivalence classes of the irreducible representations of $\G$ and, for all $\alpha\in\irr(\G)$, we choose a representative $u^\alpha\in\mathcal{L}(H_\alpha)\ot C(\G)$. The linear span of the coefficients of the $u^\alpha$, for $\alpha\in\irr(\G)$, which is a unital dense $*$-subalgebra of $C(\G)$, will be denoted by Pol$(\G)$. In all the paper we will assume that $\G$ is given in its maximal form meaning that $C(\G)$ is the enveloping C*-algebra of Pol$(\G)$. We will denote by $h\in C(\G)^*$ the Haar state of $\G$ and by $C_r(\G)$ (resp. $L^\infty(\G)$) the C*-algebra (resp. von Neumann algebra) generated by the GNS representation of $h$.

It is known that for any representation $u$ there exists a unique (up to equivalence) representation $\overline{u}$ such that $\Hom(1_\G,u\ot\overline{u})\neq\{0\}\neq \Hom(1_\G,\overline{u}\ot u)$. Such a representation is called the contragredient or conjugate of $u$. Actually the representation $\overline{u}$ is obtained by unitarizing the invertible element $\widetilde{u}\in\mathcal{L}(\overline{H})\ot C(\G)$ defined previously. The contragredient representation is well defined at the level of $\irr(\G)$: for $\alpha\in\irr(\G)$, we write $\overline{\alpha}$ the class of the contragredient of $\alpha$. Let $s_\alpha\in \Hom(1_\G,\alpha\ot\overline{\alpha})\setminus\{0\}$ and $J_\alpha\,:\, H_\alpha\rightarrow H_{\overline{\alpha}}$ the invertible antilinear map defined by $\langle J_\alpha\xi,\eta\rangle=\langle s_\alpha(1),\xi\ot\eta\rangle$ for all $\xi\in H_\alpha$, $\eta\in H_{\overline{\alpha}}$. Define $R_\alpha:=J_\alpha^*J_\alpha\in\mathcal{L}(H_\alpha)$. We may and will always choose $s_\alpha$ and $s_{\overline{\alpha}}$ normalized such that $\Vert s_\alpha(1)\Vert=\Vert s_{\overline{\alpha}}(1)\Vert$ and $J_{\overline{\alpha}}=J_\alpha^{-1}$. With such a normalization $R_\alpha$ is uniquely determined, $\Tr(R_\alpha)=\Vert s_\alpha(1)\Vert=\Tr(R_\alpha^{-1})$, where $\Tr$ is the unique trace and $\mathcal{L}(H_\alpha)$ such that $\Tr(1)=$dim$(H_\alpha)$ and $R_{\overline{\alpha}}=R_\alpha^{-1}$. The number $\Tr(R_\alpha)$ is called the quantum dimension of $\alpha$ and is denoted dim$_q(\alpha)$. A compact quantum group $\G$ is said to be of Kac type if the antilinear map $J_\alpha$ is anti-unitary for all $\alpha\in\irr(\G)$. In particular, this means that $J_\alpha^*=J_\alpha^{-1}$, therefore $R_\alpha=\id$ for all $\alpha$. This is also equivalent to dim$(\alpha)=$dim$_q(\alpha)$ for all $\alpha$.

Now, we recall some useful definitions and results from the category theory. In \cite{wor88}, Woronowicz proved that the finite dimensional representations of a compact quantum group $\G$ form a rigid C*-tensor category denoted $\mathcal{R}(\G)$ and called the \textit{representation category} of $\G$. More precisions on this notion can be found in \cite{nt13}.

\begin{deff}\label{defrigid}
A C*-tensor category $\mathscr{C}$ is called \textit{rigid} if for any $U\in Ob(\mathscr{C})$ there exist $\bar{U}\in Ob(\mathscr{C})$ and two morphisms $R\in\Hom(1,\bar{U}\otimes U)$, $\bar{R}\in\Hom(1, U\otimes \bar{U})$ such that $$(\bar{R}^*\otimes \id_U)(\id_U\otimes R)=\id_U \quad 
(R^*\otimes \id_{\bar{U}})(\id_{\bar{U}}\otimes \bar{R})=\id_{\bar{U}}$$
The object $\bar{U}$ is called the \textit{conjugate} of $U$ and the conditions satisfied by $R,\bar{R}$ are the \textit{conjugate equations}.
\end{deff}

Note that with the normalization explained above we have $(s_{\overline{\alpha}}^*\ot\id)(\id\ot s_\alpha)=\id$ for all $\alpha\in\irr(\G)$. This is why a $\mathcal{R}(\G)$ is a rigid C*-tensor category. In the abstract context of rigid C*-tensor category the Frobenius reciprocity still holds.

\begin{theorem}\label{teofrobrec}
Let $U$ be an object of a C*-tensor category with conjugate $\bar{U}$ and let $R$, $\bar{R}$ be the morphisms solving the conjugate equations. Then, the linear application $\Hom(U\ot V,W)\longrightarrow\Hom(V,\bar{U}\ot W)$ given by $T\mapsto(\id_{\bar{U}}\ot T)(R\ot \id_V)$ is an isomorphism. Similarly, we have $\Hom(V\ot U,W)\cong\Hom(V,W\ot \bar{U})$
\end{theorem}

Another notion particularly important in this paper is that of monoidal equivalence. Two compact quantum groups $\G$ and $\mathbb{H}$ are called \textit{monoidally equivalent} if their representations categories are unitarily monoidally equivalent i.e. if there exists unitary tensor functors $F\,:\,\mathcal{R}(\G)\rightarrow \mathcal{R}(\mathbb{H})$ and $G\,:\,\mathcal{R}(\mathbb{H})\rightarrow\mathcal{R}(\G)$ such that both $FG$ and $GF$ are naturally monoidally isomorphic to the identity functors and the natural isomorphisms $FG\simeq\id$ and $GF\simeq\id$ are unitary (see \cite{nt13}). By restricting the unitary tensor functor $F$ to the irreducible representations of $\G$ we get another, more concrete, equivalent definition.

\begin{deff}[\cite{bdrv06}]\label{defmoneq}
Let $\G_1$ and $\G_2$ be two compact quantum groups. They are monoidally equivalent (written $\G_1\simeq_{mon}\G_2$) if there exists a bijection $\phi:\irr(\G_1)\longrightarrow\irr(\G_2)$, $\phi(1_{\G_1})=1_{\G_2}$ such that, for any $k,l\in\N$ and for any $\alpha_i,\beta_j\in\irr(\G)$, $1\leq i\leq k$, $1\leq j\leq l$, there is an isomorphism $$\phi:\Hom_{\G_1}(\alpha_1\otimes...\otimes\alpha_k; \beta_1\otimes...\otimes \beta_l)\longrightarrow\Hom_{\G_2}(\phi(\alpha_1)\otimes...\otimes\phi(\alpha_k); \phi(\beta_1)\otimes...\otimes \phi(\beta_l))\text{ such that:}$$
\begin{enumerate}
\item[i)] $\phi(\id)=\id$,
\item[ii)] $\phi(f\otimes g)=\phi(f)\otimes \phi(g)$,
\item[iii)] $\phi(f^*)=\phi(f)^*$,
\item[iv)] $\phi(fg)=\phi(f)\phi(g)$ for $f,g$ composable morphisms.
\end{enumerate}
\end{deff}

The proof of a monoidal equivalence between two compact quantum groups can be simplified by making use of the following proposition.

\begin{prop}\label{propgenmoneq}
Let $\cc,\dd$ be two rigid C*-tensor categories, possibly non complete with respect to direct sums and subobjects. Let $\cct,\ddt$ be their completions. If $F:\cc\longrightarrow \dd$ is a unitary monoidal equivalence between the two categories $\cc$ and $\dd$, then there exists a unitary monoidal equivalence $\tilde{F}:\cct\longrightarrow \ddt$ which extends $F$.
\end{prop}

This is a standard result in category theory, we refer to \cite{bor94a,bor94b} for the proof and for further details.

Since we will widely rely on the notion of noncrossing partition in what follows, we recall some standard definitions and notations.

\begin{deff}
Let $k,l\in\N$. Let $p=P_1\sqcup P_2\sqcup...\sqcup P_t$ be a partition of the set $I_{k+l}=\{1,...,k+l\}$. The subsets $P_i$, $i=1,...,t$ are called the blocks of the partition. The partition $p$ is said to be a noncrossing partition if, for every possible choice of elements $r_1<r_2<r_3<r_4$, $r_j\in I_{k+l}$ such that $r_1$ and $r_3$ belong to the same block, then $r_2$ and $r_4$ belong to different blocks.
As we fixed $k$ and $l$, such a noncrossing partition $p$ can be represented by a diagram with $k$ upper points and $l$ lower points constructed as follows:
\begin{itemize}
\item consider two horizontal lines and draw $k$ points on the upper one and $l$ points on the lower one
\item number the $k$ upper points from $1$ to $k$ and from the left to the right
\item number the $l$ lower points from $k+1$ to $k+l$ and from the right to the left
\item connect to each other the points in a same block of the partition by drawing strings in the part of the plane between the two lines
\end{itemize}
From the non crossing condition, it follows that the strings which connect points of different blocks can be drawn in such a way that they do not intersect. We denote $NC(k,l)$ the set of noncrossing partitions between $k$ upper points and $l$ lower points. The total number of blocks of $p\in NC(k,l)$ is denoted $b(p)$.
\end{deff}

We have three fundamental operations which allow to create new noncrossing partitions from existing ones. 

\begin{deff}\label{opdiag}
Let $p\in NC(k,l)$ and $q\in NC(v,w)$. The tensor product $p\otimes q$ is the noncrossing partition obtained by horizontal concatenation of $p$ and $q$. The 
The adjoint $p^*$ is the noncrossing partition corresponding to the diagram in $NC(l,k)$ obtained by reflecting the diagram $p$ with respect to an horizontal line between the two rows of points. If $l=v$, we define the composition $qp$ as the diagram obtained by identifying the lower points of $p$ with the upper points of $q$ and by removing all the blocks made up of points of the central line only. We refer to such blocks as \textit{central blocks} and their number is denoted $cb(p,q)$.
\end{deff}

For our computations, we need to introduce also the non-standard notion of cycle. More precisions can be found in \cite{pit14}.
\begin{deff}
The composition operation between the noncrossing partitions $p$ and $q$ can create some (closed) cycles which will not be present in the final noncrossing partition. Heuristically, they are the closed surfaces which appear when two or more central points are connected both in the upper and in the lower noncrossing partition. 
In a more formal way, the number of cycles is 
\begin{equation*}
cy(p,q):=l+b(qp)+cb(p,q)-b(p)-b(q)
\end{equation*}
\end{deff}

\begin{example}
In order to clarify the definitions and the notations above consider the following example with $p\in NC(1,3)$ and $q\in NC(3,2)$.

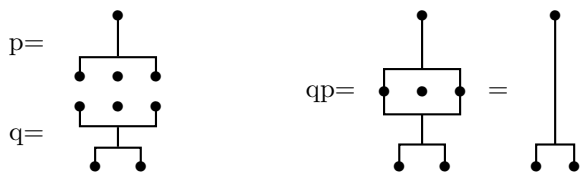
{\centering
 \begin{tikzpicture}[thick,font=\small]
    \path (.3,.6) node{p=}
    		  (.3,-.6) node{q=}
    		  (1.5,1) node{$\bullet$} node[above](a){}

          (1,.2) node{$\bullet$} node[below](d){}
          (1.5,.2) node{$\bullet$} node[below](e){}
          (2,.2) node{$\bullet$} node[below](f){}

          (1,-.2) node{$\bullet$} node[above](dd){}
          (1.5,-.2) node{$\bullet$} node[above](ee){}		  
          (2,-.2) node{$\bullet$} node[above](ff){}
          
          (1.2,-1) node{$\bullet$} node[below](l){}
          (1.8,-1) node{$\bullet$} node[below](m){}
          
          (5.5,1) node{$\bullet$}
          (5,0) node{$\bullet$}
          (5.5,0) node{$\bullet$}
          (6,0) node{$\bullet$}

          (5.2,-1) node{$\bullet$}
          (5.8,-1) node{$\bullet$}

          (4.3,-0.03) node{qp=}
          (6.5,0) node{=}
          (7.25,1) node{$\bullet$}node[above](q){}
          (7,-1) node{$\bullet$}node[below](w){}
          (7.5,-1) node{$\bullet$}node[below](z){};

    \draw (d) -- +(0,+0.4) -| (f);
    \draw (1.5,.47) -- (1.5,1);
    \draw (dd) -- +(0,-0.4) -| (ff);
    \draw (l) -- +(0,+0.4) -| (m);
    \draw (1.5,-.47) -- (1.5,-0.73);
    
    \draw (5,0) -- +(0,+0.3) -| (6,0);
    \draw (5.5,.3) -- (5.5,1);
    
    \draw (5,0) -- +(0,-0.3) -| (6,0);
    \draw (5.2,-1) -- +(0,+0.3) -| (5.8,-1);
    \draw (5.5,-.3) -- (5.5,-0.7);
    
    \draw (7,-1) -- +(0,+0.3) -| (7.5,-1);
    \draw (7.25,1) -- +(0,-1.7);
\end{tikzpicture}
\par}
\noindent
We have $b(p)=b(q)=2$, $b(qp)=1$ and $cb(p,q)=1$.\\
Then, the number of cycles is $cy(p,q)=3+1+1-2-2=1$.
\end{example}

A particularly important notion used in this paper is that of quantum automorphism group. Before recalling its definition and some related results, let us introduce some notations.

$B$ will always denote a finite dimensional C*-algebra. We denote by $m:B\otimes B\longrightarrow B$ the multiplication on $B$ and $\eta\,:\,\C\rightarrow B$ the unit of $B$. Since $B$ is a finite dimensional C*-algebra, it is of the form $B=\bigoplus_{T=1}^c M_{n_T}(\C)$. Let $\mathscr{B}=\{(e_{ij}^T)_{i,j=1,\ldots,n_T},T=1,\ldots,c\}$ be a basis of matrix units and note that $m(e_{ij}^T \otimes e_{kl}^S)=\delta_{jk}\delta_{TS} e_{il}^T$ and and $\eta(1)=\sum_{T=1}^c\sum_{i=1}^{n_T}e_{ii}^T$.

In what follows, $\psi:B\longrightarrow \C$ will always denote a faithful state on $B$. Note that $\psi$ induces a scalar product on $B$ defined by $\langle x,y\rangle:=\psi(y^*x)$ for all $x,y\in B$. This scalar product turns $B$ into a Hilbert space. The state $\psi$ is called a $\delta$\textit{-form} if $mm^*=\delta\cdot \id_B$, where $m^*$ is the adjoint with respect to the scalar product induced by $\psi$ and $\delta>0$. Recall that $\psi$ is of the form $\psi=\bigoplus_{T=1}^c \Tr(Q_T\cdot)$ for a suitable family $Q_T\in M_{n_T}(\C)$, $Q_T>0$, $\sum_T \Tr(Q_T)=1$, where $\Tr$ is the unique trace of $M_{n_T}(\C)$ such that $\Tr(1)=n_T$. It is easy to check that $\psi$ is a $\delta$-form if and only if $\Tr(Q_T^{-1})=\delta$ for all $T=1,\ldots,c$.

In what follows, we will always assume that the matrix units $e_{ij}^{T}$ are associated with a basis of $\C^{n_T}$ which diagonalizes $Q_{T}$. We will denote $Q_{i,T}$ the eigenvalue in position $(i,i)$ of the matrix $Q_T$ written with respect to this fixed diagonalizing basis. Note that $\psi(e_{ij}^T)=\Tr(Q_T e_{ij}^T)=\delta_{ij}Q_{i,T}$ and the basis $\mathscr{B}$ is then always orthogonal with respect to the scalar product induced by $\psi$. By normalizing $\mathscr{B}$ we obtain the orthonormal basis $\mathscr{B}'=\{b_{ij}^T | b_{ij}^T=\psi(e_{jj}^T)^{-\um} e_{ij}^T=Q_{j,T}^{-\um} e_{ij}^T,\; i,j=1,\ldots,n_T,\;T=1,\ldots,c\}$
which will be widely used in this paper.

We recall below the definition of the quantum automorphism group of $(B,\psi)$ and show how to describe its intertwiners by means of noncrossing partitions. More details on the construction of this object and on its representation theory can be found in \cite{ban99a, wan98, pit14}.

\begin{deff}\label{autb}
The universal unital C*-algebra $\caut$ generated by the coefficients of $u\in\mathcal{L}(B)\ot \caut$ with the relations
$$ (i)\,\,u\text{ is unitary }\quad (ii)\,\,m\in \Hom(u^{\otimes 2},u)\quad (iii)\,\,\eta\in \Hom(1,u),$$
and endowed with the unique unital $*$-homomorphism
$$\Delta\,:\,\caut\rightarrow\caut\ot\caut$$
such that $(\id\ot\Delta)(u)=u_{12}u_{23}$ is the compact quantum automorphism group of $(B,\psi)$ and it is denoted by $\aut$.
\end{deff}

In order to describe the intertwiners of a quantum automorphism group by using noncrossing partitions, we need to associate to each non-crossing partition a linear, as it was done in \cite{pit14}.

%\begin{notaz}\label{coeff}
Let $p\in NC(k,l)$ and associate to each point an element of the basis $\mathscr{B}'$. Let $(b_{i_1j_1}^{\alpha_1},...,b_{i_kj_k}^{\alpha_k})$ and $(b_{r_1s_1}^{\beta_1},...,b_{r_ls_l}^{\beta_l})$ be the elements associated to the upper and lower points respectively. 
Let $ij$ and $\alpha$ be the multi-index notation for the different indices, in particular $ij=((i_1,j_1),...,(i_k,j_k))$ and $\alpha=(\alpha_1,...,\alpha_k)$; similarly we define $rs$ and $\beta$.

Denote by $b_v,v=1,...,m$ the blocks of $p$, and let $b_v^\uparrow$ ($b_v^\downarrow$) be the ordered product of the matrix units corresponding to the upper (lower) points of the block $b_v$. This product is assumed to be the identity matrix if there are no upper (lower) points in the block. Define
\begin{equation}\label{coeffd}
\delta_p^{\alpha,\beta}(ij,rs):=\prod_{v=1}^m \psi((b_v^\downarrow)^*b_v^\uparrow)
\end{equation}
%\end{notaz}

\begin{deff}\label{map2}
The linear map corresponding to $p\in NC(k,l)$ is 
$$T_p:B^{\otimes k}\longrightarrow B^{\otimes l},\,\,\,b_{i_1j_1}^{\alpha_1} \otimes...\otimes b_{i_kj_k}^{\alpha_k}\mapsto\sum_{r,s,\beta}\delta_p^{\alpha,\beta}(ij,rs)b_{r_1s_1}^{\beta_1} \otimes...\otimes b_{r_ls_l}^{\beta_l}$$
\end{deff}

The operations between noncrossing partitions introduced in Definition \ref{opdiag} are compatible with the corresponding operations between linear maps.

\begin{prop}[\cite{pit14}]\label{propmap}
Let $p\in NC(l,k)$ and $q\in NC(v,w)$. We have:
\begin{enumerate}
\item $T_{p\otimes q}=T_p\otimes T_q$
\item $T_p^*=T_{p^*}$
\item if $k=v$ there are two possible cases:
\begin{enumerate}
\item[a.] if $\psi$ is a (unital) $\delta$-form, then $T_{qp}=\delta^{-cy(p,q)}T_q T_p$
\item[b.] if $\tpsi$ is the (possibly non unital) 1-form $\delta\psi$, then $T_{qp}=\tpsi(1)^{-cb(p,q)}T_q T_p$
\end{enumerate}
\end{enumerate}
\end{prop}

The assertion $(3b)$ is not included in the formulation of this proposition presented in \cite{pit14}. It can be proved with the same techniques used to show the assertion $(3a)$. In this case, the correction coefficient depends on $\tpsi(1)$ because, in general, $\tpsi$ is not unital, while the dependence on $\delta$ disappears because it is a 1-form.

The space of intertwiners between tensor products of the fundamental representation of the quantum automorphism group can be described as follows.

\begin{theorem}[\cite{pit14}]\label{intertautb}
Let $B$ be a $n$-dimensional C*-algebra, $n\geq 4$ and $u$ be the fundamental representation of $\aut$. Then, for all $k,l\in\N$, $\Hom(u^{\otimes k},u^{\otimes l})=span\{T_p\,|\,p\in NC(k,l)\}$. Moreover, maps corresponding to distinct noncrossing partitions are linearly independent.
\end{theorem}

%%%%%%%%%%%%%%%%%%%%%%%%%%%%%%%%%%%%%%%%%%%%%%%%%%%%%%%%
\section{Definition of the free wreath product $\wprgd$}
%%%%%%%%%%%%%%%%%%%%%%%%%%%%%%%%%%%%%%%%%%%%%%%%%%%%%%%%
In this section we define the main object of this paper: the free wreath product of a compact quantum group by the quantum automorphism group of a finite quantum graph. The definition is based on the same idea already used in the case of the dual of a discrete group (see \cite{pit14}) but it needs to be adapted to this new context. In the case of the free wreath product by the quantum automorphism group $\aut$, we will describe the spaces of intertwiners by means of specially decorated noncrossing partitions. This will be fundamental in order to prove a monoidal equivalence result, from which the fusion rules and some other properties will be deduce.

First of all we introduce the notion of finite quantum graph. We will then define the quantum automorphism group of such a quantum graph as a quotient of $\aut$.

\begin{deff}
Let $B$ be a finite dimensional C*-algebra endowed with a faithful state $\psi$. Let $d\in\mathcal{L}(B)$. The triple $(B,\psi ,d)$ will be called a finite quantum graph.
\end{deff}

\begin{example}\label{ExClassicalGraph}
When $X$ is a classical finite graph, with say $n$-vertices, we get a finite quantum graph $(\C^n,\psi,d)$, where $\psi(x)=\frac{1}{n}\sum_ix_i$, for $x=(x_1,\dots,x_n)\in\C^n$, and $d\in M_n(\C)$ is the adjacency matrix.
\end{example}

\begin{prop}
Let $(B,\psi ,d)$ be a finite quantum graph and $\cautd$ be the universal unital C*-algebra generated by the coefficients of $u_d\in\mathcal{L}(B)\ot \cautd,$ with relations
$$ (i)\,\,u_d\text{ is unitary }\quad (ii)\,\,m\in \Hom(u_d^{\otimes 2},u_d)\quad (iii)\,\,\eta\in \Hom(1,u_d)\quad (iv)\,\,d\in \End(u_d).$$
There exists a unique unital $*$-homomorphism
$$\Delta_d\,:\,\cautd\rightarrow \cautd\ot \cautd\text{ such that }(\id\ot\Delta_d)(u_d)=(u_d)_{12}(u_d)_{13}.$$
Moreover the pair $(\cautd,\Delta_d)$ is a compact quantum group, called the quantum automorphism group of the finite quantum graph $(B,\psi,d)$ and it is denoted $\autd$.\end{prop}

\begin{proof}
The uniqueness of $\Delta_d$ being obvious, let us show the existence. Let $u\in\mathcal{L}(B)\ot\caut$ be the fundamental representation of $\aut$ and $I_d\subset \caut$ be the closed two sided $\ast$-ideal generated by the relation $d\in\End(u)$ so that $\cautd=\caut/I_d$. Let $\pi_d:\caut\rightarrow \caut /I_d$ be the canonical surjection so that $u_d=(\id\ot\pi_d)(u)$. Recall that $(\id_B\ot\Delta)(u)=u_{12}u_{13}$ where $u$ is the fundamental representation of $\aut$. Note that $I_d\subset\ker((\pi_d\ot\pi_d)\circ\Delta)$ since we have the following equation
$$(d\ot 1^{\ot 2})u_{12}u_{13}=u_{12}(d\ot 1^{\ot 2})u_{13}=u_{12}u_{13}(d\ot 1^{\ot 2}).$$
The inclusion above means that $\Delta$ factorizes to a map $\Delta_d$ such that $\Delta_d\circ\pi_d=(\pi_d\ot\pi_d)\circ\Delta$. It is then clear that $\Delta_d$ satisfies the required condition and that $(\cautd,\Delta_d)$ is a compact quantum group.
\end{proof}

\begin{example}
If $(B,\psi ,d)$ is the finite quantum graph associated to a classical finite graph $X$ as in Example \ref{ExClassicalGraph} then $\autd$ is the quantum automorphism group of $X$ (see \cite{bb07,cha15}). More generally, when $X$ is a finite metric space and $B=C(X)$, $\psi(f)=\frac{1}{\vert X\vert}\sum_{x\in X}f(x)$ and $d\in M_{\vert X\vert}(\C)$ is the matrix defined by $d_{ij}=d(i,j)$, where $d$ is the metric on $X$, then $\autd$ is the quantum automorphism group of $X$ defined in \cite{ban05}.
\end{example}

\begin{remark}\label{rem:trivd}
Let $u_k\in\mathcal{L}(H_k)\ot \caut$ be a complete set of representative of $\irr(\aut)$ such that $u_0=1$ and $u=u_0\oplus u_1$. Let $S:H_1\longrightarrow B$ be the unique isometry, up to $\mathbb{S}^1$. Recall that the orthogonal projections $\eta_B\eta_B^*$ and $SS^*$ satisfy $\eta_B\eta_B^*+SS^*=\id_B$ and $\End(u)=\C\eta_B\eta_B^*+\C SS^*$. Hence, whenever $d\in span\{\eta_B\eta_B^*,SS^*\}$, we have $\autd=\aut$.
\end{remark}

\begin{remark}
If $d\in\mathcal{L}(B)$ is normal, by writing $d=\sum_i\lambda_i p_i$ the spectral decomposition of $d$, where $p_i$ are the spectral projections, it is easy to check that the ideal generated by the relations $d\in\End(u)$ is equal to the ideal generated by the relations $p_i\in\End(u)$ for all $i$.
\end{remark}

\begin{prop}\label{intertwinersautd}
Let $V_d(k,l)\subset\mathcal{L}(B^{\ot k},B^{\ot l})$ be the set of linear combinations of composable products of maps of the form $\id_B^{\ot s}\ot   m_B\ot\id_B^{\ot r}$, $\id_B^{\ot s}\ot   m_B^*\ot\id_B^{\ot r}$, $\id_B^{\ot s}\ot   \eta_B\ot\id_B^{\ot r}$, $\id_B^{\ot s}\ot   \eta_B^*\ot\id_B^{\ot r}$, $\id_B^{\ot s}\ot   d\ot\id_B^{\ot r}$ and $\id_B^{\ot s}\ot   d^*\ot\id_B^{\ot r}$ which are in $\mathcal{L}(B^{\ot k},B^{\ot l})$.
Then, for all $k,l\in\N$, we have
$$\Hom(u_d^{\ot k},u_d^{\ot l})=V_d(k,l).$$
\end{prop}

\begin{proof}
The first inclusion ($\supseteq$) easily follows by observing that all the maps which generate $V_d(k,l)$ are morphisms by Definition \ref{deffwpg}. For the second inclusion ($\subseteq$) we need to use the Tannaka-Krein duality (see \cite{wor88}). Let $\cc$ be the category such that $Ob(\cc)=\N$ and $\Hom(k,l)=V_d(k,l)$. This category is essentially the category of the noncrossing partitions with some additional morphisms. It is then clear that $\cc$ is a concrete rigid monoidal C*-category. By the Tannaka-Krein duality we know that there exists a compact quantum group $\G=(C(\G),\Delta)$ with fundamental representation $v$ and such that $\Hom(v^{\ot k},v^{\ot l})=V_d(k,l)$. Because of the universality of the Tannaka-Krein construction, from the inclusion already proved it follows that there is a surjective map $\phi:C(\G)\longrightarrow \autd$ such that $(id_B\ot\phi)(v)=u_d$. In order to complete the proof we have to show that the map is an isomorphism. This follows from the universality of the  construction of $\autd$ after observing that the matrix $v$ is unitary and that $m\in\Hom(v^{\ot 2},v)$, $\eta\in\Hom(1,v)$ and $d\in\End(v)$ because of the definition of $V_d(k,l)$.
\end{proof}

We can now define the free wreath product of a compact quantum group by the quantum automorphism group of a finite quantum graph. Let $\G$ be a compact quantum group and $(B,\psi ,d)$ be a finite quantum graph. Recall that for each $\alpha\in \irr(\G)$, we have choosen a representative $u^\alpha\in\mathcal{L}(H_\alpha)\ot C(\G)$.

\begin{deff}\label{deffwpg}
Define $\wprgdc$ to be the universal unital C*-algebra generated by the coefficients of $a(\alpha)\in \mathcal{L}(B\otimes H_\alpha)\otimes (\wprgdc)$, $\alpha\in\irr(\G)$, with the relations:
\begin{itemize}
\item $a(\alpha)$ is unitary for any $\alpha\in \irr(\G)$.
\item $\forall \alpha,\beta,\gamma\in \irr(\G)$, $\forall S\in \Hom(\alpha\otimes\beta,\gamma)$\quad $$\widetilde{m\otimes S}:=(m\otimes S)\circ \Sigma_{23} \in\Hom (a(\alpha )\otimes a(\beta),a(\gamma))$$ where $\Sigma_{23}: B\otimes H_{\alpha}\otimes B \otimes H_{\beta}\longrightarrow B^{\otimes 2}\otimes (H_{\alpha}\otimes H_{\beta}), x_1\otimes x_2\otimes x_3 \otimes x_4 \mapsto x_1\otimes x_3\otimes x_2 \otimes x_4$ is the unitary map that exchanges the legs 2 and 3 in the tensor product.
\item $\eta\in\Hom (1,a(1_\G))$, where $1$ is the unity of $\wprgdc$ and $1_\G$ denote the trivial representations of $\G$.
\item $d\in\End(a(1_{\G}))$.
\end{itemize}
\end{deff}

\begin{remark}
When $\G=\widehat{\Gamma}$ is the dual of a discrete group $\Gamma$ all the irreducible representations of $\widehat{\Gamma}$ are one dimensional; therefore, the morphisms $S$ can be ignored, since they are scalar multiples of $\id_{\C}$ and we recover the original definition of free wreath product C*-algebra from \cite{pit14}.
\end{remark}

\begin{prop}\label{deffwp}
There exists a unique $\ast$-homomorphism $$\Delta:\wprgdc\longrightarrow(\wprgdc)\otimes(\wprgdc)$$ such that, for any $\alpha\in\irr(\G)$, $(\id\ot\Delta)(a(\alpha))=a(\alpha)_{(12)}a(\alpha)_{(13)}$. The pair $(\wprgdc,\Delta)$ is a compact quantum group, called the free wreath product of $\G$ by $\autd$ and denoted $\wprgd$ or $\wprhgd$. For all $\alpha\in\irr(\G)$, $a(\alpha)$ is a representation of $\wprhgd$, its contragredient representation is $a(\overline{\alpha})$ and dim$_q(a(\alpha))=$ dim$_q(\alpha)\sum_{T}$Tr$(Q_T)$Tr$(Q_T^{-1})$. Moreover, $\wprhgd$ is of Kac type if and only if $\G$ is of Kac type and $\psi$ is a trace.
\end{prop}

\begin{proof}Since the coefficients of $a(\alpha)$ generate the C*-algebra $A:=\wprgdc$, the uniqueness of $\Delta$ is obvious. Let us show the existence. For $\alpha\in\irr(\G)$, define the unitary $v(\alpha)=a(\alpha)_{(12)}a(\alpha)_{(13)}\in\mathcal{L}(B\ot H_\alpha)\ot (A\ot A)$. Let $S\in\Hom(\alpha\ot\beta,\gamma)$. We need to show that $(m\ot S)\Sigma_{23}\in \Hom(v(\alpha)\ot v(\beta),v(\gamma))$ and to do so, we will view the $a(\alpha)$ as three legs objects in $\mathcal{L}(B)\ot\mathcal{L}(H_\alpha)\ot A$. Since $(m\ot S)\Sigma_{23}\in \Hom(a(\alpha)\ot a(\beta),a(\gamma))$ we have

\begin{eqnarray*}
((m\ot S)\Sigma_{23})\ot 1_{A\ot A})(v(\alpha)\ot v(\beta))&=&((m\ot S)\Sigma_{23}\ot 1_{A\ot A})a(\alpha)_{(125)}a(\alpha)_{(126)}a(\beta)_{(345)}a(\beta)_{(346)}\\
&=&((m\ot S)\Sigma_{23}\ot 1_{A\ot A})a(\alpha)_{(125)}a(\beta)_{(345)}a(\alpha)_{(126)}a(\beta)_{(346)}\\
&=&a(\gamma)_{(123)}((m\ot S)\Sigma_{23}\ot 1_{A\ot A})a(\alpha)_{(126)}a(\beta)_{(346)}\\
&=&a(\gamma)_{(123)} a(\gamma)_{(124)}((m\ot S)\Sigma_{23}\ot 1_{A\ot A})\\
&=&v(\gamma)((m\ot S)\Sigma_{23}\ot 1_{A\ot A}).
\end{eqnarray*}

Moreover, $v(1_\G\eta\ot 1_{A\ot A}=a(1_{\G})_{(12)}a(1_{\G})_{(13)}(\eta\ot 1_{A\ot A})=\eta\ot 1_{A\ot A}$ and,
$$(d\ot 1_{A\ot A})v(A_\G)=(d\ot 1_{A\ot A})a(1_{\G})_{(12)}a(1_{\G})_{(13)}=a(1_{\G})_{(12)}(d\ot 1_{A\ot A})a(1_{\G})_{(13)}=v(1_\G)(d\ot 1^{\ot 2}).$$
Hence, the existence of $\Delta$ follows from the universal property of the C*-algebra $A$.

Let us now check the conditions of Definition \ref{cqg2} for the pair $(A,\Delta)$. It is clear that the matrices $a(\alpha)$ are unitary and we just proved that the comultiplication $\Delta$ exists. What is left is to show that the $\widetilde{a(\alpha)}$ are invertible. By Remark \ref{RmkInvertible} it suffices to show that there exists a non-degenerate $t_{\overline{\alpha}}\in \Hom(1,a(\overline{\alpha})\ot a(\alpha))$. Let $s_{\alpha}\in \Hom(1_\G,\alpha\ot\overline{\alpha})$ and $s_{\overline{\alpha}}\in \Hom(1_\G,\overline{\alpha}\ot\alpha)$ be the non-zero well-normalized intertwiners and define $t_{\overline{\alpha}}:=(\widetilde{m\ot s_{\overline{\alpha}}^*})^*\circ\eta\in \Hom(1,a(\overline{\alpha})\ot a(\alpha))$. Let us check that $t_{\overline{\alpha}}$ is non-degenerate. Take an orthonormal basis $(e^\alpha_i)_i$ of $H_\alpha$ which diaganalizes $R_\alpha:=J_\alpha^* J_\alpha$ and write $R_\alpha e^\alpha_i=\lambda_{\alpha,i}e^\alpha_i$. Note that $(e^{\overline{\alpha}})_i$ is an orthonormal basis of $H_{\overline{\alpha}}$, where $e^{\overline{\alpha}}_i:=(\lambda_{\alpha,i})^{-\frac{1}{2}}J_\alpha e^\alpha_i$. Moreover, since $\langle s_\alpha(1),e^\alpha_i\ot e^{\overline{\alpha}}_j\rangle=\langle J_\alpha e^\alpha_i,e^{\overline{\alpha}}_j\rangle=(\lambda_{\alpha,i})^{-\frac{1}{2}}\langle J_\alpha e^\alpha_i,J_\alpha e^\alpha_j\rangle=\sqrt{\lambda_{\alpha,i}}\delta_{i,j}$, we find $s_\alpha(1)=\sum_i\sqrt{\lambda_{\alpha,i}}e_i^\alpha\ot e_i^{\overline{\alpha}}$. Similarly we have $s_{\overline{\alpha}}(1)=\sum_i(\lambda_{\alpha,i})^{-\frac{1}{2}}e_i^{\overline{\alpha}}\ot e_i^\alpha$ and note that  dim$_q^2(\alpha)=\sum_i\lambda_{\alpha,i}=\sum_i(\lambda_{\alpha,i})^{-1}$.

An easy computation gives $t_{\overline{\alpha}}(1)=\sum_i(\lambda_{\alpha,i})^{-\frac{1}{2}}\Sigma_{23}(m^*(1_B)\ot e_i^{\overline{\alpha}}\ot e_i^\alpha)$. We consider the orthonormal basis $\mathcal{B}'=\{b^T_{ij}\}$ of $B=\oplus_{T=1}^c M_{n_T}(\C)$ introduced in Section $1$ and for which we obviously have $1_B=\sum_{T,i}Q_{i,T}^{\frac{1}{2}}b^T_{ii}$ and $m^*(b^T_{ij})=\sum_k Q_{k,T}^{-\frac{1}{2}}b_{ik}^T\ot b_{kj}^T$. It follows that $m^*(1_B)=\sum_{T,i,k}\sqrt{\frac{Q_{i,T}}{Q_{k,T}}}b_{ik}^T\ot b_{ki}^T$ hence, $t_{\overline{\alpha}}(1)=\sum_{i,j,k,T}\sqrt{\frac{Q_{j,T}}{\lambda_{\alpha,i}Q_{k,T}}}b_{jk}^T\ot e^{\overline{\alpha}}_i\ot b_{kj}^T\ot e^\alpha_i$. From the previous formulae it is clear that $t_{\overline{\alpha}}$ is non-degenerate. This shows that $(A,\Delta)$ is a compact quantum group.

Let us define $t_\alpha:=(\widetilde{m\ot s_{\alpha}^*})^*\circ\eta\in \Hom(1,a(\alpha)\ot a(\overline{\alpha}))$. As before, an easy computation gives $t_{\alpha}(1)=\sum_{i,j,k,T}\sqrt{\frac{\lambda_{\alpha,i}Q_{j,T}}{Q_{k,T}}}b_{jk}^T\ot e^\alpha_i\ot b_{kj}^T\ot e^{\overline{\alpha}}_i$. In particular, $t_\alpha$ is non-zero. This shows that the contragredient of $a(\alpha)$ is, up to equivalence, $a(\overline{\alpha})$. Moreover, denoting by $J_{a(\alpha)}$ the antilinear map $B\ot H_\alpha\rightarrow B\ot H_{\overline{\alpha}}$ associated to $t_\alpha(1)$ (and by $J_{a(\overline{\alpha})}$ the one associated with $t_{\overline{\alpha}}(1)$), we see from the formulas of $t_\alpha(1)$ and $t_{\overline{\alpha}}(1)$ that $J_{a(\alpha)}^{-1}=J_{a(\overline{\alpha})}$. Moreover, since
$\Vert t_\alpha(1)\Vert^2=\sum_{i,j,k,T}\frac{\lambda_{\alpha,i}Q_{j,T}}{Q_{k,T}}=$ dim$_q(\alpha)\sum_{j,k,T}\frac{Q_{j,T}}{Q_{k,T}}=$ dim$_q(\alpha)\sum_{T}$Tr$(Q_T)$Tr$(Q_T^{-1})=\Vert t_{\overline{\alpha}}(1)\Vert^2$, it follows that our pair of intertwiners $(t_\alpha,t_{\overline{\alpha}})$ is well-normalized. Hence, dim$_q(a(\alpha))=$ dim$_q(\alpha)\sum_{T}$Tr$(Q_T)$Tr$(Q_T^{-1})$ and
$$R_{a(\alpha)}(b^T_{jk}\ot e^\alpha_i)=\frac{\lambda_{\alpha,i}Q_{j,T}}{Q_{k,T}}b^T_{jk}\ot e^\alpha_i,\quad\text{where }R_{a(\alpha)}:=J_{a(\alpha)}^*J_{a(\alpha)}.$$
Note that, since the algebra generated by the coefficients of the representations $a(\alpha)$, for $\alpha\in\irr(\G)$ is dense in $A$, it follows from the general theory that $\wprhgd$ is of Kac type if and only if $R_{a(\alpha)}=\id$ for all $\alpha\in\irr(\G)$. By the explicit formula for $R_{a(\alpha)}$, it follows that $\wprhgd$ is of Kac type if and only if $\frac{\lambda_{\alpha,i}Q_{j,T}}{Q_{k,T}}=1$ for all $\alpha,T,i,j,k$. It is then easy to deduce the last assertion of the Proposition.\end{proof}

\begin{remark}
As observed in Remark \ref{rem:trivd}, if $d\in\End(u)$, where $u$ is the fundamental representation of $\aut$, the condition $d\in\End(a(1_{\G}))$ in the definition of the free wreath product is a consequence of the previous ones. In this case the free wreath product by $\G$ will be denoted $\wprg$ or $\wprhg$. In the following sections we will take into account only this case. The more general situation (with the endomorphism $d$ is non trivial) will be considered only in the last section, where we will give a geometric justification of Definition \ref{deffwpg}.
\end{remark}

\begin{remark}
The definition of the compact quantum group $\wprhgd$ implies that every irreducible representation can be obtained as a sub-representation of a suitable tensor product of the basic representations $a(\alpha)$, for  $\alpha\in\irr(\G)$, and actually this remark was used in the proof of the Kac type assertion of the previous Proposition.
\end{remark}

%%%%%%%%%%%%%%%%%%%%%%%%%%%%%%%%%%%%%%%%%%%%%%%%%%%%%%%%%%%%%%%%
\section{Spaces of intertwiners}
%%%%%%%%%%%%%%%%%%%%%%%%%%%%%%%%%%%%%%%%%%%%%%%%%%%%%%%%%%%%%%%%

In this section we want to combine and generalize some results from \cite{lt14} and from \cite{pit14}, in order to describe the spaces of intertwiners of the free wreath product $\wprhg$ by means of decorated noncrossing partitions.

\begin{deff}
Let $p\in NC(k,l)$ and decorate from left to right the $k$ upper points of $p$ with the representations of the tuple $\alpha=(\alpha_1,...,\alpha_k)\in\rep(\G)^k$  and the $l$ lower points with the representations of the tuple $\beta=(\beta_1,...,\beta_l)\in\rep(\G)^l$. Denote by $b_v$, $v=1,..,m$ the different blocks of $p$, let $U_v$ (resp. $L_v$) be the upper (reps. lower) points of the block $b_v$ and let $\alpha_{U_v}$ (resp. $\beta_{L_v}$) be the tensor product, ordered from left to right, of the representations which decorate the upper (resp. lower) points of $b_v$. By convention we define $\alpha_{U_v}=1_{\G}$ (resp. $\beta_{L_v}=1_{\G}$) if $b_v$ does not have upper point (resp. lower points).  Similarly, let $H_{U_v}$ (resp. $H_{L_v}$) be the tensor product, ordered from left to right, of the Hilbert spaces associated to these representations. Again, if $b_v$ does not have upper point (resp. lower points) we set $H_{U_v}=\C$ (resp. $H_{L_v}=\C$).

The partition $p$ is said to be \textit{well decorated} if, for every block $b_v$, one has $\Hom (\alpha_{U_v},\beta_{U_v})\neq\{0\}$. We denote by $NC_{\G}(\alpha_1,...,\alpha_k;\beta_1,...,\beta_l)$ the set of well decorated noncrossing partitions.
\end{deff}

Let $p\in NC_{\G}(\alpha_1,...,\alpha_k;\beta_1,...,\beta_l)$. For each block $b_v$ of $p$ choose a non-zero morphism $S_v\in \Hom (\alpha_{U_v},\beta_{U_v})$ and define $S:=\bigotimes_{v=1}^m S_v:\bigotimes_{v=1}^m H_{U_v}\longrightarrow \bigotimes_{v=1}^m H_{L_v}$ the tensor product, ordered from left to right. Then, it is quite natural to consider the map $$T_p\otimes S:B^{\otimes k}\otimes\bigotimes_{v=1}^m H_{U_v}\longrightarrow B^{\otimes l}\otimes\bigotimes_{v=1}^m H_{L_v}.$$

To ultimately obtain an intertwiner in $\Hom(a(\alpha_1)\otimes...\otimes a(\alpha_k),a(\beta_1)\otimes...\otimes a(\beta_l))$ we need to reorder the spaces. To do so, we define $s_{p,U}:\bigotimes_{i=1}^k(B\otimes H_{\alpha_i})\longrightarrow B^{\otimes k}\otimes\bigotimes_{v=1}^m H_{U_v}$ and $s_{p,L}:\bigotimes_{j=1}^l(B\otimes H_{\beta_j})\longrightarrow B^{\otimes l}\otimes\bigotimes_{v=1}^m H_{L_v}$ as the applications which reorder the spaces associated to the upper and lower points of $p$ respectively.

\begin{deff}
The map in $\mathcal{L}(\bigotimes_{i=1}^k(B\otimes H_{\alpha_i}),\bigotimes_{j=1}^l(B\otimes H_{\beta_j}))$ associated to a decorated noncrossing partition $p\in NC_{\G}(\alpha_1,...,\alpha_k;\beta_1,...,\beta_l)$ endowed with a morphism $S\in\bigotimes_{v=1}^m \Hom(\alpha_{U_v},\beta_{L_v})$ is $$T_{p,S}:=s_{p,L}^{-1}\circ (T_p\otimes S)\circ s_{p,U}$$
\end{deff}

The next result easily follows from Proposition \ref{propmap}.

\begin{prop}\label{propmapdec}
Let $p\in NC_{\G}(\alpha_1,...,\alpha_k;\beta_1,...,\beta_l)$ be a decorated noncrossing partition endowed with the morphism $S\in\bigotimes_{v=1}^m \Hom(\alpha_{U_v},\beta_{L_v})$. Similarly, let $q\in NC_{\G}(\alpha'_1,...,\alpha_{k'}';\beta_1',...,\beta_{l'}')$ be a decorated noncrossing partition endowed with the morphism $S'\in\bigotimes_{v=1}^{m'}\Hom(\alpha'_{U_v},\beta'_{L_v})$. One has:
\begin{enumerate}
\item $T_{p\otimes q, S\otimes S'}=T_{p,S}\otimes T_{q,S'}$
\item $T_{p,S}^*=T_{p^*,S^*}$
\item if $l=k'$ and $\beta_i=\alpha'_i$ for all $i=1,...,k'$ there are two possibilities:
\begin{enumerate}
\item[a.] if $\psi$ is a (unital) $\delta$-form, then $T_{qp,S'S}=\delta^{-cy(p,q)}T_{q,S'} T_{p,S}$,
\item[b.] if $\tpsi$ is a (possibly non unital) 1-form, then $T_{qp,S'S}=\tpsi(1)^{-cb(p,q)}T_{q,S'} T_{p,S}$.
\end{enumerate}
\end{enumerate}
\end{prop}

\begin{proof}
The first relation follows from
\begin{eqnarray*}
T_{p,S}\otimes T_{q,S'}&=&(s_{p,L}^{-1}\circ (T_p\otimes S)\circ s_{p,U})\otimes (s_{q,L}^{-1}\circ (T_q\otimes S')\circ s_{q,U})\\
&=&(s_{p,L}^{-1}\otimes s_{q,L}^{-1})\circ (T_p\otimes S\otimes T_q\otimes S')\circ (s_{p,U}\otimes s_{q,U})\\
&=&(s_{p,L}^{-1}\otimes s_{q,L}^{-1})(\id\otimes \sigma_1^{-1}\otimes \id) (\id\otimes \sigma_1\otimes \id)(T_p\otimes S\otimes T_q\otimes S') (\id\otimes \sigma_2^{-1}\otimes \id)\\
&=&(\id\otimes \sigma_2\otimes \id) (s_{p,U}\otimes s_{q,U})=s_{p\otimes q,L}^{-1}\circ (T_{p\otimes q}\otimes S\otimes S')\circ s_{p\otimes q,U}=T_{p\otimes q,S\otimes S'},
\end{eqnarray*}

where $\sigma_1$ and $\sigma_2$ are maps which reorder the spaces as necessary. In particular, $\sigma_1:\bigotimes_{i=1}^l H_{\beta_i}\otimes B^{\otimes l'}\longrightarrow B^{\otimes l'}\otimes \bigotimes_{i=1}^l H_{\beta_i}$ and $\sigma_2:\bigotimes_{i=1}^k H_{\alpha_i}\otimes B^{\otimes k'}\longrightarrow B^{\otimes k'}\otimes \bigotimes_{i=1}^k H_{\alpha_i}$. For the second relation we observe that
$
T_{p,S}^*  =(s_{p,L}^{-1}\circ (T_p\otimes S)\circ s_{p,U})^*=
s_{p,U}^{-1}\circ(T_p^*\otimes S^*)\circ s_{p,L}=
s_{p^*,L}^{-1}\circ(T_{p^*}\otimes S^*)\circ s_{p^*,U}=
T_{p^*,S^*}.
$
The compatibility with the multiplication (case 3a) follows from
\begin{eqnarray*}
T_{q,S'} T_{p,S} &= &(s_{q,L}^{-1}\circ (T_q\otimes S')\circ s_{q,U})\circ (s_{p,L}^{-1}\circ (T_p\otimes S)\circ s_{p,U})=
s_{q,L}^{-1}\circ (T_q\otimes S')\circ (T_p\otimes S)\circ s_{p,U}\\&=& 
\delta^{cy(p,q)} (s_{qp,L}^{-1}\circ (T_{qp}\otimes S'S)\circ s_{qp,U})=
\delta^{cy(p,q)}T_{qp,S'S}
\end{eqnarray*}
The proof of the case 3b is analogous.
\end{proof}

The following lemma is a linearity result concerning these morphisms.

\begin{lemma}\label{lintp}
Let $p\in NC_{\G}(\alpha_1,...,\alpha_k;\beta_1,...,\beta_l)$ endowed with the morphisms $S,S'\in\bigotimes_{v=1}^m \Hom(\alpha_{U_v},\beta_{L_v})$. For all $\lambda,\mu\in\C$ we have $\lambda T_{p,S}+\mu T_{p,S'}=T_{p,\lambda S+\mu S'}$.
\end{lemma}

\begin{proof}
By applying the definition and by using the linearity of the different maps we have:
\begin{eqnarray*}
\lambda T_{p,S}+\mu T_{p,S'}&=&\lambda(s_{p,L}^{-1}\circ (T_p\otimes S)\circ s_{p,U})+\mu (s_{p,L}^{-1}\circ (T_p\otimes S')\circ s_{p,U})\\ &=&
s_{p,L}^{-1}\circ (T_p\otimes \lambda S)\circ s_{p,U}+s_{p,L}^{-1}\circ (T_p\otimes \mu S')\circ s_{p,U}\\ & =&
s_{p,L}^{-1}\circ ((T_p\otimes \lambda S)\circ s_{p,U}+(T_p\otimes \mu S')\circ s_{p,U})\\ & =&
s_{p,L}^{-1}\circ ((T_p\otimes \lambda S+T_p\otimes \mu S')\circ s_{p,U})\\ & =&
s_{p,L}^{-1}\circ ((T_p\otimes \lambda S+\mu S')\circ s_{p,U}) =T_{p,\lambda S+\mu S'}.
\end{eqnarray*}
\end{proof}

\begin{remark}
The category $\mathscr{NC}_{\G}$ with objects the free monoid $M$ over $\irr(\G)$ and morphisms
$$\Hom((\alpha_1,\dots,\alpha_k);(\beta_1,\dots,\beta_l):=span\{T_{p,S} | p\in NC_{\G}(\alpha_1,...,\alpha_k;\beta_1,...,\beta_l), S\in \bigotimes_{v=1}^m \Hom(\alpha_{U_v},\beta_{L_v})\}$$
is a rigid C*-tensor category. The result follows easily from Proposition \ref{propmapdec} and Lemma \ref{lintp}. Indeed, the tensor product is given by concatenation in the monoid: $(\alpha_1,\dots,\alpha_k)\ot(\beta_1,\dots,\beta_l):=(\alpha_1,\dots,\alpha_k,\beta_1,\dots,\beta_l)$. The object $1$ is given by the empty word $\emptyset\in M$ and the conjugate of $(\alpha_1,\dots,\alpha_k)$ is given by $(\overline{\alpha}_k,\dots,\overline{\alpha}_1)$. Actually, the pair $(R,\overline{R})$ satisfying the conjugate equations for the pair $((\alpha_1,\dots,\alpha_k),(\overline{\alpha}_k,\dots,\overline{\alpha}_1))$ is defined as follows. Let $(S,\overline{S})$ be a pair satisfying the conjugate equation for the pair $u=\alpha_1\ot\ldots\ot\alpha_k$ and $\overline{u}=\overline{\alpha}_k\ot\ldots\ot\overline{\alpha}_1$ in the rigid C*-tensor category $\mathcal{R}(\G)$. Hence, the partition $p\in NC(0,2k)$ with no upper points and all the lower points connected together is well decorated in either $NC_\G(\emptyset,\alpha_1\ot\ldots\alpha_k\ot\overline{\alpha}_k\ot\ldots\ot\overline{\alpha}_1)$ or  $NC_\G(\emptyset,\overline{\alpha}_k\ot\ldots\ot\overline{\alpha}_1\ot \alpha_1\ot\ldots\alpha_k)$ and we may define $R:=T_{p,S}\in\Hom(1,(\alpha_1,\dots,\alpha_k)\ot(\overline{\alpha}_k,\dots,\overline{\alpha}_1))$ and $\overline{R}:=T_{p,\overline{S}}\in\Hom(1,(\overline{\alpha}_k,\dots,\overline{\alpha}_1)\ot(\alpha_1,\dots,\alpha_k))$.
\end{remark}

\begin{theorem}\label{intertwrpr}
Let $B$ be a $n$-dimensional C*-algebra ($n\geq 4$) endowed with a $\delta$-form $\psi$ and $\G$ a compact quantum group. Consider the free wreath product $\wprhg$ with basic representations $a(\alpha)$, where $\alpha\in\irr(\G)$. Then, for all $k,l\in\N$
$$\begin{array}{l}
\Hom(\bigotimes_{i=1}^k a(\alpha_i),\bigotimes_{j=1}^l a(\beta_j))=\\ \hspace{4.7cm}
  span\{T_{p,S} | p\in NC_{\G}(\alpha_1,...,\alpha_k;\beta_1,...,\beta_l), S\in \bigotimes_{v=1}^m \Hom(\alpha_{U_v},\beta_{L_v})\}
\end{array}$$
with the convention that, if $k=0$, $\bigotimes_{i=1}^k a(\alpha_i)=1_{\wprhg}$ and the space of the noncrossing partitions is $NC_{\G}(\emptyset;\bigotimes_{j=1}^l a(\beta_j))$, i.e. it does not have upper points. Similarly, if $l=0$.\\
Moreover, its dimension is given by $\sum_{p\in NC_{\G}(\alpha_1,...,\alpha_k;\beta_1,...,\beta_l)}\prod_{v=1}^m \dim \Hom(\alpha_{U_v},\beta_{L_v})$.
\end{theorem}

\begin{proof}
In order to prove the inclusion $\supseteq$, we have to show that every linear map $T_{p,S}$ obtained from a decorated noncrossing partition $p$ endowed with a suitable morphism $S$ is an intertwiner of $\wprhg$. In particular, we will prove that every $T_{p,S}$ can be decomposed as a linear combination of tensor products, compositions and adjoints of the basic morphisms $\widetilde{m\ot S}$, $\eta$ and $\id$. From Theorem \ref{intertautb}, we know that such a decomposition at the level of the noncrossing partitions exists. The more difficult point here is to decorate every block of the decomposition with irreducible representations and to associate the right morphisms such that, if we compose all the diagrams, we obtain the original map. By Frobenius reciprocity (Theorem \ref{teofrobrec}) we have

$$\Hom(\bigotimes_{i=1}^k a(\alpha_i),\bigotimes_{j=1}^l a(\beta_j))\cong \Hom(1,a(\beta_1)\ot...\ot a(\beta_l)\ot a(\bar{\alpha}_k)\ot ...\ot a(\bar{\alpha}_1)).$$

Moreover, by applying the Frobenius reciprocity to the rigid C*-tensor category $\mathscr{NC}_{\G}$ we get
$$\Hom((\alpha_1,...,\alpha_k),(\beta_1,...,\beta_l))\cong \Hom(\emptyset,(\beta_1,...,\beta_l,\bar{\alpha}_k,...,\bar{\alpha}_1)).$$
It follows that it is enough to prove the inclusion for $k=0$. Moreover, we can restrict ourselves to prove the result in the case of a one-block non crossing partition, because the map associated to any decorated noncrossing partition in $NC_{\G}(\emptyset,(\beta_1,...,\beta_l))$ can easily be obtained through compositions and tensor products of the maps associated to one-block noncrossing partitions and of the identity map. Let $p\in NC_{\G}(\emptyset;\beta_1,...,\beta_l)$, $b(p)=1$ be a decorated noncrossing partition endowed with the morphism $S$ and consider the map $T_{p,S}$. 
The condition $b(p)=1$ implies that $S\in\Hom(1_{\G},\bigotimes_{j=1}^l \beta_{l})$ and the diagram we have to consider is as follows.

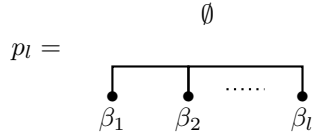
{\centering
 \begin{tikzpicture}[thick,font=\small]
    \path (2.25,1.5) node{$\emptyset$} 
          (0,1) node{$p_l=$} 
          (1,.4) node{$\bullet$} node[below](d){$\beta_1$}
          (2,.4) node{$\bullet$} node[below](e){$\beta_2$}
          (3.5,.4) node{$\bullet$} node[below](f){$\beta_l$};

    \draw (d) -- +(0,+0.7) -| (e);
    \draw (e) -- +(0,+0.7) -| (f);

     \draw [dotted] (2.5,.5) -- (3,.5);
     
\end{tikzpicture}
\par}
\noindent
We will prove the result by induction on $l$. If $l=0$ the result is trivial, if $l=1$ we obtain the map $\eta$ which is in $\Hom(1,a(1_{\G}))$ by definition. If $l=2$, $S\in\Hom(1_{\G},\beta_1\ot \beta_2)$ and the diagram can be decomposed as follows.

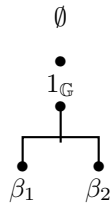
{\centering
 \begin{tikzpicture}[thick,font=\small]
    \path (1.5,2.4) node{$\emptyset$} 
             
          (1.5,1.8) node{$\bullet$} 
          (1.5,1.2) node{$\bullet$} node[above](c){$1_{\G}$} 
          (1,.4) node{$\bullet$} node[below](d){$\beta_1$}
          (2,.4) node{$\bullet$} node[below](e){$\beta_2$};

    \draw (d) -- +(0,+0.7) -| (e);
    \draw (c) -- +(0,-0.75);

\end{tikzpicture}
\par}
\noindent
Therefore $T_{p_2,S}=\Sigma_{23}(m^*\ot S)\eta\in\Hom(1,a(\beta_1)\ot a(\beta_2))$ because it can be seen as the composition of two intertwiners. Moreover, we observe that this situation is possible only if $\beta_2=\bar{\beta}_1$. If $l=3$ and the morphism associated to the noncrossing partition is $S\in\Hom(1_{\G},\beta_1\ot \beta_2)$, we have the following decomposition.

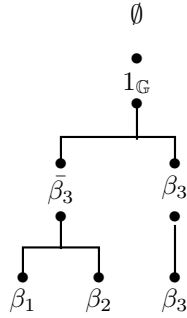
{\centering
 \begin{tikzpicture}[thick,font=\small]
    \path (1.5,2.4) node{$\emptyset$} 
             
          (1.5,1.8) node{$\bullet$} 
          (1.5,1.2) node{$\bullet$} node[above](c){$1_{\G}$} 
          (.5,.4) node{$\bullet$} node[below](d){$\bar{\beta_3}$}
          (2,.4) node{$\bullet$} node[below](e){$\beta_3$}
          (.5,-.3) node{$\bullet$} 
          (2,-.3) node{$\bullet$}
          (0,-1.1) node{$\bullet$} node[below](h){$\beta_1$}
          (1,-1.1) node{$\bullet$} node[below](i){$\beta_2$}
          (2,-1.1) node{$\bullet$} node[below](l){$\beta_3$};

    \draw (d) -- +(0,+0.7) -| (e);
    \draw (c) -- +(0,-0.75);
    \draw (h) -- +(0,+0.7) -| (i);
    \draw (2,-.4) -- (l);
    \draw (.5,-.3) -- +(0,-0.42);

\end{tikzpicture}
\par}
\noindent
In order to complete the description of the decomposition, we need to associate a morphism to every noncrossing partition. The morphism of the noncrossing partition corresponding to $\eta$ is clearly $\id_{1_{\G}}$, while the morphism of the lower block corresponding to the identity is $\id_{H_{\beta_3}}$. In order to define the remaining morphisms we recall the notation introduced in Definition \ref{defrigid} to denote the invariant vectors. Let $R\in\Hom(1,\bar{\beta_3}\ot \beta_3)$ and $\bar{R}\in\Hom(1,\beta_3\ot \bar{\beta_3})$ be the morphisms satisfying the conjugate equations. Then, the morphisms associated to the two blocks corresponding to $m^*$ are $R\in\Hom(1,\bar{\beta_3}\ot \beta_3)$ and $(\id_{H_{\beta_1}\ot H_{\beta_2}}\ot \bar{R}^*)(S\ot \id_{H_{\bar{\beta_3}}})\in\Hom(\bar{\beta_3},\beta_1 \ot \beta_2)$ respectively. An easy computation allows us to verify that 
$S=[(\id_{H_{\beta_1}\ot H_{\beta_2}}\ot \bar{R}^*)(S\ot \id_{H_{\bar{\beta_3}}})\ot \id_{H_{\beta_3}}]R$. This means that $T_{p_3,S}$ can be decomposed in term of some of the basic morphisms introduced in the definition of the free wreath product and therefore is in $\Hom(1,\beta_1\ot \beta_2\ot \beta_3)$.\\
Now, we are ready for the inductive step. Let us suppose the inclusion true for a $l=t\geq 3$ and prove it for $t+1$. As usual, let $S\in\Hom(1,\bigotimes_{j=1}^{t+1}\beta_j)$ be the morphism associated to the noncrossing partition $p_{t+1}$. The decomposition which we need to consider in this case is (with $\alpha_i\subset \bar{\beta}_{t+1}\ot \bar{\beta}_{t}$):

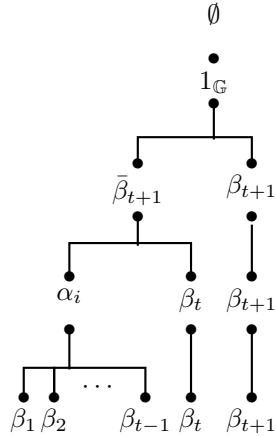
{\centering
 \begin{tikzpicture}[thick,font=\small]
    \path (1.5,2.4) node{$\emptyset$} 
             
          (1.5,1.8) node{$\bullet$} 
          (1.5,1.2) node{$\bullet$} node[above](c){$1_{\G}$} 
          (.5,.4) node{$\bullet$} node[below](d){$\bar{\beta}_{t+1}$}
          (2,.4) node{$\bullet$} node[below](e){$\beta_{t+1}$}
          (.5,-.3) node{$\bullet$} 
          (2,-.3) node{$\bullet$}
          (-.4,-1.1) node{$\bullet$} node[below](h){$\alpha_i$}
          (1.2,-1.1) node{$\bullet$} node[below](i){$\beta_t$}
          (2,-1.1) node{$\bullet$} node[below](l){$\beta_{t+1}$}
          (-.4,-1.8) node{$\bullet$}
          (1.2,-1.8) node{$\bullet$}
          (2,-1.8) node{$\bullet$}
          (-1,-2.7) node{$\bullet$} node[below](m){$\beta_1$}
          (-.6,-2.7) node{$\bullet$} node[below](n){$\beta_2$}
          (.6,-2.7) node{$\bullet$} node[below](o){$\beta_{t-1}$}
          (1.2,-2.7) node{$\bullet$} node[below](p){$\beta_t$}
          (2,-2.7) node{$\bullet$} node[below](q){$\beta_{t+1}$}
          (0,-2.55) node{$\cdots$};
          
    \draw (d) -- +(0,+0.7) -| (e);
    \draw (c) -- +(0,-0.75);
    \draw (h) -- +(0,+0.7) -| (i);
    \draw (2,-.4) -- (l);
    \draw (.5,-.3) -- +(0,-0.38);
    \draw (m) -- +(0,+0.7) -| (n);
    \draw (n) -- +(0,+0.7) -| (o);
    \draw (1.2,-1.8) -- (p);
    \draw (2,-1.8) -- (q);
    \draw (-.4,-1.8) -- +(0,-0.52);

\end{tikzpicture}
\par}
\noindent
Now, we have to assign a suitable morphism to every noncrossing partition of the decomposition. As in the case $l=3$, we associate the identity map to the diagrams corresponding to $\eta$ and to $\id$. The morphisms on the other blocks are less obvious and we need to introduce some notations.
Let us denote $R_t\in\Hom(1,\bar{\beta_t}\ot\beta_t)$, $\bar{R}_t\in\Hom(1,\beta_t\ot \bar{\beta_t})$, $R_{t+1}\in\Hom(1,\bar{\beta}_{t+1}\ot\beta_{t+1})$ and $\bar{R}_{t+1}\in\Hom(1,\beta_{t+1}\ot \bar{\beta}_{t+1})$ two pair of invariant vectors satisfying the conjugate equations. For every $\alpha_i\subset \bar{\beta}_{t+1}\ot \bar{\beta}_{t}$, we know that there is an isometry $r_i\in\Hom(\alpha_i,\bar{\beta}_{t+1}\ot \bar{\beta}_{t})$ such that $r_ir_i^*\in\End(\bar{\beta}_{t+1}\ot \bar{\beta}_{t})$ is a projection and $\sum_i r_ir_i^*=\id_{H_{\bar{\beta}_{t+1}}\ot H_{\bar{\beta}_{t}}}$.
There are still three morphisms to assign; they will be denoted $S_{1,i}$, $S_{2,i}$ and $S_{3,i}$. From the top to the bottom, they are the following ones. The morphism $S_{1,i}\in\Hom(1_{\G},\bar{\beta}_{t+1}\ot\beta_{t+1})$ is $R_{t+1}$. The morphism $S_{2,i}\in\Hom(\bar{\beta}_{t+1},\alpha_i\ot \beta_t)$ is $(r_i^*\ot \id_{H_{\beta_t}})(\id_{H_{\bar{\beta}_{t+1}}}\ot R_t)$. Finally, the morphism $S_{3,i}\in\Hom(\alpha_i,\bigotimes_{j=1}^{t-i}\beta_j)$ is $$(\id_{H_{\beta_1}\ot...\ot H_{\beta_{t-1}}}\ot \bar{R}_t^*(\id_{H_{\beta_t}}\ot \bar{R}_{t+1}^*\ot \id_{H_{\bar{\beta}_t}}))(S\ot\id_{H_{\bar{\beta}_{t+1}}\ot H_{\bar{\beta}_t}})r_i.$$
These are all morphisms because they are obtained through the operations of tensor product, composition and adjoint from known morphisms. Moreover, by making use of the Frobenius reciprocity and of the inductive hypothesis, we know that the linear map in $\mathcal{L}(B\ot H_{\alpha_i},\bigotimes_{j=1}^{t-1}(B\ot H_{\beta_j}))$ associated to the noncrossing partition endowed with the morphism $S_{3,i}$ is in $\Hom(a(\alpha_i),\bigotimes_{j=1}^{t-1}a(\beta_j))$. The fact that the linear maps corresponding to the other decorated noncrossing partitions are intertwiners follows from the definition of free wreath product. An easy computation allows us to verify that $\sum_i(S_{3,i}\ot \id_{H_{\beta_t}\ot H_{\beta_{t+1}}})(S_{2,i}\ot \id_{H_{\beta_{t+1}}})S_{1,i}=S$, therefore, by making use of Lemma \ref{lintp} and of Proposition \ref{propmapdec}, we have that $T_{p_{t+1},S}\in\Hom(1,\bigotimes_{j=1}^{t+1}a(\beta_j))$ as it is possible to write $T_{p_{t+1},S}$ as a linear combination of compositions, adjoints and tensor products of intertwiners.

For the second inclusion ($\subseteq$), we apply the Tannaka-Krein duality to the concrete rigid monoidal C*-category $\mathscr{NC}_{\G}$. Then, there exists a compact quantum group $\G=(C(\G),\Delta)$ such that $C(\G)$ is generated by the coefficients of a family of finite dimensional unitary representations $a(\alpha_i)'$ and $\Hom(\bigotimes_{i=1}^k a(\alpha_i)',\bigotimes_{j=1}^l a(\beta_j)')=span\{T_p | p\in NC_{\widehat{\Gamma}}(\alpha_1,...,\alpha_k;\beta_1,...,\beta_l)\}$. Moreover, because of the universality of the Tannaka-Krein construction, from the inclusion already proved we can deduce that there is a surjective map $\phi:C(\G)\longrightarrow \wprhg$ such that $(\id \ot \phi)(a(\alpha)')=a(\alpha)$, for all $\alpha\in\irr(\G)$. In order to complete the proof, we have to show that this map is an isomorphism. The existence of the inverse morphism follows from the universality of the free wreath product construction. It is enough to observe that the unitary representations $a(\alpha)'$ are such that $\widetilde{m\ot S}\in\Hom(a(\alpha)'\ot a(\beta)',a(\gamma)')$ for any $S\in\Hom(\alpha\ot \beta,\gamma)$ and $\eta\in\Hom(1,a(1_{\G})')$ because the maps $\widetilde{m\ot S}$ and $\eta$ correspond to well decorated noncrossing partitions.

Since the maps $T_p$ associated to distinct noncrossing partitions in $NC(k,l)$ are linearly independent (see Theorem \ref{intertautb}), the dimension formula follows.
\end{proof}

\begin{remark}
The description of the space of intertwiners does not depend on considering a unital $\delta$-form $\psi$ or the associated non-unital 1-form $\tilde{\psi}$. In order to prove the monoidal equivalence result in the next Section it will be necessary to consider the non-unital $1$-form $\tilde{\psi}$. 
\end{remark}

\begin{remark}\label{cara1}
As in the case of the dual of a discrete group (see \cite{pit14}), we can compute the Haar measure of some particular elements. Consider the free wreath product $\wprhg$, where $\psi$ is a $\delta$-form, and let $\chi(a(1_{\G})):=(\Tr\otimes \id)(a(1_{\G}))$ be the character of the representation $a(1_{\G})$. We easily observe that $\chi(a(1_{\G}))$ is self-adjoint. Now, we want to compute the moments $h(\chi(a(1_{\G}))^k)$; let $p_k$ be the orthogonal projection onto the fixed points space $Hom(1,a(1_{\G})^{\otimes k})$. Thanks to some classic results of Woronowicz (see \cite{wor88}) we have 
$h(\chi(a(1_{\G}))^k)=h((\Tr\otimes \id)(a(1_{\G}))^k)=\Tr((\id\otimes h)(a(1_{\G})^{\otimes k}))=\Tr(p_k)=\dim(Hom(1,a(1_{\G})^{\otimes k}))=\text{\#}NC(0,k)=C_k$ where $C_k$ are the Catalan numbers. They are the moments of the free Poisson law of parameter 1 which is then the spectral measure of $\chi(a(1_{\G}))$.
\end{remark}

%%%%%%%%%%%%%%%%%%%%%%%%%%%%%%%%%%%%%%
\section{Monoidal equivalence}
%%%%%%%%%%%%%%%%%%%%%%%%%%%%%%%%%%%%%%

In this section, we obtain a useful monoidal equivalence result for the free wreath product $\wprhg$. An analogous result has been proved in \cite[Theorem 5.11]{lt14}, in the particular case of the free wreath product of a compact matrix quantum group of Kac type by the quantum symmetric group. We will show that it can be extended to our more general context. The monoidal equivalence will allow us to reconstruct the representation theory of $\wprhg$ and to prove some properties of the operator algebras associated to the free wreath product.

By doing some minor changes and remarks, the proof of \cite[Theorem 5.11]{lt14}  also works in our more general situation. Hence, we will only sketch the proof of the following theorem, pointing out the critical passages in the adaptation to our context. Recall that the symbol $*$ denotes the maximal free product of unital C*-algebras and $\widehat{*}$ denotes the free product of compact quantum groups.

\begin{theorem}\label{teomoneq}
Let $(B,\tpsi)$ be a finite dimensional C*-algebra, $dim(B)\geq 4$, endowed with a possibly non-unital 1-form $\tpsi$. Let $0<q\leq 1$ be such that $q+q^{-1}=\sqrt{\tpsi(1)}$. Let $\G$ be a compact quantum group and consider the free wreath product $\G\wr_*\G^{aut}(B,\tpsi)$. 
Let $\HH$ be the compact quantum subgroup of $\G\hat{\ast}SU_{q}(2)$ given by
$$C(\HH)=<b_{ij}ab_{kl}\ |\ 1\le i,j,k,l\le 2, a\in C(\G)>\subset C(\G)\ast C(SU_{q}(2))$$
$$\Delta_{\HH}:=\Delta_{\G\hat{\ast}SU_{q}(2)|_{\HH}}$$
More precisely
\begin{align*}
\Delta(b_{ij}ab_{kl})=\sum_{r,s,v}b_{ir}a_{(1)}b_{ks}\otimes b_{rj}a_{(2)}b_{sl}\in C(\HH)\otimes C(\HH)
\end{align*}
where $b=(b_{ij})_{ij}$ is the generating matrix of $SU_{q}(2)$ and $\Delta_{\G}(a)=\sum a_{(1)}\otimes a_{(2)}$.\\
Then $$\G\wr_*\G^{aut}(B,\tpsi)\simeq_{mon} \HH$$
\end{theorem}

\begin{proof}
The first remark is about the existence of a $q\in(0,1]$ such that $q+q^{-1}=\sqrt{\tpsi(1)}$. An easy computation shows that $q+q^{-1}:(0,1]\longrightarrow [2,\infty)$ is a bijection, therefore the monoidal equivalence makes sense only if $\tpsi(1)\geq 4$. We recall that $\tpsi:B\longrightarrow \C$ is a 1-form so 
it can be rewritten as $\tpsi(\cdot)=\sum_{\lambda}\Tr(Q_{\lambda}\cdot)$ for a suitable family of positive diagonal matrices $Q_{\lambda}$ such that $Tr(Q_{\lambda}^{-1})=1$ for every $\lambda$. With this in mind, the condition $\tpsi(1)\geq 4$ is a consequence of the following arithmetic lemma, which follows from the inequality between the harmonic mean and the arithmetic mean.
\begin{lemma}
Let $(x_i)_{i=1,...,n}, x_i> 0$ and $n\geq 2$ be a family of positive real numbers such that $\sum_{i=1}^n x_i\leq 1$. Then $\sum_{i=1}^n x_i^{-1}\geq 4$.
\end{lemma}

The proof of the monoidal equivalence is based on the construction of an explicit isomorphism between the intertwining spaces. In this first phase, we will take into account intertwiners between the tensor products of the representations which generate the compact quantum groups; only in a second time, this isomorphism will be extended to intertwiners between tensor products of all the irreducible representations. Consider the family of representations of $\HH$ given by $s(\alpha):=b\otimes \alpha \otimes b, \, \alpha\in\irr(\G)$. By using the description of the irreducible representations of a free product given by Wang in \cite{wan95}, we have that $\{s(\alpha)|\alpha\in\irr(\G), \alpha\neq 1_{\G}\}\subset \irr(\HH)$. Moreover, we observe that, for any $\beta\in\irr(\HH)$, there exists a finite family $(\alpha_i)_{i=1,...,k}$, $\alpha\in\irr(\G)$ such that $\beta\subset \bigotimes_{i=1}^k s(\alpha_i)$ because the coefficients of the representations $s(\alpha)$ are dense in $C(\HH)$.
Let us denote $\phi$ the map such that $\phi(s(\alpha)):=a(\alpha)$, $\phi(1_{\HH})=1_{\wprhg}$. Then, it is possible to define an isomorphism $$\phi:\Hom(\bigotimes_{i=1}^k s(\alpha_i),\bigotimes_{j=1}^l s(\beta_j))\longrightarrow\Hom(\bigotimes_{i=1}^k a(\alpha_i),\bigotimes_{j=1}^l a(\beta_j))$$ which satisfies the properties of a monoidal equivalence.
The core idea in order to define this map is to find a good description of the two spaces of intertwiners. The spaces on the right have been described in terms of decorated noncrossing partitions in Theorem \ref{intertwrpr}.\\
The intertwiners of $\HH$ can be described by means of semi-decorated noncrossing partitions in $NC(3k,3l)$ such that, when numbering each line of points from the left to the right, the points with a number equal to 0 or 2 modulo 3 form a Temperley-Lieb diagram in $TL(2k,2l)$ and the remaining points form a decorated noncrossing partition in $NC((\alpha_1,...,\alpha_k),(\beta_1,...,\beta_l))$ endowed with a morphism $S$. This presentation can be proved by recalling that the intertwiners of $SU_q(2)$ can be described in terms of Temperley-Lieb diagrams (such that the coefficient $q+q^{-1}$ is introduced for every central block removed during the composition operation) and by knowing the description of the intertwiners of a free product of two compact quantum groups in terms of the intertwiners of the factors (see Proposition 2.15 in \cite{lem14} and observe that the result is true for every compact quantum groups, even if it is stated only for compact matrix quantum groups).

Now, in order to describe the map $\phi$, we recall that there is an isomorphism $$\rho:TL_x(2k,2l)\longrightarrow NC_{x^2}(k,l),\, x\in\R^+$$ which satisfies all the compatibility properties of Definition \ref{defmoneq} (but at the level of the diagrams). The subscripts $x$ and $x^2$ mean that when composing two diagrams, the final diagram is multiplied by a coefficient $x$ or $x^2$ for every central block appeared. More precisely, $\rho$ acts by associating to each Temperley-Lieb diagram the noncrossing partition obtained by identifying the pairs of consecutive points and by multiplying this partition by a suitable coefficient (which depends on the diagram). This coefficient is crucial to assure the compatibility with the multiplication given by $\rho(t_2t_1)=\rho(t_2)\circ\rho(t_1)$ for $t_1,t_2$ composable Temperley-Lieb diagrams.

The map $\phi$ is then defined by sending every special diagram described above to the noncrossing partition obtained after applying the map $\rho$ to the Temperley-Lieb diagram and decorating the points with the $\alpha_i,\beta_j$. Finally, this noncrossing partition is endowed with the map $S$ (actually, some twist operations can be necessary, but for simplicity we keep the same notation). In this case, the subscript $x$, introduced in the definition of $\rho$, has to be chosen equal to $q+q^{-1}$, the coefficient corresponding to central blocks for $SU_q(2)$.

Then, every central block appeared when composing the associated noncrossing partitions, will correspond to $(q+q^{-1})^2$ and this factor is by hypothesis equal to $\tpsi(1)$, the coefficient corresponding to central blocks for $\wprthg$. Thanks to this choice, it is possible to verify that $\phi$ is a well defined isomorphism and satisfies all the properties of Definition \ref{defmoneq}.

We observe that, in order to use the isomorphism $\rho$, it has been crucial the dependence on the number of central blocks (instead of on the number of cycles) of the coefficient possibly appeared when composing two noncrossing partitions. This explains the use of $\tpsi$ instead of $\psi$.

To complete the proof, it is enough to observe that the map $\phi$ is an equivalence between the categories containing the tensor products of the generating representations of the two compact quantum groups and that there is a correspondence between these generators.
By applying Proposition \ref{propgenmoneq}, we can extend $\phi$ to an equivalence $\tilde{\phi}$ between the completions of the two categories with respect to direct sums and sub-objects and the monoidal equivalence is proved. %The map $\tilde{\phi}$ is in particular a bijection between the irreducible representations and the monoidal equivalence is proved.
\end{proof}

%%%%%%%%%%%%%%%%%%%%%%%%%%%%%%%%%%%%%%%%%	
\section{Irreducible representations and fusion rules}
%%%%%%%%%%%%%%%%%%%%%%%%%%%%%%%%%%%%%%%%%

In this section, we find the irreducible representations and describe the fusion rules of the free wreath product $\wprg$. The next result follows from Theorem \ref{intertwrpr} and is a generalisation of \cite[Cor 3.9]{lt14}.

\begin{prop}
Let $B$ be a finite dimensional C*-algebra, $dim(B)\geq 4$, endowed with a $\delta$-form $\psi$. Let $\G$ be a compact quantum group. The basic representations $a(\alpha)$, $\alpha\in\irr(\G)$ of $\wprhg$ are irreducible and pairwise non-equivalent if $\alpha\not\simeq 1_{\G}$. The representation $a(1_{\G})$ can be decomposed  as $1_{\wprhg}\oplus r_{1_{\G}}$, where $r_{1_{\G}}$ is irreducible and non-equivalent to any $a(\alpha), \alpha\not\simeq 1_{\G}$.
\end{prop}

\begin{proof}
Let $\alpha,\beta\in\irr(\G)$. There are only two possible decorated diagrams $p\in NC_\G(\alpha,\beta)$:

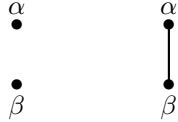
{\centering
 \begin{tikzpicture}[thick,font=\small]
    \path (3,.8)node{$\bullet$} node[above](c) {$\alpha$}
          (3,0) node{$\bullet$} node[below](d) {$\beta$}
          (5,.8) node{$\bullet$} node[above](e) {$\alpha$}
          (5,0) node{$\bullet$} node[below](f) {$\beta$};
               
    \draw (e) -- (f);

 \end{tikzpicture}
\par}
\noindent

First assume that both $\alpha,\beta\neq 1_\G$. Then, the first diagram is not well decorated. If $\alpha=\beta$, since $\Hom(\alpha,\alpha)=\C\id$, it follows from Theorem \ref{intertwrpr} that $\dim \Hom(a(\alpha),a(\alpha))=1$ and the irreducibility of $a(\alpha)$ is proved. If $\alpha\neq\beta$, since $\dim \Hom(\alpha,\beta)=0$, the second diagram is not well decorated as well. Hence, $\dim \Hom(a(\alpha),a(\beta))=0$.

If $\alpha=\beta=1_G$ then both the first and the second diagrams are well decorated. Moreover, since $\Hom(1_G,1_G)=\C$, it follows that $\dim \Hom(a(1_{\G}),a(1_{\G}))=2$, so the linear independence of the intertwiners associated to distinct noncrossing partitions together with the remark that $\dim \Hom(1_{\wprhg},a(1_{\G}))=1$ (the multiples of $\eta$) give the decomposition $a(1_{\G})\simeq1_{\wprhg}\oplus r_{1_{\G}}$.

Finally, If $\beta\neq 1_{\G}$ then the first diagram is not well decorated and if $\alpha=1_\G$, the second diagram is not well decorated as well since $\Hom(a(1_\G),a(\beta))=0$ for all $\beta\neq 1_\G$, which implies that $r_{1_\G}\not\simeq a(\beta)$ for all $\beta\neq 1_\G$. 
\end{proof}

Let us now describe the fusion semiring of $\wprhg$.

\begin{deff}\label{opermonm}
Let $M$ be the monoid whose elements are the words written by using the irreducible representations of $\G$ as letters. We define the following operations:
\begin{itemize}
\item[-] involution: $\overline{(\alpha_1,...,\alpha_k)}=(\overline{\alpha}_k,...,\overline{\alpha}_1)$
\item[-] concatenation: $(\alpha_1,...,\alpha_k),(\beta_1,...,\beta_l)=(\alpha_1,...,\alpha_k,\beta_1,...,\beta_l)$
\item[-] fusion of two non-empty words: $(\alpha_1,...,\alpha_k).(\beta_1,...,\beta_l)$
is the multiset composed by the words $(\alpha_1,...,\alpha_{k-1},\gamma,\beta_2,...,\beta_l)$ for all the possible $\gamma\subset \alpha_k\otimes \beta_1$; the multiplicity of each word is given by $\dim(\Hom(\gamma,\alpha_k\ot\beta_1))$, i.e. by the multiplicity of the representation $\gamma$ in the tensor product $\alpha_k\otimes \beta_1$.
\end{itemize}
\end{deff}

\begin{theorem}\label{ThmFusion}
Let $B$ be a finite dimensional C*-algebra, $dim(B)\geq 4$, endowed with a $\delta$-form $\psi$. Let $\G$ be a compact quantum group. The classes of irreducible non-equivalent representations of $\wprhg$ can be indexed by the elements of the monoid $M$ and denoted $r_x, \, x\in M$. The involution is given by $\overline{r}_x=r_{\overline{x}}$ and the fusion rules are:
$$r_x\otimes r_y=\sum_{\substack{x=u,t \\ y=\bar{t},v}}r_{u,v}\oplus \sum_{\substack{x=u,t\;  y=\bar{t},v \\  u\neq\emptyset , v\neq \emptyset\\ w\in u.v}}r_{w}$$
\end{theorem}

Because of the monoidal equivalence proved in Theorem \ref{teomoneq}, it is enough to verify that these are the fusion rules of $\HH$ and it has been done in \cite{lt14}. Let us note that it is possible (and more elegant) to give a direct proof of the preceding theorem, without relying on the monoidal equivalence.

\begin{remark}\label{remfondrep}
If $\G$ is a matrix quantum group then $\wprhg$ is again a matrix quantum group. More precisely, if $\alpha_1,\ldots,\alpha_n\in\irr(\G)\setminus\{1_G\}$ generate the representation category of $\G$ then $a(\alpha_1),\dots,a(\alpha_n)$ generate the representation category of $\wprhg$. Indeed, since the algebra generated by the coefficients of the $a(\alpha)$, for $\alpha\in\irr(\G)$, is dense in $C(\wprhg)$, it suffices to show that, for all $\alpha\in\irr(\G)\setminus\{1_\G\}$, $a(\alpha)$ is a subrepresentation of a tensor product of representations of the form $a(\alpha_i), a(\overline{\alpha}_i)$. Let $\alpha\in\irr(\G)$, $\alpha\neq 1_\G$, and write $\alpha\subset \alpha_{i_1}^{\varepsilon_1}\ot\dots\ot\alpha_{i_l}^{\varepsilon_l}$, where $i_1,\dots,i_l\in\{1,\ldots,n\}$ and $\varepsilon_k\in\{-1,1\}$ and $\alpha^\varepsilon:=\overline{\alpha}$ if $\varepsilon=-1$ or $\alpha^\varepsilon:=\alpha$ if $\varepsilon=1$. By making use of the fusion rules just found, we deduce that $a(\alpha)\subset a(\alpha_{i_1}^{\varepsilon_1})\ot\dots\ot a(\alpha_{i_l}^{\varepsilon_l})$.\end{remark}

When the state $\psi$ is not a $\delta$-form, it is still possible to find the irreducible representations and the fusion rules. Indeed, we show in the next proposition how to reduce the general case to the case of a $\delta$-form by using a free product decomposition. Then, one may use the description of the irreducible representations and the fusion rules of a free product (\cite{wan95}) together with Theorem \ref{ThmFusion} to obtain the irreducible representation and the fusion rules for any given faithful state $\psi$. This free product decomposition will also be useful to extend algebraic and analytical properties from the $\delta$-form case to the general case.

\begin{prop}\label{freeprodcqggen}
Let $B=\bigoplus_{T=1}^c M_{n_T}(\C)$ be a finite dimensional C*-algebra and $\psi=\bigoplus_{T=1}^c \Tr(Q_T\cdot)$ a state on $B$.
Consider the decomposition $B=\bigoplus_{i=1}^{d}B_{i}$ obtained by summing up all the matrix spaces $M_{n_T}(\C)$ with a common value of $\Tr(Q_T^{-1})$ in a unique summand $B_i$ and let $\delta_i$ be the common value of $\Tr(Q_T^{-1})$ on the summand $B_i$. Let $\psi_i\,:\,B_i\rightarrow \C$ be the $\delta_i$-form defined by $\psi_i:=\psi_i(1_{B_i})^{-1}\psi_{|_{B_i}}$. Then there is a canonical isomorphism of compact quantum groups
$\G\wr_{*} \aut\cong \hat{\ast}_{i=1}^d \G\wr_{*}\G^{aut}(B_i,\psi_i)$.
\end{prop}

\begin{proof}
During the proof we use the notations $M=C(H_{(B,\psi)}^+(\G))$ and $N_i=C(H_{(B_i,\psi_i)}^+(\G))$ for $1\leq i\leq d$. Let $a(\alpha)\in\mathcal{L}(B\ot H_\alpha)\ot M$, $\alpha\in\irr(\G)$ be the family of generators of $M$ and let $a(\alpha)_i\in\mathcal{L}(B_i\ot H_\alpha)\ot N_i$, $\alpha\in\irr(\G)$, be the family of generators of $N_i$, for $1\leq i\leq d$. Let $m$, $\eta$ be the multiplication and the unity of $B$ and let $m_i$, $\eta_i$ be the multiplication and the unity of $B_i$. Moreover, let $\nu_i:B_i\longrightarrow B$ be the inclusions: it is a family of isometries such that $\nu_i\nu_i^*$ are pairwise orthogonal projections and $\sum_i\nu_i\nu_i^*=\id_B$. Define the element $v(\alpha)\in\mathcal{L}(B\ot H_\alpha)\ot \ast_{i=1}^d N_i$ by $$v(\alpha)=\sum_i(\nu_i\ot\id_{H_\alpha}\ot 1)a(\alpha)_i(\nu_i^*\ot\id_{H_\alpha}\ot 1)$$
We claim that there exists a unital $\ast$-homomorphism $\Psi: M\longrightarrow \ast_{i=1}^d N_i$ such that
$$(\id_{B\ot H_\alpha}\ot \Psi)a(\alpha)=v(\alpha).$$
By the universal property of the free wreath product, it suffices to check the following.
\begin{enumerate}
\item $v(\alpha)$ is unitary.
\item $\widetilde{m\ot S}\in\Hom(v(\alpha)\ot v(\beta),v(\gamma))$ for any $\alpha,\beta,\gamma\in\irr(\G)$ and $S\in\Hom(\alpha\ot \beta,\gamma)$.
\item $\eta\in\Hom(1,v(1))$.
\end{enumerate}

$(1).$ Since the $\nu_i\nu_i^*$ are pairwise orthogonal we have $\nu_i^*\nu_k=0$ if $i\neq k$ and $\nu_i^*\nu_i=\id_{B_i}$. It follows that
\begin{eqnarray*}
v(\alpha)v(\alpha)^* &=& \sum_{i,k}(\nu_i\ot\id\ot 1)a(\alpha)_i(\nu_i^*\ot\id\ot 1)(\nu_k\ot\id\ot 1)a(\alpha)_k^*(\nu_k^*\ot\id\ot 1)\\
&=& \sum_{i}(\nu_i\ot\id\ot 1)a(\alpha)_ia(\alpha)^*_i(\nu_i^*\ot\id\ot 1)=\id_B\ot \id\ot 1.
\end{eqnarray*}
Similarly, $v(\alpha)^*v(\alpha)= \id_B\ot \id\ot 1$.

$(2).$ Observe that $\nu_j^*m(\nu_i\ot \nu_k)=\delta_{ik}\delta_{ij}m$ and $\sum_i \nu_i m_i(\nu_i^*\ot \nu_i^*)=m$. Using this, we find that the element $(\widetilde{m\ot S}\ot 1)v(\alpha)\ot v(\beta)$ is equal to
\begin{eqnarray*}
&&((m\ot S)\Sigma_{23}\ot 1)\sum_{i,k}(\nu_i\ot \id\ot \nu_k\ot \id \ot 1)(a(\alpha)_i\ot a(\beta)_k)(\nu_i^*\ot \id\ot \nu_k^*\ot \id \ot 1)\\
&=&\sum_{i,k}((m(\nu_i\ot \nu_k)\ot S)\Sigma_{23}\ot 1)(a(\alpha)_i\ot a(\beta)_k)(\nu_i^*\ot \id\ot \nu_k^*\ot \id \ot 1)\\
&=&\sum_{i,j,k}(\nu_j\ot \id\ot 1)((\nu_j^*m(\nu_i\ot \nu_k)\ot S)\Sigma_{23}\ot 1)(a(\alpha)_i\ot a(\beta)_k)(\nu_i^*\ot \id\ot \nu_k^*\ot \id \ot 1)\\
&=&\sum_{i}(\nu_i\ot \id\ot 1)((m_i\ot S)\Sigma_{23}\ot 1)(a(\alpha)_i\ot a(\beta)_i)(\nu_i^*\ot \id\ot \nu_i^*\ot \id \ot 1)\\
&=&\sum_{i}(\nu_i\ot \id\ot 1)a(\gamma)_i((m_i(\nu_i^*\ot\nu_i^*)\ot S)\Sigma_{23}\ot 1)\\
&=&\sum_{i}(\nu_i\ot \id\ot 1)a(\gamma)_i(\nu_i^*\ot \id\ot 1)((\sum_k \nu_k m_k(\nu_k^*\ot\nu_k^*)\ot S)\Sigma_{23}\ot 1)\\
&=&v(\gamma)((m\ot S)\Sigma_{23}\ot 1).
\end{eqnarray*}
$(3).$ Observe that $\nu_i^*\eta=\eta_i$ and $\sum_i \nu_i\eta_i=\eta$. Hence,
\begin{eqnarray*}
v(1)(\eta\ot 1)&=&\sum_i(\nu_i\ot 1)a(1)_i(\nu_i^*\ot 1)(\eta\ot 1)=\sum_i(\nu_i\ot 1)a(1)_i(\eta_i\ot 1)\\
&=&\sum_i(\nu_i\ot 1)(\eta_i\ot 1)=\eta\ot 1.
\end{eqnarray*}
A simple verification allows us to show that this homomorphism intertwines the comultiplications. This ends the first part of the proof.

In order to construct the inverse homomorphism we first note that, for all $1\leq i\leq d$, we have $\nu_i\nu_i^*\ot \id_{H_\alpha}\in\Hom(a(\alpha),a(\alpha))$. Indeed, consider the morphism $\widetilde{m\otimes S}\in\Hom(a(\alpha)\otimes a(1_{\G}),a(\alpha))$, where $S\in\Hom(\alpha\otimes 1_{\G},\alpha)$ is the identity morphism, and observe that
$$x_\alpha:=(\widetilde{m\otimes S})\circ (\widetilde{m\otimes S})^*= mm^*\otimes \id_{H_\alpha}=\sum_{i=1}^d \delta_i\cdot \nu_i\nu_i^* \otimes \id_{H_\alpha} \in\Hom(a(\alpha),a(\alpha)).$$
Since the projections $\nu_i\nu_i^*\ot\id_{H_\alpha}\in\mathcal{L}(B\ot H_\alpha)$ are pairwise orthogonal, sum up to $1$ and the $\delta_i$ are pairwise distinct it follows that the spectrum of $x_\alpha$ is $\{\delta_i\,:\,1\leq i\leq d\}$ and $p_i$ is the orthogonal projection on the eigenspace of $\delta_i$. Since $\Hom(a(\alpha),a(\alpha))\subset\mathcal{L}(B\ot H_\alpha)$ is a C*-subalgebra and $x_\alpha\in \Hom(a(\alpha),a(\alpha))$, it follows that each spectral projection $\nu_i\nu_i^*\ot \id_{H_\alpha}$ of $x_\alpha$ is in $\Hom(a(\alpha),a(\alpha))$.

\vspace{0.2cm}

\noindent Now, for all $1\leq i\leq d$ define $v(\alpha)_i=(\nu_i^*\ot\id_{H_\alpha}\ot 1)a(\alpha)(\nu_i\ot\id_{H_\alpha}\ot 1)\in\mathcal{L}(B_i\ot H_\alpha)\ot M$. We claim that, for all $i$, there exists a unital $\ast$-homomorphism $\Phi_i: N_i\longrightarrow M$ such that
$$(\id_{B_i\ot H_\alpha}\ot \Phi_i)a(\alpha)_i=v(\alpha)_i.$$
By the universality of the C*-algebra $N_i$ it is enough to check that
\begin{enumerate}
\item $v(\alpha)_i$ is unitary
\item $(m_i\ot S)\Sigma_{23}\in\Hom(v(\alpha)_i\ot v(\beta)_i,v(\gamma)_i)$ for any $\alpha,\beta,\gamma\in\irr(\G)$ and $S\in\Hom(\alpha\ot \beta,\gamma)$
\item $\eta_i\in\Hom(1,v(1)_i)$
\end{enumerate}
$(1).$ We have
\begin{eqnarray*}
v(\alpha)_iv(\alpha)_i^*&=&(\nu_i^*\ot\id \ot 1)a(\alpha)(\nu_i\ot\id \ot 1)(\nu_i^*\ot\id \ot 1)a(\alpha)^*(\nu_i\ot\id \ot 1)\\
&=&(\nu_i^*\ot\id \ot 1)a(\alpha)(\nu_i\nu_i^*\ot\id \ot 1)a(\alpha)^*(\nu_i\ot\id \ot 1)\\ &=&(\nu_i^*\ot\id \ot 1)(\nu_i\nu_i^*\ot\id \ot 1)a(\alpha)a(\alpha)^*(\nu_i\ot\id \ot 1)=\id_{B_i}\ot\id \ot 1.
\end{eqnarray*}
Similarly, $v(\alpha)_i^*v(\alpha)_i= \id_{B_i}\ot \id\ot 1$.

$(2).$ Since $m_i=\nu_i^*m(\nu_i\ot \nu_i)$ we have: 
$$(\nu_i m_i(\nu_i^*\ot \nu_i^*)\ot S)\Sigma_{23}=(\nu_i\nu_i^*\ot \id)(m\ot S)\Sigma_{23}(\nu_i\nu_i^*\ot \id\ot \nu_i\nu_i^*\ot \id)\in\Hom(a(\alpha)_i\ot a(\beta)_i,a(\gamma)_i).$$
Hence, the element $((m_i\ot S)\Sigma_{23}\ot 1)v(\alpha)_i\ot v(\beta)_i$ is equal to
\begin{eqnarray*}
&&((m_i\ot S)\Sigma_{23}\ot 1)(\nu_i^*\ot \id\ot \nu_i^*\ot \id \ot 1)(a(\alpha)\ot a(\beta))(\nu_i\ot \id\ot \nu_i\ot \id \ot 1)\\
&=&((m_i(\nu_i^*\ot \nu_i^*)\ot S)\Sigma_{23}\ot 1)(a(\alpha)\ot a(\beta))(\nu_i\ot \id\ot \nu_i\ot \id \ot 1)\\
&=&(\nu_i^*\ot \id\ot 1)((\nu_i m_i(\nu_i^*\ot \nu_i^*)\ot S)\Sigma_{23}\ot 1)(a(\alpha)\ot a(\beta))(\nu_i\ot \id\ot \nu_i\ot \id \ot 1)\\
&=&(\nu_i^*\ot \id\ot 1)a(\gamma)((\nu_i m_i(\nu_i^*\ot \nu_i^*)\ot S)\Sigma_{23}\ot 1)(\nu_i\ot \id\ot \nu_i\ot \id \ot 1)\\
&=&(\nu_i^*\ot \id\ot 1)a(\gamma)((\nu_i m_i(\nu_i^*\nu_i\ot \nu_i^*\nu_i)\ot S)\Sigma_{23}\ot 1)\\
&=&(\nu_i^*\ot \id\ot 1)a(\gamma)(\nu_i\ot \id\ot 1)((m_i\ot S)\Sigma_{23}\ot 1)=v(\gamma)_i((m_i\ot S)\Sigma_{23}\ot 1).
\end{eqnarray*}
$(3).$ Since $\nu_i\eta_i=(\nu_i\nu_i^*)\eta\in\Hom(1,a(1))$ we have:
$$
v(1)_i(\eta_i\ot 1)=(\nu_i^*\ot 1)a(1)(\nu_i\ot 1)(\eta_i\ot 1)=(\nu_i^*\nu_i\ot 1)(\eta_i\ot 1)=\eta_i\ot 1.
$$
This completes the proof of the existence of the morphism $\Phi_i: N_i\longrightarrow M$, for all $i$. By the universal property of the maximal free product C*-algebra, there exists a unital $\ast$-homomorphism $\Phi:\ast_{i=1}^d N_i\longrightarrow M$ such that $(\id_{B_i\ot H_\alpha}\ot \Phi)a(\alpha)_i=v(\alpha)_i$ for all $i$. It is easy to check that $\Psi$ and $\Phi$ are inverse to each other and $\Phi$ intertwines the comultiplications.
\end{proof}

%%%%%%%%%%%%%%%%%%%%%%%%%%%%%%%%%%%%%%%%%%%%%%%%%%
\section{Stability properties of the free wreath product}
%%%%%%%%%%%%%%%%%%%%%%%%%%%%%%%%%%%%%%%%%%%%%%%%%%

In this section, we present some stability results concerning the operation of free wreath product.
More precisely, we prove that the free wreath product preserves the relation of monoidal equivalence and we find under which conditions two free wreath products have isomorphic fusion semiring.
 
We start by recalling a result from \cite{drvv10} about the monoidal equivalence of quantum automorphism groups.
\begin{theorem}\label{moneqqag}
Let $\psi$ be a $\delta$ form on $B$ and $\psi'$ be a $\delta'$ form on $B'$. Then $\aut$ and $\G^{aut}(B',\psi')$ are monoidally equivalent if and only if $\delta=\delta'$.
\end{theorem}

We provide below a different proof (and we believe a simpler proof) of the first "if" using the description of the intwerners in terms of non crossing partitions discovered in \cite{pit14}.

\begin{proof}
Let $u$ (resp. $u'$) be the fundamental representation of $\aut$ (resp. $\G^{aut}(B',\psi')$). Let $\phi\,:\,\{u^{\ot k}\,:\,k\in\N\}\rightarrow\{u'^{\ot k}\,:\,k\in\N\}$ be the bijection defined by $\phi(u^{\ot k})=u'^{\ot k}$ for all $k\in\N$. We use the same notation to denote, for every $k,l\in\N$, the map $\phi(T_p)=T'_p$, where $T_p$ is the morphism in $\Hom_{\aut}(u^{\otimes k},u^{\otimes l})$ associated to a noncrossing partition $p\in NC(k,l)$ and $T_p'$ is the morphism in $\Hom_{\G^{aut}(B',\psi')}(u'^{\otimes k},u'^{\otimes l})$ associated to the same noncrossing partition $p$. As the $T_p$ and the $T'_p$ are a basis of the respective spaces of intertwiners, $\phi$ can be extended by linearity to a linear isomorphism
$$\phi:\Hom_{\aut}(u^{\otimes k},u^{\otimes l})\longrightarrow \Hom_{\G^{aut}(B',\psi')}(u'^{\otimes k},u'^{\otimes l})$$
It is clear that $\phi(\id)=\id$ because $\phi(T_{|\cdots|})=T'_{|\cdots|}$, where $|\cdots|$ is the noncrossing partition in $NC(k,k)$ which connects each of the $k$ upper points to the respective lower point.

The second property which is required is the compatibility with the tensor product: $\phi(P\otimes Q)=\phi(P)\otimes \phi(Q)$ for all $P,Q$ morphisms. For $P=T_p$ and $Q=T_q$ we have $\phi(T_p\otimes T_q)=\phi(T_{p\otimes q})=T'_{p\otimes q}=T'_p\otimes T'_q=\phi(T_p)\otimes\phi(T_q)$. The results holds for all the pairs $P,Q$ of morphisms by linearity of $\phi$.

The third property is the compatibility with respect to the adjoint: $\phi(P^*)=\phi(P)^*$ for all morphisms $P$. If $P=T_p$ we have $\phi(T_p^*)=\phi(T_{p^*})=T'_{p^*}=T'^*_{p}=\phi(T_p)^*$. The results holds for all morphisms $P$ by linearity of $\phi$.

The compatibility with the composition is a little more subtle and it is the part of the proof where the hypothesis $\delta=\delta'$ is used. We want to prove that $\phi(S\circ R)=\phi(S)\circ \phi(R)$ for all $R,S$ composable morphisms. Suppose $S=T_p$ and $R=T_q$, then we have $\phi(T_p\circ T_q)=\phi(\delta^{cy(p,q)}T_{pq})=\delta^{cy(p,q)}T'_{pq}=T'_p\circ T'_q=\phi(T_p)\circ \phi(T_q)$, where the second to last equality is true only because of the assumption $\delta=\delta'$. The results holds for all the pairs of composable morphisms by linearity of $\phi$.

In order to complete the proof, we observe that $\phi$ is an equivalence between the categories $$\mathscr{C}=\{(u^{\otimes k}, k\in \N),(\Hom(u^{\otimes k},u^{\otimes l}),\,k,l\in\N)\}\text{ and }\mathscr{C}'=\{(u'^{\otimes k}, k\in \N),(\Hom(u'^{\otimes k},u'^{\otimes l}),\,k,l\in\N)\}.$$
Then, by applying Proposition \ref{propgenmoneq}, it is possible to extend $\phi$ to an equivalence $\tilde{\phi}$ between the completion of the two categories with respect to direct sums and sub-objects and the monoidal equivalence is proved.\end{proof}
Using the same approach, we can prove the following result.

\begin{theorem}
Let $B$ (resp. $B'$) be a finite dimensional C*-algebras of dimension at least $4$ with a $\delta$ (resp. $\delta'$)-form $\psi$ (resp. $\psi'$). If $\G_1$ and $\G_2$ are monoidally equivalent and $\delta=\delta'$ then $H^+_{(B,\psi)}(\G_1)$ and $H^+_{(B',\psi')}(\G_2)$ are monoidally equivalent.
\end{theorem}

\begin{proof}
Let $\phi:\irr(\G_1)\longrightarrow\irr(\G_2)$ be the map which establishes the monoidal equivalence between $\G_1$ and $\G_2$. 
The proof is divided into two parts: firstly, we define the map $\Phi$ (which satisfies the properties of the monoidal equivalence) on the basic representations which generate the two free wreath products and later on we will observe that we can extend $\Phi$ to all the irreducible representations.
Let us denote by $a(\alpha), \alpha\in\irr(\G_1)$ (resp. $a'(\beta), \beta\in\irr(\G_2)$) the basic representations of $H^+_{(B,\psi)}(\G_1)$ (resp. $H^+_{(B',\psi')}(\G_2)$). Let $\phi\,:\,\{a(\alpha)\,:\,\alpha\in\irr(\G_1)\setminus\{1\}\}\sqcup\{1\}\rightarrow\{a'(\beta)\,:\,\beta\in\irr(G_2)\setminus\{1\}\}\sqcup\{1\}$ be the bijection defined by $\Phi(a(\alpha))=a'(\phi(\alpha))$ and $\Phi(1)=1$.

We use the same notation to define, for every $\alpha_1,...,\alpha_k,\beta_1,...,\beta_l\in\irr(\G_1)$, the map $$\Phi(T_{p,S})=T'_{p,\phi(S)}$$ where $T_{p,S}$ is the morphism in $\Hom_{H^+_{(B,\psi)}(\G_1)}(a(\alpha_1)\otimes...\otimes a(\alpha_k),a(\beta_1)\otimes...\otimes a(\beta_l))$ associated to a noncrossing partition $p\in NC_{\G_1}(\alpha_1,...,\alpha_k;\beta_1,...,\beta_l)$ decorated with an intertwiner $S$ and $T'_{p,\phi(S)}$ is the morphism in $\Hom_{H^+_{(B',\psi')}(\G_2)}(a'(\phi(\alpha_1))\otimes...\otimes a'(\phi(\alpha_k)),a'(\phi(\beta_1))\otimes...\otimes a'(\phi(\beta_l)))$ associated to the same noncrossing partition $p$, decorated with the intertwiner $\phi(S)$. 

Recall that the morphisms of type $T_{p,S}$ and $T'_{p,\phi(S)}$ linearly span the respective spaces of intertwiners, the maps $T_p$ $(T'_p)$ associated to distinct noncrossing partitions are linearly independent, $\phi$ is a linear isomorphism and, by Lemma \ref{lintp}, the maps $S\mapsto T_{p,S}$ and $S\mapsto T'_{p,\phi(S)}$ are linear. The map $\Phi$ can then be extended by linearity to a linear isomorphism:
$$
\Hom(a(\alpha_1)\otimes...\otimes a(\alpha_k),a(\beta_1)\otimes...\otimes a(\beta_l))\longrightarrow
\Hom(a'(\phi(\alpha_1))\otimes...\otimes a'(\phi(\alpha_k)),a'(\phi(\beta_1))\otimes...\otimes a'(\phi(\beta_l)))
$$
Let us check that $\Phi$ satisfies the conditions of Definition \ref{defmoneq}.

The first condition $\Phi(\id)=\id$ is clear because $\Phi(T_{|\cdots|,\id})=T'_{|\cdots|,\phi(\id)}$ and $\phi(\id)=\id$.

The second property which is required is the compatibility with the tensor product: $\Phi(P\otimes Q)=\Phi(P)\otimes \Phi(Q)$ for all $P,Q$ morphisms. If $P=T_{p,S}$ and $Q=T_{q,R}$ we have $\Phi(T_{p,S}\otimes T_{q,R})=\Phi(T_{p\otimes q,S\otimes R})=T'_{p\otimes q,\phi(S\otimes R)}=T'_{p\otimes q,\phi(S)\otimes \phi(R)}=T'_{p,\phi(S)}\otimes T'_{q,\phi(R)}=\Phi(T_{p,S})\otimes\Phi(T_{q,R})$. The results holds for all the pairs $P,Q$ of morphisms by linearity of $\Phi$.

The third property is the compatibility with respect to the adjoint: $\Phi(P^*)=\Phi(P)^*$ for all morphisms $P$. If $P=T_{p,S}$ we have $\Phi(T_{p,S}^*)=\Phi(T_{p^*,S^*})=T'_{p^*,\phi(S^*)}=T'_{p^*,\phi(S)^*}=T'^*_{p,\phi(S)}=\Phi(T_{p,S})^*$. The results holds for all the morphisms $P$ by linearity of $\Phi$.

The last condition is the compatibility with respect to the composition: $\Phi(P\circ Q)=\Phi(P)\circ \Phi(Q)$ for all composable morphisms $P,Q$. We observe that, because of Theorem \ref{moneqqag}, we have $\delta=\delta'$.
Now, suppose that $P=T_{p,S}$ and $Q=T_{q,R}$. We have $\Phi(T_{p,S}\circ T_{q,R})=\Phi(\delta^{cy(p,q)}T_{qp,RS})=\delta^{cy(p,q)}T'_{qp,\phi(RS)}= \delta^{cy(p,q)}T'_{qp,\phi(R)\circ\phi(S)}=T'_{p,\phi(R)}\circ T'_{q,\phi(S)}=\Phi(T_{p,R})\circ\Phi(T_{q,S})$. The results holds for all the pairs $P,Q$ of composable morphisms by linearity of $\Phi$.\\
As in the proof of Theorem \ref{moneqqag}, we can observe that $\Phi$ is an equivalence between the categories 
$$\mathscr{C}=\{(\bigotimes_{i=1}^k a(\alpha_i), \alpha_i\in\irr(\G_1), k\in \N),(\Hom(\bigotimes_{i=1}^k a(\alpha_i),\bigotimes_{j=1}^l a(\beta_i)),\, \alpha_i,\beta_i\in\irr(\G_1)\,k,l\in\N)\}$$
and 
$$\mathscr{D}=\{(\bigotimes_{i=1}^k a'(\alpha_i), \alpha_i\in\irr(\G_2), k\in \N),(\Hom(\bigotimes_{i=1}^k a'(\alpha_i),\bigotimes_{j=1}^l a'(\beta_i)),\, \alpha_i,\beta_i\in\irr(\G_2),\,k,l\in\N)\}$$
By Proposition \ref{propgenmoneq} we extend $\Phi$ to an equivalence $\tilde{\Phi}$ between the completion of the two categories with respect to the direct sums and  the sub-objects  which are the representation categories of $H^+_{(B,\psi)}(\G_1)$ and $H^+_{(B',\psi')}(\G_2)$ respectively. The extension $\tilde{\Phi}$ is the required monoidal equivalence.
\end{proof}

Given a compact quantum group $\G$, we denote by $R^+(\G)$ its fusion semi-ring.

\begin{theorem}
Let $B,B'$ be two finite dimensional C*-algebras of dimension at least $4$ endowed with the $\delta$-form $\psi$ and the $\delta'$-form $\psi'$ respectively and $\G_1$, $\G_2$ two compact quantum groups. If $R^+(\G1)\simeq R^+(\G_2)$ then $R^+(H^+_{(B,\psi)}(\G_1))\simeq R^+(H^+_{(B',\psi')}(\G_2))$.
\end{theorem}

\begin{proof}
We recall that the irreducible representations of the free wreath product $H^+_{(B,\psi)}(\G_1)$ can be indexed by the elements of the monoid $M$, i.e. by the words written using as letters the irreducible representations of $\G_1$. The monoid is endowed with the three operations of involution, concatenation and fusion introduced in Definition \ref{opermonm}. We will denote $r_x, x\in M$ the irreducible representations. 
Let $\Phi$ be the map given by $\Phi(r_x)=r_{\phi(x)}$, where for every word $x=(\alpha_1,...,\alpha_k)\in M$, we define $\phi(x):=(\phi(\alpha_1),...,\phi(\alpha_k))$. This definitions makes sense because an isomorphism of the fusion semirings like $\phi$, when restricted to $\irr(\G_1)$ is a bijection onto $\irr(\G_2)$.
We observe that the map $\Phi$ can be extended by additivity to
$$\Phi:R^+(H^+_{(B,\psi)}(\G_1))\longrightarrow R^+(H^+_{(B',\psi')}(\G_2))$$
Moreover, $\Phi$ is an isomorphism because $\phi$ is a bijection and $\irr(\G_i)$ is a basis of $R^+(\G_i),\, i=1,2$.
Then, the proof reduces to show that, for all $x,y\in M$, we have 
\begin{equation}\label{compatib}
\Phi(r_x\otimes r_y)=\Phi(r_x)\otimes \Phi(r_y)
\end{equation}
\begin{equation}\label{compatib1}
\Phi(\overline{r_x})=\overline{\Phi(r_x)}
\end{equation}
For this verification it is necessary to state a preliminary result which assures the compatibility between the map $\phi$ and the operations of the monoid $M$: more precisely, we have, for all $u,v\in M$, $\phi(u,v)=(\phi(u),\phi(v))$, $\phi(\overline{u})=\overline{\phi(u)}$ and $\phi(u.v)=\phi(u).\phi(v)$. Indeed, for $u=(\alpha_1,...,\alpha_k)\in M$ and $v=(\beta_1,...,\beta_l)\in M$ we have $\phi(u,v)=\phi((\alpha_1,...,\alpha_k,\beta_1,...,\beta_l))=(\phi(\alpha_1),...,\phi(\alpha_k),\phi(\beta_1),...,\phi(\beta_l))=(\phi(u),\phi(v))$. Similarly 
$\phi(\overline{u})=\phi((\overline{\alpha_k},...,\overline{\alpha_1}))=
(\phi(\overline{\alpha_k}),...,\phi(\overline{\alpha_1})) =
(\overline{\phi(\alpha_k)},...,\overline{\phi(\alpha_1)}) =
\overline{\phi(u)}$ and,
\begin{eqnarray*}
\phi(u.v)&=&\phi(\bigoplus_{\gamma\subset \alpha_k\otimes \beta_1}(\alpha_1,...,\alpha_{k-1},\gamma,\beta_2,...,\beta_l)) \\&=&
\bigoplus_{\phi(\gamma)\subset \phi(\alpha_k\otimes \beta_1)}(\phi(\alpha_1),...,\phi(\alpha_{k-1}),\phi(\gamma),\phi(\beta_2),...,\phi(\beta_l)) \\ &=&
\bigoplus_{\phi(\gamma)\subset \phi(\alpha_k)\otimes \phi(\beta_1)}(\phi(\alpha_1),...,\phi(\alpha_{k-1}),\phi(\gamma),\phi(\beta_2),...,\phi(\beta_l)) \\ &=&
\phi((\alpha_1,...,\alpha_k)).\phi((\beta_1,...,\beta_l)).
\end{eqnarray*}
Let us prove equation \ref{compatib}. We have
\begin{eqnarray*}
\Phi(r_x\otimes r_y) &= &\Phi(\sum_{\substack{x=u,t \\ y=\overline{t},v}}r_{u,v}\oplus \sum_{\substack{x=u,t \\ y=\overline{t},v \\  u\neq\emptyset , v\neq \emptyset}}r_{u.v})=
\sum_{\substack{x=u,t \\ y=\overline{t},v}}\Phi(r_{u,v})\oplus \sum_{\substack{x=u,t \\ y=\overline{t},v \\  u\neq\emptyset , v\neq \emptyset}}\Phi(r_{u.v}) \\ &=&
\sum_{\substack{\phi(x)=\phi(u,t) \\ \phi(y)=\phi(\overline{t},v)}}r_{\phi(u,v)}\oplus \sum_{\substack{\phi(x)=\phi(u,t) \\ \phi(y)=\phi(\overline{t},v) \\  \phi(u)\neq\emptyset , \phi(v)\neq \emptyset}}r_{\phi(u.v)} =
\sum_{\substack{\phi(x)=\phi(u),\phi(t) \\ \phi(y)=\overline{\phi(t)},\phi(v)}}r_{\phi(u),\phi(v)}\oplus \sum_{\substack{\phi(x)=\phi(u),\phi(t) \\ \phi(y)=\overline{\phi(t)},\phi(v) \\  \phi(u)\neq\emptyset , \phi(v)\neq \emptyset}}r_{\phi(u).\phi(v)} \\ &=&
r_{\phi(x)}\otimes r_{\phi(y)} =
\Phi(r_x)\otimes \Phi(r_y).
\end{eqnarray*}
The relation \ref{compatib1} is clear because  $\Phi(\overline{r_x})=\Phi(r_{\overline{x}})=r_{\phi(\overline{x})}=r_{\overline{\phi(x)}}=\overline{r_{\phi(x)}}=\overline{\Phi(r_x)}$.
\end{proof}

%%%%%%%%%%%%%%%%%%%%%%%%%%%%%%%%%%%%%%%%%%%%%%%%%%
\section{Algebraic and analytic properties}\label{proprieta}
%%%%%%%%%%%%%%%%%%%%%%%%%%%%%%%%%%%%%%%%%%%%%%%%%%

In this section we apply the monoidal equivalence result proved in Theorem \ref{teomoneq} and the free product decomposition proved in Proposition \ref{freeprodcqggen} to deduce some properties of the reduced C*-algebras and the von Neumann algebras associated to a free wreath product.

Let us start we a remark about the free product decomposition of the reduced C*-algebras and the von Neumann algebras associated to a free wreath product.

\begin{remark}\label{decomp11}
By using the free product decomposition of Proposition \ref{freeprodcqggen} and a remark of Wang \cite{wan95}, we observe that the Haar measure of $\wprhg$ is the free product of the Haar measures of its factors. Hence, the following isomorphisms hold:
$$(C_r(\wprhg),h)\cong \underset{\text{red}}{\ast}_{i=1}^k (C_r(H^+_{(B_i,\psi_i)}(\G)),h_i)$$
$$(L^\infty(\wprhg),h)\cong \ast_{i=1}^k (L^\infty(H^+_{(B_i,\psi_i)}(\G)),h_i)$$
where $h$ and $h_i$ are the Haar states on the respective C*-algebras and the second free product is the free product of von Neumann algebras.
\end{remark}

We can now prove some approximation properties. We refer to \cite{dcfy14} for the definition of the central ACPAP.

\begin{prop}\label{propacpap}
Let $B$ be a finite dimensional C*-algebra endowed with a state $\psi$. Let $\widehat{\G}$ be a discrete quantum group with the central ACPAP and consider the free wreath product $\wprhg$. Suppose that, in the free product decomposition $\hat{\ast}_{i=1}^k H^+_{(B_i,\psi_i)}(\G)$, we have $\dim(B_i)\geq 4$ for all $i$. Then, the dual of $\wprhg$ has the central ACPAP.
\end{prop}
\begin{proof}
From \cite[Proposition 24]{dcfy14} the central ACPAP is preserved by the operation of free product so it is enough to prove the result when $\psi$ is a $\delta$-form. In this case, the free wreath product is monoidally equivalent to a quantum subgroup of $\G\hat{\ast} SU_q(2),\, q\in(0,1]$, by Theorem \ref{teomoneq}. Now, the dual of $SU_q(2), q\in[-1,1],q\neq 0$ has the central ACPAP (\cite[Theorem 25]{dcfy14}) so the free product has the central ACPAP. This property is also preserved when passing to quantum subgroups (\cite[Lemma 23]{dcfy14}). Since the central ACPAP is preserved under monoidal equivalence the proof is complete.
\end{proof}

The Haagerup property is implied by the central ACPAP, therefore we have the following corollary.

\begin{cor}
Consider the assumptions and the notations of Proposition \ref{propacpap}. Then, the von Neumann algebra $L^\infty(\wprhg)$ has the Haagerup property and the W*-CCAP.
\end{cor}

\begin{example}
From \cite{dcfy14}, we know that the dual of $SU_q(2)$ with $q\in[-1,1],q\neq 0$ as well as the duals of the free unitary and orthogonal quantum groups $U^+(F)$ and $O^+(F)$ with $\dim(F)\geq 2$, have the central $ACPAP$. It follows that, for any finite dimensional C*-algebra $B$, $\dim(B)\geq 4$ endowed with a $\delta$-form $\psi$, the von Neumann algebras $L^{\infty}(H_{(B,\psi)}^+(U^+(F)))$, $L^{\infty}(H_{(B,\psi)}^+(O^+(F)))$ and $L^{\infty}(H_{(B,\psi)}^+(SU_q(2)))$ have the Haagerup property.
\end{example}

Recall that a discrete quantum group $\widehat{\G}$ is said to be exact if its reduced C*-algebra $C_r(\G)$ is exact.

\begin{prop}\label{exactness}
Let $B$ be a finite dimensional C*-algebra endowed with a state $\psi$. Let $\widehat{\G}$ be an exact discrete quantum group and consider the reduced C*-algebra $C_r(\wprhg)$ with free product decomposition $\underset{\text{red}}{\ast}_{i=1}^k C_r(H^+_{(B_i,\psi_i)}(\G))$. If $\dim(B_i)\geq 4$ for all $i$, then the dual of $\wprhg$ is exact.
\end{prop}

\begin{proof}
Since exactness is preserved by reduced free products \cite{dyk04}, it is enough to show the result for the factors $\G\wr_{*}\G^{aut}(B_i,\psi_i)$. Since exactness is conserved under monoidal equivalence \cite{vv07}, it suffices to prove the exactness of $C_r(\HH)$. This is a subalgebra of a free product whose factors are exact: $\widehat{\G}$ is exact by hypothesis and $\widehat{SU_q(2)}$ is exact as a consequence of its amenability. Exactness passes to subalgebras and free products so $C_r(\HH)$ is exact.
\end{proof}

\begin{example}
Let $B$ a finite dimensional C*-algebra, $\dim(B)\geq 4$, endowed with a $\delta$-form $\psi$. Then, the dual of $H^+_{(B,\psi)}(G)$ is exact for any classical compact group $G$. Moreover, $SU_q(2)$, $q\in (-1,1)$, $q\neq 0$ is exact, so the dual of $H^+_{(B,\psi)}(SU_q(2))$ is exact. In \cite{vv07} it is proved that also $O^+(F)$, for $\dim(F)\geq 3$ is exact; the exactness of the dual of $H_{(B,\psi)}^+(O^+(F))$ follows.
\end{example}

We now study the simplicity of the reduced C*-algebra and the uniqueness of the trace.

\begin{prop}
Let $B$ be a finite dimensional C*-algebra endowed with a trace $\psi$. Let $\G$ be a compact quantum group of Kac type. Consider the reduced C*-algebra $C_r(\wprhg)$ and its free product decomposition $\underset{\text{red}}{\ast}_{i=1}^k C_r(H^+_{(B_i,\psi_i)}(\G))$. If there is either only one factor (i.e. $\psi$ is a $\delta$-trace) and $\dim(B)\geq 8$ or there are two or more factors with $\dim(B_i)\geq 4$ for all $i$, then $C_r(\wprhg)$ is simple with unique trace.
\end{prop}

\begin{proof}
The techniques and the results used in \cite{pit14}[Proposition 3.2] for the proof of the analogous result in the case of the dual of a discrete group can be applied also here. More precisely, if $k\geq 2$, the proof is based on a proposition of Avitzour (see \cite[Section 3]{avi82}) stating that all the free product satisfying some basic conditions are simple with unique trace. In the other case, when the decomposition has only one factor, we can generalise the proof presented by Lemeux in \cite[Theorem 3.5]{lem14}. The assumption of Lemeux on the minimum number of irreducible representations of $\G$ can be removed by making use of a trick introduced by Wahl in \cite{wah14}.
\end{proof}

%%%%%%%%%%%%%%%%%%%%%%%%%%%%%%%%%%%%%%%%
\section{The free wreath product $\G^{aut}(B',\psi')\wr_*\autd$}
%%%%%%%%%%%%%%%%%%%%%%%%%%%%%%%%%%%%

In this section we prove our free wreath product formula. First, we obtain some general results which will be fundamental for the proof of such a formula. Let us introduce some notations.

Let $B$ be a finite dimensional C*-algebra endowed with a $\delta$-form $\psi$. Let $m$ and  $\eta$ be the multiplication and unity on $B$ respectively. Consider the free wreath product $\wprhg$ of a compact quantum group $\G$ by the quantum automorphism group $\aut$. Choose a complete set of irreducible representations $u_\alpha\in\mathcal{L}(H_\alpha)\ot C(\Gq)$, $\alpha\in\irr(\Gq)$. We recall that $C(\wprhg)$ is generated by the coefficients of a family of unitary representations $a(\alpha), \alpha\in\irr(\G)$ such that $\eta\in\Hom(1,a(1))$ and, for all $\alpha,\beta,\gamma\in\irr(\Gq)$ and all $S\in\Hom(u_\alpha\ot u_\beta,u_\gamma)$, we have a morphism $a(S):=\widetilde{m\ot S}=(m\ot S)\circ\Sigma_{23}\in\Hom(a(\alpha)\ot a(\beta),a(\gamma))$.

For any finite dimensional representation $u\in\mathcal{L}(H_u)\ot C(\Gq)$ we define an element $a(u)\in\mathcal{L}(B\ot H_u)\ot C(\wprhg)$ in the following way. For all $\alpha\in\irr(\Gq)$ such that $\alpha\subset u$, we choose a family of isometries $S_{\alpha,k}\in\mathcal{L}(H_\alpha,H_u)$ such that $S_{\alpha,k}\in\Hom(u_\alpha,u)$, $1\leq k\leq \dim(\Hom(u_\alpha,u))$ and $S_{\alpha,k}S_{\alpha,k}^*$ are pairwise orthogonal projections with $\sum_{\alpha,k}S_{\alpha,k}S_{\alpha,k}^*=\id_{H_u}$. Hence, $u=\sum_{\alpha,k}(S_{\alpha,k}\ot 1)u_\alpha(S_{\alpha,k}^*\ot 1)$. Define
$$a(u)=\sum_{\alpha,k}(\id_B\ot S_{\alpha,k}\ot 1)a(\alpha)(\id_B\ot S^*_{\alpha,k}\ot 1)\in\mathcal{L}(B\ot H_u)\ot C(\wprhg)$$
Observe that the definition is coherent with the original $a(\alpha)$. For every finite dimensional unitary representations of $u,v$ and $w$ of $\Gq$ and every morphism $S\in\Hom(u\ot v,w)$ define the linear map $a(S):B\ot H_u\ot B\ot H_v\rightarrow B\ot H_w$ by $a(S)=(m_B\ot S)\Sigma_{23}$. Observe that $a(S)$ is coherent with the original definition, when $u,v$ and $w$ are in $\irr(\Gq)$.

\begin{prop}\label{PropFunctor}
For all finite dimensional unitary representation $u,v$ of $\Gq$ and all $S\in\Hom(u,v)$ the following holds.
\begin{enumerate}
\item The definition of $a(u)$ does not depend on the choice of the family of isometries $S_{\alpha,k}$.
\item $a(u)$ is a unitary representation of $\wprhg$ on $B\ot H_u$.
\item $\id_B\ot S\in\Hom(a(u),a(v))$.
\item If $u\simeq v$ then $a(u)\simeq a(v)$.
\item For all finite dimensional unitary representations $w$ of $\G$ and all $S\in\Hom(u\ot v,w)$, $a(S)\in\Hom(a(u)\ot a(v),a(w))$. Moreover, 
$a(S)a(S)^*=\delta\id_B\ot SS^*$. In particular, if $S^*$ is isometric then $\delta^{-\frac{1}{2}}a(S)^*$ is isometric.
\end{enumerate}
\end{prop}

\begin{proof}
$(1).$ Let $T_{\beta,l}\in\mathcal{L}(H_\beta,H_u)$ be another family of isometries such that $T_{\beta,l}\in\Hom(u_\beta,u)$, $1\leq l\leq \dim(\Hom(u_\beta,u))$ and $T_{\beta,l}T_{\beta,l}^*$ are pairwise orthogonal projections with $\sum_{\beta,l}T_{\beta,l}T_{\beta,l}^*=\id_{H_u}$. Observe that $T_{\beta,l}^*S_{\alpha,k}\in\Hom(u_\alpha,u_\beta)$. Therefore, there exists $\lambda^\beta_{kl}\in\C$ such that $T_{\beta,l}^*S_{\alpha,k}=\delta_{\alpha,\beta}\lambda^\beta_{kl}\id_{H_{u_\beta}}$. Also note that $\sum_{k}\lambda^\beta_{kl}S_{\beta,k}^*=\sum_{k}T_{\beta,l}^*S_{\beta,k}S_{\beta,k}^*=\sum_{\alpha,k}T_{\beta,l}^*S_{\alpha,k}S_{\alpha,k}^*$ since $T_{\beta,l}^*S_{\alpha,k}=0$ for $\alpha\neq\beta$. Hence, $\sum_{k}\lambda^\beta_{kl}S_{\beta,k}^*=T_{\beta,l}^*\left(\sum_{\alpha,k}S_{\alpha,k}S_{\alpha,k}^*\right)=T_{\beta,l}^*$. It follows that
\begin{eqnarray*}
&&\sum_{\alpha,k}(\id_B\ot S_{\alpha,k}\ot 1)a(\alpha)(\id_B\ot S^*_{\alpha,k}\ot 1)=\sum_{\alpha,\beta,k,l}(\id_B\ot T_{\beta,l}T_{\beta,l}^*S_{\alpha,k}\ot 1)a(\alpha)(\id_B\ot S^*_{\alpha,k}\ot 1)\\
&=&\sum_{\alpha,\beta,k,l}(\id_B\ot T_{\beta,l}\delta_{\alpha,\beta}\lambda^\beta_{kl}\ot 1)a(\alpha)(\id_B\ot S^*_{\alpha,k}\ot 1)=\sum_{\beta,k,l}(\id_B\ot T_{\beta,l}\ot 1)a(\beta)(\id_B\ot \lambda^\beta_{kl}S^*_{\beta,k}\ot 1)\\
&=&\sum_{\beta,l}(\id_B\ot T_{\beta,l}\ot 1)a(\beta)(\id_B\ot\left(\sum_k\lambda^\beta_{kl}S^*_{\beta,k}\right)\ot 1)=\sum_{\beta,l}(\id_B\ot T_{\beta,l}\ot 1)a(\beta)(\id_B\ot T^*_{\beta,l}\ot 1).
\end{eqnarray*}

$(2)$ is obvious and $(4)$ follows from $(3)$. Let us prove $(3)$. Write $T_{\beta,l}\in\mathcal{L}(H_\beta,H_v)$ the chosen isometries such that $T_{\beta,l}\in\Hom(u_\beta,v)$, $1\leq l\leq \dim(\Hom(u_\beta,v))$ and $T_{\beta,l}T_{\beta,l}^*$ are pairwise orthogonal projections with $\sum_{\beta,l}T_{\beta,l}T_{\beta,l}^*=\id_{H_v}$. Observe that $T_{\beta,l}^*SS_{\alpha,k}\in\Hom(u_\alpha,u_\beta)$. Hence there exists $\lambda^\beta_{kl}\in\C$ such that 
$T_{\beta,l}^*SS_{\alpha,k}=\delta_{\alpha,\beta}\lambda^\beta_{kl}$. Also note that $\sum_{k}\lambda^\beta_{kl}S_{\beta,k}^*=\sum_{k}T_{\beta,l}^*SS_{\beta,k}S_{\beta,k}^*=\sum_{\alpha,k}T_{\beta,l}^*SS_{\alpha,k}S_{\alpha,k}^*$ since $T_{\beta,l}^*SS_{\alpha,k}=0$ for $\alpha\neq\beta$. Hence, $\sum_{k}\lambda^\beta_{kl}S_{\beta,k}^*=T_{\beta,l}^*S\left(\sum_{\alpha,k}S_{\alpha,k}S_{\alpha,k}^*\right)=T_{\beta,l}^*S$. It follows that
\begin{eqnarray*}
(\id_B\ot S\ot 1)a(u)
&=&\sum_{\alpha,k}(1\ot SS_{\alpha,k}\ot 1)a(\alpha)(1\ot S^*_{\alpha,k}\ot 1)\\&=&\sum_{\alpha,\beta,k,l}(1\ot T_{\beta,l}T_{\beta,l}^*SS_{\alpha,k}\ot 1)a(\alpha)(1\ot S^*_{\alpha,k}\ot 1)\\
&=&\sum_{\alpha,\beta,k,l}(1\ot T_{\beta,l}\delta_{\alpha,\beta}\lambda^\beta_{kl}\ot 1)a(\alpha)(1\ot S^*_{\alpha,k}\ot 1)\\
&=&\sum_{\beta,k,l}(1\ot T_{\beta,l}\ot 1)a(\beta)(1\ot \lambda^\beta_{kl}S^*_{\beta,k}\ot 1)\\
&=&\sum_{\beta,l}(1\ot T_{\beta,l}\ot 1)a(\beta)(1\ot\left(\sum_k\lambda^\beta_{kl}S^*_{\beta,k}\right)\ot 1)\\
&=&a(v)(\id_B\ot S\ot 1).
\end{eqnarray*}

$(5)$. Consider the decompositions $u=\sum_{\alpha,k}(U_{\alpha,k}\ot 1)u_\alpha(U^*_{\alpha,k}\ot 1)$, $v=\sum_{\beta,l}(V_{\beta,l}\ot 1)u_\beta(V^*_{\beta,l}\ot 1)$ and $w=\sum_{\gamma,j}(W_{\gamma,j}\ot 1)u_\gamma(W^*_{\gamma,j}\ot 1)$. Then, with $A=C(\wprhg)$, $a(u)_{13}a(v)_{23}$ is equal to:
$$\sum_{\alpha,\beta,k,l}(\id_B\ot U_{\alpha,k}\ot\id_{B\ot H_v}\ot 1_A)a(\alpha)_{13}(\id_B\ot U_{\alpha,k}^*\ot\id_B\ot V_{\beta,l}\ot 1_A) a(\beta)_{23}(\id_{B\ot H_u}\ot\id_B\ot V_{\beta,l}^*\ot 1_A)
.$$
We have $a(S)(\id_B\ot U_{\alpha,k}\ot\id_{B\ot H_v})=m_B\ot\left( S\circ(U_{\alpha,k}\ot\id_{H_v})\right)\circ\Sigma_{23}$ and, by using $\id_{H_v}=\sum V_{\beta,l} V_{\beta,l}^*$ and $\id_{H_w}=\sum W_{\gamma,j} W_{\gamma,j}^*$, we find $$ S\circ(U_{\alpha,k}\ot\id_{H_v})=\sum W_{\gamma,j}\left(W^*_{\gamma,j}\circ S\circ(U_{\alpha,k}\ot V_{\beta,l})\right)V_{\beta,l}^*.$$ Hence $a(S)(\id_B\ot U_{\alpha,k}\ot\id_{B\ot H_v})$ is equal to:
\begin{eqnarray*}
&&\sum_{\beta,\gamma,l,j}\left[m_B\ot\left( W_{\gamma,j}\left(W^*_{\gamma,j}\circ S\circ(U_{\alpha,k}\ot V_{\beta,l})\right)V_{\beta,l}^*\right) \right]\circ\Sigma_{23}\\
&=&\sum(\id_B\ot W_{\gamma,j})\left[m_B\ot\left(W_{\gamma,j}^*\circ S\circ U_{\alpha,k}\ot V_{\beta,l}\right)\right](\id_B\ot \id_B\ot\id_{H_\alpha}\ot V_{\beta,l}^*)\circ\Sigma_{23}\\
&=&\sum(\id_B\ot W_{\gamma,j})\left[m_B\ot\left(W_{\gamma,j}^*\circ S\circ U_{\alpha,k}\ot V_{\beta,l}\right)\right]\Sigma_{23}(\id_B\ot\id_{H_\alpha}\ot \id_B\ot V_{\beta,l}^*)\\
&=&\sum_{\beta,\gamma,l,j}(\id_B\ot W_{\gamma,j})a(W_{\gamma,j}^*\circ S\circ U_{\alpha,k}\ot V_{\beta,l})(\id_B\ot\id_{H_\alpha}\ot \id_B\ot V_{\beta,l}^*),
\end{eqnarray*}
where $W_{\gamma,j}^*\circ S\circ U_{\alpha,k}\ot V_{\beta,l}\in\Hom(u_\alpha\ot u_\beta,u_\gamma)$. Hence, $a(W_{\gamma,j}^*\circ S\circ U_{\alpha,k}\ot V_{\beta,l})\in\Hom(a(\alpha)\ot a(\beta), a(\gamma))$ and, by using $(3)$, we find $T_{\alpha,\beta,\gamma,k,l,j}=(\id_B\ot W_{\gamma,j})a(W_{\gamma,j}^*\circ S\circ U_{\alpha,k}\ot V_{\beta,l})\in\Hom(a(\alpha)\ot a(\beta), a(w))$. Hence, we find that $(a(S)\ot 1_A)(a(u)_{13}a(v)_{23})$ is equal to:
$$\begin{array}{l}
\sum_{\alpha,\beta,\beta',\gamma,k,l,l',j}T_{\alpha,\beta',\gamma,k,l',j}(\id_B\ot\id_{H_\alpha}\ot \id_B\ot V_{\beta',l'}^*)\ot 1_Aa(\alpha)_{13} \\ \hspace{5.3cm} (\id_B\ot U_{\alpha,k}^*\ot\id_B\ot V_{\beta,l}\ot 1_A) a(\beta)_{23}(\id_{B\ot H_u}\ot\id_B\ot V_{\beta,l}^*\ot 1_A)\end{array}$$
Since $(\id_B\ot\id_{H_\alpha}\ot \id_B\ot V_{\beta',l'}^*\ot 1_A)a(\alpha)_{13}(\id_B\ot U_{\alpha,k}^*\ot\id_B\ot V_{\beta,l}\ot 1_A)$ is equal to:
$$a(\alpha)_{13}(\id_B\ot U_{\alpha,k}^*\ot\id_B\ot V_{\beta',l'}^*V_{\beta,l}\ot 1_A)=\delta_{\beta,\beta'}\delta_{l,l'}a(\alpha)_{13}(\id_B\ot U_{\alpha,k}^*\ot\id_B\ot \id_{H_\beta}\ot 1_A),$$
we find that $(a(S)\ot 1_A)(a(u)_{13}a(v)_{23})$ is equal to:
\begin{eqnarray*}
&&\sum_{\alpha,\beta,\gamma,k,l,,j}(T_{\alpha,\beta,\gamma,k,l,j}\ot 1_A)a(\alpha)_{13}(\id_B\ot U_{\alpha,k}^*\ot\id_{B\ot H_\beta}\ot 1_A)a(\beta)_{23}(\id_{B\ot H_u}\ot\id_B\ot V_{\beta,l}^*\ot 1_A)\\
&=&\sum_{\alpha,\beta,\gamma,k,l,,j}(T_{\alpha,\beta,\gamma,k,l,j}\ot 1_A)a(\alpha)_{13}a(\beta)_{23}(\id_B\ot U_{\alpha,k}^*\ot\id_B\ot V_{\beta,l}^*\ot 1_A)\\
&=&\sum_{\alpha,\beta,\gamma,k,l,,j}a(w)(T_{\alpha,\beta,\gamma,k,l,j}\ot 1_A)(\id_B\ot U_{\alpha,k}^*\ot\id_B\ot V_{\beta,l}^*\ot 1_A).
\end{eqnarray*}
Hence, it suffices to check that $a(S)=\sum_{\alpha,\beta,\gamma,k,l,,j}T_{\alpha,\beta,\gamma,k,l,j}\circ(\id_B\ot U_{\alpha,k}^*\ot\id_B\ot V_{\beta,l}^*)$. This follows from the equation $a(S)(\id_B\ot U_{\alpha,k}\ot\id_{B\ot H_v})=\sum_{\beta,\gamma,l,j}T_{\alpha,\beta,\gamma,k,l,j}(\id_B\ot\id_{H_\alpha}\ot \id_B\ot V_{\beta,l}^*)$ for all $\alpha,k$ and the fact that $\id_{H_\alpha}=\sum_{\alpha,k} U_{\alpha,k} U_{\alpha,k}^*$. Finally, we have
$$a(S)a(S)^*=(m_B\ot S)\Sigma_{23}\Sigma_{23}^*(m_B^*\ot S^*)=(m_Bm_B^*\ot SS^*)=\delta\id_{B}\ot SS^*$$
\end{proof}

\begin{remark}\label{RmkIsometries}
If we apply assertion $(5)$ of Proposition \ref{PropFunctor} with $w=u\ot v$ and $S=\id_{H_u\ot H_v}$, we get a morphism $S_{u,v}\in\Hom(a(u)\ot a(v),a(u\ot v))$.
Hence, $T_{u,v}=\delta^{-\frac{1}{2}}S_{u,v}^*\in\Hom(a(u\ot v),a(u)\ot a(v))$ is isometric so we always have $a(u\ot v)\subset a(u)\ot a(v)$.
\end{remark}

\begin{theorem}\label{teoisom}
Let $B,B'$ be two finite dimensional C*-algebras, $\dim B,B'\geq 4$, endowed with a $\delta$-form $\psi$ and a $\delta'$-form $\psi'$ respectively. Let $(B,\psi,d)$ be a finite quantum graph. There is a unital $*$-isomorphism of C*-algebras:
$$C(\G^{aut}(B\otimes B',\psi\otimes \psi'))/I\cong C(\G^{aut}(B',\psi'))\ast_w \cautd$$
where $I\subset C(\G^{aut}(B\otimes B',\psi\otimes \psi'))$ is the closed two-sided $\ast$-ideal generated by the relations corresponding to the condition $d\otimes \eta_{B'}\eta^*_{B'}\in\End(U)$ and $U$ denotes the fundamental representation of $\G^{aut}(B\otimes B',\psi\otimes \psi')$.
\end{theorem}

\begin{proof}
We fix the notations $M=C(\G^{aut}(B\otimes B',\psi\otimes \psi'))$ and $N=C(\G^{aut}(B',\psi'))\ast_w \cautd$. Let $u\in\mathcal{L}(B')\ot C(\G^{aut}(B',\psi'))$ be the fundamental representation of $\G^{aut}(B',\psi')$. Choose a complete set of representative of irreducible representations $u_n\in\mathcal{L}(H_n)\ot C(\G^{aut}(B',\psi'))$ with $u_0=1$ and $u\simeq u_0\oplus u_1$. Define the unitary representation $v\in\mathcal{L}(B\ot B')\ot N$ of $\G^{aut}(B',\psi')\wr_*\autd$ by $v=a(u)=(\id_{B}\ot\eta_{B'}\ot 1_N)a(u_0)(\id_B\ot\eta_{B'}^*\ot 1_N)+(\id_{B}\ot S_1\ot 1_N)a(u_1)(\id_{B}\ot S_1^*\ot 1_N)$, where $S_1\in\Hom(u_1,u)$ is the unique isometry, up to $\mathbb{S}^1$, such that $\eta_{B'}\eta_{B'}^*$ and $S_1S_1^*$ are two orthogonal projections such that $\eta_{B'}\eta_{B'}^*+S_1S_1^*=\id_{B'}$.

We claim that there exists a unital $*$-homomorphism $\Psi\,:\, M\rightarrow N$ such that $(\id\otimes\Psi)(U)=v$. By the universal property of the C*-algebra $M$, it suffices to check the following conditions.
\begin{enumerate}
\item $\eta_{B\otimes B'}\in\Hom(1,v)$.
\item $m_{B\otimes B'}\in\Hom(v^{\otimes 2},v)$.
\end{enumerate}
Let us prove $(1)$. Since $\eta_{B'}\eta_{B'}^*$ and $S_1S_1^*$ are orthogonal we have $S_1^*\eta_{B'}=0$; it follows that $(\id_{B}\ot S_1^*)\eta_{B\ot B'}=0$. Hence, $v(\eta_{B\ot B'}\ot 1_N)=(id_{B}\ot\eta_{B'}\ot 1_N)a(u_0)\left((\id_B\ot\eta_{B'}^*)\circ\eta_{B\ot B'})\ot 1_N\right)$. Since  $(id_{B}\ot\eta_{B'}^*)\circ\eta_{B\ot B'}=\eta_B\in\Hom(1,a(u_0))$ and $(id_{B}\ot\eta_{B'})\circ\eta_B=\eta_{B\ot B'}$ we find
$$v(\eta_{B\ot B'}\ot 1_N)=(id_{B}\ot\eta_{B'}\ot 1_N)a(u_0)(\eta_B\ot 1_N)=\left((id_{B}\ot\eta_{B'})\circ\eta_B)\ot 1_N\right)=\eta_{B\ot B'}\ot 1_N.$$
This proves $(1)$. Let us prove $(2)$. Observe that $m_{B\ot B'}=(m_B\ot m_{B'})\Sigma_{23}=a(m_{B'})$, where $m_{B'}\in\Hom(u\ot u,u)$. It follows from Proposition \ref{PropFunctor} that $m_{B\ot B'}\in\Hom(v\ot v,v)$.

Let $\pi:M\longrightarrow M/I$ be the canonical quotient map. If we apply assertion (2) of Proposition \ref{PropFunctor} we have that $d\ot \eta_{B'}\eta_{B'}^*=(\id_B\ot\eta_{B'})d(\id_B\ot\eta_{B'}^*)\in\End(v)$ because $\eta_{B'}\in\Hom(1,u)$ and $d\in\End(a(1_{\G}))$. It follows that the map $\Psi$ can be factorized through $M/I$. This means that there exists a unique map $\widetilde{\Psi}:M/I\longrightarrow N$ such that $v=(\idbb\ot \Psi)(U)=(\idbb\ot \widetilde{\Psi})(\idbb\ot \pi)(U)$.

In order to construct the inverse homomorphism we first introduce some notations. All the following definitions are given, unless otherwise stated, for $k\in\N, k\geq 1$. Consider the unitary operator  $\Sigma_k:(B\ot B')^{\ot k}\longrightarrow B^{\ot k}\ot B'^{\ot k}$, $\bigotimes_{i=1}^k(b_i\ot b'_i)\mapsto \bigotimes_{i=1}^k b_i\ot \bigotimes_{i=1}^k b'_i$, where $b_i\in B$ and $b_i'\in B'$. We observe that, with this new notation $\Sigma_2=\Sigma_{23}$.

Let $m_B^{(k)}:B^{\ot k}\longrightarrow B$ be the map which multiplies $k$ elements of $B$; we set $m_B^{(1)}=\id_B$ by convention. We observe that this map is unique and well defined by the associativity of the multiplication. In particular, we have $m_B^{(2)}=m_B$. We claim that, for any $k\geq 2$, $m_B^{(k)}(m_B^{(k)})^*=\delta^{k-1}\id_B$. The proof is by induction. If $k=1$, it is trivially true; if $k=2$, it is clear that $m_Bm_B^*=\delta \id_B$. Let us suppose the result true for $k=l$, i.e. $m_B^{(l)}(m_B^{(l)})^*=\delta^{l-1}\id_B$. We have that $m_B^{(l+1)}=m_B^{(l)}(m_B\ot \id_B^{\ot l-1})$ by associativity. Hence, $m_B^{(l+1)}(m_B^{(l+1)})^*=m_B^{(l)}(m_Bm_B^*\ot \id_B^{\ot l-1})(m_B^{(l)})^*=\delta m_B^{(l)}(m_B^{(l)})^*=\delta^{l}\id_B$. Then, the equality is true for $k=l+1$. This completes the proof.

Define the map $T_k=(m_B^{(k)}\ot \id_{B'}^{\ot k})\Sigma_k\in\mathcal{L}((B\ot B')^{\ot k}, B\ot B'^{\ot k})$.
Let $S_k\in\Hom(u_k,u^{\ot k})$ be the unique isometry, up to $\mathbb{S}^1$, and define the isometry $Q_k=\delta^{-\frac{k-1}{2}}T_k^*\circ (\id_B\ot S_k)\in\mathcal{L}(B\ot H_k,(B\ot B')^{\ot k})$ for $k\geq 1$.
For $k=0$, we define the isometry $Q_0=\id_B\ot\eta_{B'}\in\Hom(a(u_0),v)$.

Finally, for $k\geq 1$, consider the elements $A_k=(Q_k^*\ot 1_M)U^{\ot k}(Q_k\ot 1_M)\in\mathcal{L}(B\ot H_k)\ot M$ and, for $k=0$, $A_0=(Q_0^*\ot 1_M)U(Q_0\ot 1_M)\in\mathcal{L}(B)\ot M$. Denote by $\widetilde{A}_k=(\id\ot\pi)(A_k)\in\mathcal{L}(B\ot H_k)\ot M/I$ the projections of the $A_k,k\in\N$ on the quotient. Let $V=(\id\ot\pi)(U)\in\mathcal{L}(B\ot B')\ot M/I$ be the projection of the fundamental representation $U$.

We claim that there exists a unital $\ast$-homomorphism $\Phi:N\longrightarrow M/I$ such that $(\id\ot\Phi)(a(u_k))=\widetilde{A}_k$ for all $k\in\N$. By the universal property of the C*-algebra $N$ it suffices to check the following.
\begin{enumerate}
\item $A_0(\eta_B\ot 1_M)=\eta_B\ot 1_M$.
\item $\widetilde{A}_k\in\mathcal{L}(B\ot H_k)\ot M/I$ is unitary for all $k\geq 0$.
\item For all $k,l,t\in\N$ and $R\in\Hom(u_k\ot u_l,u_t)$, $(m_B\ot R)\Sigma_{23}\in\Hom(\widetilde{A}_k\ot \widetilde{A}_l,\widetilde{A}_t)$.
\item $d\in\End(\widetilde{A}_0)$.
\end{enumerate}

Let us prove $(1)$. Since $\eta_B\ot\eta_{B'}=\eta_{B\ot B'}\in\Hom(1,U)$, we have
\begin{eqnarray*}
A_0(\eta_B\ot 1_M)&=&(\id_B\ot\eta_{B'}^*\ot 1_M)U(\id_B\ot\eta_{B'}\ot 1_M)(\eta_B\ot 1_M)\\&=&(\id_B\ot\eta_{B'}^*\ot 1_M)U(\eta_B\ot\eta_{B'}\ot 1_M)=\eta_B\ot\eta_{B'}^*\eta_{B'}\ot 1_M=\eta_B\ot 1_M.
\end{eqnarray*}
Let us prove $(2)$. We have $A_kA_k^*=(Q_k^*\ot 1_M)U^{\ot k}(Q_kQ_k^*\ot 1_M)(U^{\ot k})^*(Q_k\ot 1_M)$. Moreover, $Q_0Q_0^*=\id_B\ot\eta_{B'}\eta_{B'}^*$ and $Q_1Q_1^*=\id_B\ot S_1S_1^*=\id_{B}\ot(\id_{B'}-\eta_{B'}\eta_{B'}^*)=\id_{B\ot B'}-Q_0Q_0^*$. By definition of $I$, we have $V(Q_0Q_0^*\ot 1_{M/I})=(Q_0Q_0^*\ot 1_{M/I})V$ and,
$$
V(Q_1Q_1^*\ot 1_{M/I})=V(1-Q_0Q_0^*\ot 1_{M/I})=(1-Q_0Q_0^*\ot 1_{M/I})V=(Q_1Q_1^*\ot 1_{M/I})V.
$$
It follows that, for $k=0,1$,
\begin{eqnarray*}
\widetilde{A}_k\widetilde{A}_k^*&=&(Q_k^*\ot 1_{M/I})V(Q_kQ_k^*\ot 1_{M/I})V^*(Q_k\ot 1_{M/I})
=(Q_k^*Q_kQ_k^*\ot 1_{M/I})VV^*(Q_k\ot 1_{M/I})\\
&=&(Q_k^*Q_kQ_k^*Q_k\ot 1_{M/I})=\id_{B\ot H_k}\ot 1_{M/I}.
\end{eqnarray*}

\noindent
The proof of $\widetilde{A}_k^*\widetilde{A}_k=\id_{B\ot H_k}\ot 1_{M/I}$ when $k=0,1$ is the same.

In order to prove $(2)$ when $k\geq 2$ and to check $(3)$ we need the following lemma.

\begin{lemma}\label{lemphi}
For all $k,l\in\N,\, k,l\geq 1$ and $T\in\mathcal{L}(B'^{\ot k},B'^{\ot l})$ we define the map
$$\phi_{k,l}(T)=\Sigma_l^* ((m_B^{(l)})^*m_B^{(k)}\ot T)\Sigma_k\in\mathcal{L}((B\ot B')^{\ot k},(B\ot B')^{\ot l})$$
Define $\phi_k:=\phi_{k,k}(\id_{B'}^{\ot k})$. The following holds.
\begin{itemize}
\item[(i)] $\phi_k\in\End(V^{\ot k})$.
\item[(ii)] For all $T\in\mathcal{L}(B'^{\ot k},B'^{\ot l})$ and $S\in\mathcal{L}(B'^{\ot l},B'^{\ot t})$ we have:
\begin{itemize}
\item[$\bullet$] $\phi_{l,t}(S)\phi_{k,l}(T)=\delta^{l-1}\phi_{k,t}(S\circ T)$,
\item[$\bullet$] $\phi_{k,l}(T)^*=\phi_{l,k}(T^*)$.
\end{itemize}
\item[(iii)] $\phi_{k,k-1}(\id_{B'}^{\ot s}\ot m_{B'}\ot \id_{B'}^{\ot s'})=\phi_{k-1}\circ(\idbb^{\ot s} \ot m_{B\ot B'}\ot\idbb^{\ot s'})\in\Hom(V^{\ot k},V^{\ot k-1})$ for all $k\geq 2$ and all $s,s'\geq 0$ such that $s+s'+2=k$.
\item[(iv)] \begin{itemize}
\item[$\bullet$] $\phi_{k,k-1}(\id_{B'}^{\ot s}\ot \eta^*_{B'}\ot \id_{B'}^{\ot s'})=\phi_{k-1}\circ(\idbb^{\ot s-1}\ot (m_B\ot id_{B'}\ot \eta^*_{B'})\Sigma_{23}\ot \idbb^{\ot s'})\in\Hom(V^{\ot k},V^{\ot k-1})$ for all $k\geq 2$ and all $s\geq 1,s'\geq 0$ such that $s+s'+1=k$,
\item[$\bullet$] $\phi_{k,k-1}( \eta^*_{B'}\ot \id_{B'}^{\ot k-1})=\phi_{k-1}\circ((m_B\ot \eta^*_{B'}\ot id_{B'})\Sigma_{23}\ot \idbb^{\ot k-2})\in\Hom(V^{\ot k},V^{\ot k-1})$ for all $k\geq 2$.
\end{itemize}
\item[(v)] $\phi_{k,l}(T)\in\Hom(V^{\ot k},V^{\ot l})$ for all $T\in\Hom(u^{\ot k},u^{\ot l})$.
\end{itemize}
\end{lemma}
\begin{proof}
$(i)$. If $k=1$, $\phi_1=\idbb\in\End(V)$. When $k\geq 2$ we prove the result by induction on $k$. If $k=2$, we have $\phi_2=\Sigma_{23}(m_B^*m_B\ot \id_{B'}^{\ot 2})\Sigma_{23}$. We want to prove that $\phi_2\in\End(V^{\ot 2})$. Define $L=\Sigma_3^*((m_B^{(3)})^*m_B\ot \id_{B'}\ot \eta_{B'}\ot \id_{B'})\Sigma_2$.
We claim that 
\begin{equation}\label{relclaim1}
L=(\idbb\ot m_{B\ot B'}\ot \idbb)(\Sigma_2^*(m_B^*\ot \id_{B'}\ot \eta_{B'})\ot \Sigma_2^*(m_B^*\ot \eta_{B'}\ot \id_{B'})).
\end{equation}
Let us evaluate the two maps on the element $b_1\ot b_1'\ot b_2\ot b_2'\in(B\ot B')^{\ot 2}$. The left hand side is:
\begin{eqnarray*}
\Sigma_3^*((m_B^{(3)})^*m_B\ot \id_{B'}\ot \eta_{B'}\ot \id_{B'})\Sigma_2(b_1\ot b_1'\ot b_2\ot b_2')
&=&\Sigma_3^*((m_B^{(3)})^*m_B(b_1\ot b_2)\ot b_1'\ot 1_{B'}\ot b_2'\\&=&
\Sigma b_{(1)}^{12}\ot b_1'\ot b_{(2)}^{12} \ot 1_{B'}\ot b_{(3)}^{12}\ot b_2'.
\end{eqnarray*}

where we used the notation $(m_B^{(3)})^*m_B(b_1\ot b_2)=\Sigma b_{(1)}^{12}\ot b_{(2)}^{12} \ot b_{(3)}^{12}$. The right hand side is:
\begin{eqnarray*}
&&(\idbb\ot m_{B\ot B'}\ot \idbb)(\Sigma_2^*(m_B^*\ot \id_{B'}\ot \eta_{B'})\ot \Sigma_2^*(m_B^*\ot \eta_{B'}\ot \id_{B'}))(b_1\ot b_1'\ot b_2\ot b_2')\\
&=&(\idbb\ot m_{B\ot B'}\ot \idbb)(\Sigma_2^*(m_B^*(b_1)\ot b_1'\ot 1_{B'})\ot \Sigma_2^*(m_B^*(b_2)\ot 1_{B'}\ot b_2'))\\
&=&(\idbb\ot m_{B\ot B'}\ot \idbb)(\Sigma b_{(1)}^1\ot b_1'\ot b_{(2)}^1\ot 1_{B'}\ot b_{(1)}^2\ot 1_{B'}\ot b_{(2)}^2\ot b_2')\\
&=&(\Sigma b_{(1)}^1\ot b_1'\ot m_B(b_{(2)}^1\ot b_{(1)}^2)\ot 1_{B'}\ot b_{(2)}^2\ot b_2').
\end{eqnarray*}
where we used the notation $m_B^*(b_i)=\Sigma b_{(1)}^i\ot b_{(2)}^i$ for $i=1,2$. Note that the two elements of $(B\ot B')^{\ot 3}$ we obtained are equal if and only if $(m_B^{(3)})^*m_B=(\id_B \ot m_B\ot \id_B)(m_B^*\ot m_B^*)$. This can be verified by drawing the noncrossing partitions associated to the different maps and by using the compatibility with respect to the multiplication proved in Proposition \ref{propmap}.
In both cases, the noncrossing partition obtained after the composition is

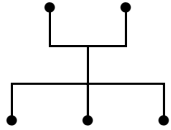
{\centering
 \begin{tikzpicture}[thick,font=\small]
    \path (0.5,-.5) node(a){$\bullet$}
          (1.5,-.5) node(c){$\bullet$} 
          (0,-2) node(d){$\bullet$} 
          (1,-2) node(e){$\bullet$} 
          (2,-2) node(f){$\bullet$};

	    \draw (0.5,-.5) -- +(0,-0.5) -| (1.5,-.5);
	    \draw (0,-2) -- +(0,+0.5) -| (1,-2);
	    \draw (2,-2) -- +(0,+0.5) -| (1,-2);        
	    \draw (1,-1) -- (1,-1.5);

 \end{tikzpicture}
\par}

It follows that relation \ref{relclaim1} is verified.

Define $T=(m_B\ot\eta_{B'}^*\ot\id_{B'})\Sigma_{23}$ and note that $m_{B'}\circ(\eta_{B'}\eta_{B'}^*\ot\id_{B'})=\eta_{B'}^*\ot\id_{B'}$. Hence,
\begin{eqnarray*}
T&=&\left[m_B\ot(m_{B'}\circ(\eta_{B'}\eta_{B'}^*\ot\id_{B'}))\right]\Sigma_{23}=\left[(m_B\ot m_{B'})\circ(\id_B\ot\id_B\ot\eta_{B'}\eta_{B'}^*\ot\id_{B'}))\right]\Sigma_{23}\\
&=&(m_B\ot m_{B'})\Sigma_{23}(\id_B\ot\eta_{B'}\eta_{B'}^*\ot\id_{B}\ot\id_{B'})=m_{B\ot B'}\circ((\id_B\ot\eta_{B'}\eta_{B'}^*)\ot\id_{B\ot B'}).
\end{eqnarray*}
By definition of $I$ we have $\id_B\ot\eta_{B'}\eta_{B'}^*\in\End(V)$ and since $m_{B\ot B'}\in\Hom(U^{\ot 2},U)$ we deduce that $T\in\End(V^{\ot 2},V)$. Similarly, if we define $Z=(m_B\ot\id_{B'}\ot\eta_{B'}^*)\Sigma_{23}$ then, by using the relation $\id_{B'}\ot\eta_{B'}^*=m_{B'}\circ(\id_{B'}\ot\eta_{B'}\eta_{B'}^*)$, we find $Z=m_{B\ot B'}\circ(\id_{B\ot B'}\ot(\id_B\ot\eta_{B'}\eta_{B'}^*))$. Hence $Z\in\Hom(V^{\ot 2},V)$. Finally, note that $L=(\idbb\ot m_{B\ot B'}\ot \idbb)(Z^*\ot T^*)$. Hence $L\in\Hom(V^{\ot 2},V^{\ot 3})$ and since we have
\begin{eqnarray*}
L^*L&=&
\Sigma_2^*(m_B^*m_B^{(3)}\ot \id_{B'}\ot \eta_{B'}^*\ot \id_{B'})\Sigma_3\Sigma_3^*((m_B^{(3)})^*m_B\ot \id_{B'}\ot \eta_{B'}^*\ot \id_{B'})\Sigma_2\\ &=&
\Sigma_2^*(m_B^*m_B^{(3)}(m_B^{(3)})^*m_B\ot \id_{B'}^{\ot 2})\Sigma_2=\delta^2\phi_2,
\end{eqnarray*}
where we used that $m_B^{(3)}(m_B^{(3)})^*=\delta^{2}\id_B$, it follows that $\phi_2\in\End(V^{\ot 2})$.

Now, let us prove that, if $\phi_{k-1}\in\End(V^{\ot k-1})$, then $\phi_{k}\in\End(V^{\ot k})$.  We first claim that, for any $k\geq 2$, the following holds
\begin{equation}\label{decphik}
\phi_k=(\idbb^{\ot k-2}\ot \phi_2)(\phi_{k-1}\ot \idbb).
\end{equation} 
Let us evaluate this map on the general element $\bigotimes_{i=1}^k (b_i\ot b_i')\in(B\ot B')^{\ot k}$. We have:
\begin{eqnarray*}
&&(\idbb^{\ot k-2}\ot \phi_2)(\phi_{k-1}\ot \idbb)(\bigotimes_{i=1}^k (b_i\ot b_i'))\\
&=&(\idbb^{\ot k-2}\ot \phi_2)[(\Sigma_{k-1}^*((m_B^{(k-1)})^*m_B^{(k-1)}\ot\id_{B'}^{\ot k-1})\Sigma_{k-1})\ot \idbb](\bigotimes_{i=1}^k (b_i\ot b_i'))\\
&=&(\idbb^{\ot k-2}\ot \Sigma_2^*(m_B^*m_B\ot\id_{B'}^{\ot 2})\Sigma_{2})(\Sigma_{k-1}^*(\Sigma\bigotimes_{i=1}^{k-1}b_{(i)}^{C_1}\ot \bigotimes_{i=1}^{k-1}b'_i)\ot b_k \ot b_k')\\
&=&\Sigma \bigotimes_{i=1}^{k-2}(b_{(i)}^{C_1}\ot b_i')\ot \Sigma_2^*(m_B^*m_B(b_{(k-1)}^{C_1}\ot b_k)\ot b_{k-1}'\ot b_k')=\Sigma \bigotimes_{i=1}^{k-2}(b_{(i)}^{C_1}\ot b_i')\ot b_{(1)}^{C_2}\ot b_{k-1}'\ot  b_{(2)}^{C_2}\ot b_{k}',
\end{eqnarray*}
where we used the notations $(m_B^{(k-1)})^*m_B^{(k-1)}(\bigotimes_{i=1}^{k-1} b_i)=\Sigma \bigotimes_{i=1}^{k-1} b_{(i)}^{C_1}$ and $m_B^*m_B(b_{(k-1)}^{C_1}\ot b_k)=\Sigma  b_{(1)}^{C_2}\ot b_{(2)}^{C_2}$. When we evaluate $\phi_k$ according to its definition, we get
$$
\Sigma_{k}^*((m_B^{(k)})^*m_B^{(k)}\ot\id_{B'}^{\ot k})\Sigma_{k}(\bigotimes_{i=1}^k (b_i\ot b_i'))=\Sigma \bigotimes_{i=1}^k (b_{(i)}^{C_3}\ot b_i')
$$
where we used the notation $(m_B^{(k)})^*m_B^{(k)}(\bigotimes_{i=1}^{k} b_i)=\Sigma \bigotimes_{i=1}^{k} b_{(i)}^{C_3}$. The two elements obtained are equal if and only if 
$$(m_B^{(k)})^*m_B^{(k)}=(\id_B^{\ot k-2}\ot m_B^*m_B)((m_B^{(k-1)})^*m_B^{(k-1)}\ot \id_B).$$

This formula can be verified by considering the noncrossing partitions associated to the different maps and by applying Proposition \ref{propmap}. We have

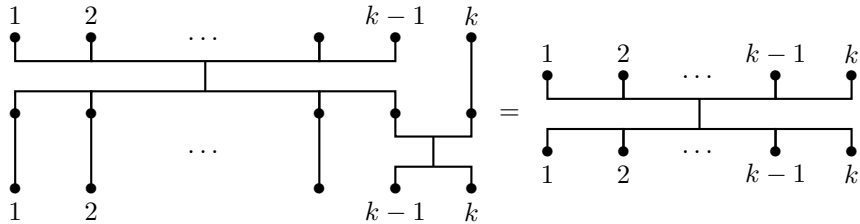
{\centering
 \begin{tikzpicture}[thick,font=\small]
    \path (1,0) node(a){$\bullet$}
    		  (1,.3) node(a){$1$}
          (2,0) node(c){$\bullet$} 
          (2,.3) node(a){$2$}
          (3.5,0) node(a){$\dots$}
          (5,0) node(c){$\bullet$}
          (6,0) node(c){$\bullet$}
          (7,0) node(c){$\bullet$}
          (6,.3) node(a){$k-1$}
          (7,.3) node(a){$k$}
          
          (1,-2.3) node(a){$1$}
          (2,-2.3) node(a){$2$}
          (6,-2.3) node(a){$k-1$}
          (7,-2.3) node(a){$k$}

		  (1,-1) node(a){$\bullet$}
          (2,-1) node(c){$\bullet$} 
         (3.5,-1.5) node(a){$\dots$}
          (5,-1) node(c){$\bullet$}
          (6,-1) node(c){$\bullet$}
          (7,-1) node(c){$\bullet$}          
          
          (1,-2) node(a){$\bullet$}
          (2,-2) node(c){$\bullet$} 
         % (2,-1) node(a){$\bullet$}
          (5,-2) node(c){$\bullet$}
          (6,-2) node(c){$\bullet$}
          (7,-2) node(c){$\bullet$}
          
          (7.5,-1) node(a){$=$}
          
          (8,-.5) node(a){$\bullet$}
          (9,-.5) node(c){$\bullet$} 
          (10,-.5) node(a){$\dots$}
          (11,-.5) node(c){$\bullet$}
          (12,-.5) node(c){$\bullet$}
          
          (8,-.2) node(a){$1$}
          (9,-.2) node(a){$2$}
          (11,-.2) node(a){$k-1$}
          (12,-.2) node(a){$k$}
          
          (8,-1.8) node(a){$1$}
          (9,-1.8) node(a){$2$}
          (11,-1.8) node(a){$k-1$}
          (12,-1.8) node(a){$k$}
          
          (8,-1.5) node(a){$\bullet$}
          (9,-1.5) node(c){$\bullet$} 
          (10,-1.5) node(a){$\dots$}
          (11,-1.5) node(c){$\bullet$}
          (12,-1.5) node(c){$\bullet$};

	    \draw (1,0) -- +(0,-0.3) -| (2,0);
	    \draw (5,0) -- +(0,-0.3) -| (2,0);
	    \draw (5,0) -- +(0,-0.3) -| (6,0);
	    
	    \draw (1,-1) -- +(0,0.3) -| (2,-1);
	    \draw (5,-1) -- +(0,0.3) -| (2,-1);
	    \draw (5,-1) -- +(0,0.3) -| (6,-1);
	    
	    \draw (6,-1) -- +(0,-0.3) -| (7,-1);
	    \draw (6,-2) -- +(0,0.3) -| (7,-2);
	           
	    \draw (7,0) -- (7,-1);
	    \draw (2,-1) -- (2,-2);
	    \draw (1,-1) -- (1,-2);
	    \draw (5,-1) -- (5,-2);
	    \draw (3.5,-.3) -- (3.5,-.7);
	    \draw (6.5,-1.3) -- (6.5,-1.7);

	    \draw (8,-.5) -- +(0,-0.3) -| (9,-.5);
	    \draw (11,-.5) -- +(0,-0.3) -| (9,-.5);
	    \draw (12,-.5) -- +(0,-0.3) -| (11,-.5);
	    
	    \draw (8,-1.5) -- +(0,0.3) -| (9,-1.5);
	    \draw (11,-1.5) -- +(0,0.3) -| (9,-1.5);
	    \draw (12,-1.5) -- +(0,0.3) -| (11,-1.5);
	    
	    \draw (10,-.8) -- (10,-1.2);
	    
 \end{tikzpicture}
\par}
\noindent
This completes the proof of relation \ref{decphik} and finishes the proof of $(1)$ since it shows that the map $\phi_k$ can be obtained through tensor products and compositions of the maps $\idbb, \phi_2, \phi_{k-1}$ which are in $\End(V^{\ot k})$, by the induction hypothesis.

$(ii)$. The composition formula can be checked as follows.
\begin{eqnarray*}
\phi_{l,t}(S)\phi_{k,l}(T) &=&
\Sigma_t^* ((m_B^{(t)})^*m_B^{(l)}\ot S)\Sigma_l\Sigma_l^* ((m_B^{(l)})^*m_B^{(k)}\ot T)\Sigma_k\\&=&
\Sigma_t^* ((m_B^{(t)})^*m_B^{(l)}(m_B^{(l)})^*m_B^{(k)}\ot ST)\Sigma_k\\&=&
\delta^{l-1}\Sigma_t^* ((m_B^{(t)})^*m_B^{(k)}\ot ST)\Sigma_k=
\delta^{l-1}\phi_{k,t}(S\circ T),
\end{eqnarray*}
where we used the relation $m_B^{(l)}(m_B^{(l)})^*=\delta^{l-1}\id_B$. The second formula follows from:
$$\phi_{k,l}(T)^* =(\Sigma_l^* ((m_B^{(l)})^*m_B^{(k)}\ot T)\Sigma_k)^*=\Sigma_k^* ((m_B^{(k)})^*m_B^{(l)}\ot T^*)\Sigma_l=\phi_{l,k}(T^*).$$

$(iii)$. The element $\phi_{k-1}\circ(\idbb^{\ot s} \ot m_{B\ot B'}\ot\idbb^{\ot s'})$ is equal to:
\begin{eqnarray*}
&&\Sigma_{k-1}^* ((m_B^{(k-1)})^*m_B^{(k-1)}\ot \id_{B'}^{\ot k-1})\Sigma_{k-1}(\idbb^{\ot s}\ot m_{B\ot B'}\ot \idbb^{\ot s'})\\
&=&\Sigma_{k-1}^* ((m_B^{(k-1)})^*m_B^{(k-1)}\ot \id_{B'}^{\ot k-1})(\id_B^{\ot s}\ot m_B \ot \id_B^{\ot s'}\ot\id_{B'}^{\ot s}\ot m_{B'}\ot \id_{B'}^{\ot s'})\Sigma_k\\
&=&\Sigma_{k-1}^* ((m_B^{(k-1)})^*m_B^{(k)}\ot \id_{B'}^{\ot s}\ot m_{B'}\ot \id_{B'}^{\ot s'})\Sigma_{k}\\
&=&\phi_{k,k-1}(\id_{B'}^{\ot s}\ot m_{B'}\ot \id_{B'}^{\ot s'}),
\end{eqnarray*}
where the first equality follows from
\begin{eqnarray*}
&&\Sigma_{k-1}(\idbb^{\ot s}\ot m_{B\ot B'}\ot \idbb^{\ot s'})(\bigotimes_{i=1}^k(b_i\ot b'_i))\\
&=&\Sigma_{k-1}(\bigotimes_{i=1}^s(b_i\ot b'_i)\ot m_B(b_{s+1}\ot b_{s+2})\ot m_{B'}(b'_{s+1}\ot b'_{s+2})\ot \bigotimes_{i=s+3}^k(b_i\ot b'_i))\\
&=&\bigotimes_{i=1}^sb_i\ot m_B(b_{s+1}\ot b_{s+2})\ot \bigotimes_{i=s+3}^k b_i\ot
\bigotimes_{i=1}^sb'_i\ot m_{B'}(b'_{s+1}\ot b'_{s+2})\ot \bigotimes_{i=s+3}^k b'_i\\
&=&(\id_B^{\ot s}\ot m_B \ot \id_B^{\ot s'}\ot\id_{B'}^{\ot s}\ot m_{B'}\ot \id_{B'}^{\ot s'})\Sigma_k(\bigotimes_{i=1}^k(b_i\ot b'_i)),
\end{eqnarray*}
and for the second equality we used the relation $m_B^{(k-1)}(\id_B^{\ot s}\ot m_B \ot \id_B^{\ot s'})=m_B^{(k)}$ which is a consequence of associativity. Now, since $m_{B\ot B'}\in\Hom(U^{\ot 2},U)$ and $\phi_{k-1}\in\End(V^{\ot k-1})$ by the assertion $(i)$ of this Lemma, it follows that $\phi_{k,k-1}(\id_{B'}^{\ot s}\ot m_{B'}\ot \id_{B'}^{\ot s'})\in\Hom(V^{\ot k},V^{\ot k-1})$.

$(iv)$. We have
\begin{eqnarray*}
&&\phi_{k-1}\circ(\idbb^{\ot s-1} \ot (m_B\ot \id_{B'}\ot \eta_{B'}^*)\Sigma_{23}\ot\idbb^{\ot s'})\\
&=&\Sigma_{k-1}^* ((m_B^{(k-1)})^*m_B^{(k-1)}\ot \id_{B'}^{\ot k-1})\Sigma_{k-1}(\idbb^{\ot s-1} \ot (m_B\ot \id_{B'}\ot \eta_{B'}^*)\Sigma_{23}\ot\idbb^{\ot s'})\\
&=&\Sigma_{k-1}^* ((m_B^{(k-1)})^*m_B^{(k-1)}\ot \id_{B'}^{\ot k-1})(\id_B^{\ot s-1}\ot m_B \ot \id_B^{\ot s'}\ot\id_{B'}^{\ot s}\ot \eta^*_{B'}\ot \id_{B'}^{\ot s'})\Sigma_k\\
&=&\Sigma_{k-1}^* ((m_B^{(k-1)})^*m_B^{(k)}\ot \id_{B'}^{\ot s}\ot \eta^*_{B'}\ot \id_{B'}^{\ot s'})\Sigma_{k}=\phi_{k,k-1}(\id_{B'}^{\ot s}\ot \eta^*_{B'}\ot \id_{B'}^{\ot s'}).
\end{eqnarray*}
As in the proof of assertion $(iii)$ the second equality can be checked by evaluating the two maps on an element of $(B\ot B')^{\ot k}$ and the third follows from the associativity of $m_B$.\\
If $s=0$ and $s'=k-1$ the formula is slightly different but the computations are analogous. We have
\begin{eqnarray*}
&&\phi_{k-1}\circ((m_B\ot \eta_{B'}^*\ot \id_{B'})\Sigma_{23}\ot\idbb^{\ot k-2})\\
&=&\Sigma_{k-1}^* ((m_B^{(k-1)})^*m_B^{(k-1)}\ot \id_{B'}^{\ot k-1})\Sigma_{k-1}((m_B\ot \eta_{B'}^*\ot \id_{B'})\Sigma_{23}\ot\idbb^{\ot k-2})\\
&=&\Sigma_{k-1}^* ((m_B^{(k-1)})^*m_B^{(k-1)}\ot \id_{B'}^{\ot k-1})(m_B \ot \id_B^{\ot k-2}\ot \eta^*_{B'}\ot \id_{B'}^{\ot k-1})\Sigma_k\\
&=&\Sigma_{k-1}^* ((m_B^{(k-1)})^*m_B^{(k)}\ot \eta^*_{B'}\ot \id_{B'}^{\ot k-1})\Sigma_{k}=\phi_{k,k-1}( \eta^*_{B'}\ot \id_{B'}^{\ot k-1}).
\end{eqnarray*}
Since in the proof of assertion $(i)$ we showed that $(m_B\ot \id_{B'}\ot \eta_{B'}^*)\Sigma_{23}$ and $(m_B\ot \eta_{B'}^*\ot \id_{B'})\Sigma_{23}$ are in $\Hom(V^{\ot 2},V)$ and $\phi_{k-1}\in\End(V^{\ot k-1})$ by $(i)$, we conclude that $\phi_{k,k-1}(\id_{B'}^{\ot s}\ot \eta^*_{B'}\ot \id_{B'}^{\ot s'})$ is in $\Hom(V^{\ot k},V^{\ot k-1})$.

$(v)$. Thanks to Theorem \ref{intertautb} we know that the morphisms associated to the noncrossing partitions in $NC(k,l)$ form a linear basis of $\Hom(u^{\ot k},u^{\ot l})$. Moreover, every morphism of such a basis can be seen as the composition of the morphisms $\id_{B'}^{\ot s}\ot m_{B'}\ot \id_{B'}^{\ot s'}\in\Hom(u^{\ot s+s'+2},u^{\ot s+s'+1})$ and $\id_{B'}^{\ot s}\ot \eta_{B'}\ot \id_{B'}^{\ot s'}\in\Hom(u^{\ot s+s'},u^{\ot s+s'+1})$ and of their adjoints. This fact, together with the assertions $(ii),(iii)$ and $(iv)$ of Lemma \ref{lemphi}, implies that $\phi_{k,l}(T)\in\Hom(V^{\ot k},V^{\ot l})$ for all intertwiner $T\in\Hom(u^{\ot k},u^{\ot l})$.
\end{proof}

Now, we go back to the proof of $(2)$, when $k\geq 2$. We have:
$$
Q_kQ_k^* =\delta^{-(k-1)}T_k^*\circ (\id_B\ot S_kS_k^*)\circ T_k
=\delta^{-(k-1)}\Sigma_k^* ((m_B^{(k)})^*m_B^{(k)}\ot S_kS_k^*)\Sigma_k
=\delta^{-(k-1)}\phi_{k,k}(S_kS_k^*).
$$
Since $S_kS_k^*\in\End(u^{\ot k})$; it follows that $Q_kQ_k^*\in\End(V^{\ot k})$ by Lemma \ref{lemphi} $(v)$. It is now easy to check that $\widetilde{A}_k$ is unitary.

Let us check $(3)$. We have to prove that, for all $k,l,t\in\N$ and $R\in\Hom(u_k\ot u_l,u_t)$, $(m_B\ot R)\Sigma_{23}$ is in $\Hom(\widetilde{A}_k\ot \widetilde{A}_l,\widetilde{A}_t)$. Since the $Q_s$ are isometries and $Q_sQ_s^*\in\End(V^{\ot s})$, we have the following sequence of equivalent conditions.

$\begin{array}{l}
((m_B\ot R)\Sigma_{23}\uq)(Q_k^*\ot Q_l^*\uq)V^{\ot k+l}(Q_k\ot Q_l\uq)=\\ \hspace{7.5cm}(Q_t^*\uq)V^{\ot t}(Q_t\uq)((m_B\ot R)\Sigma_{23}\uq)
\vspace{.5cm}\end{array}$

$\begin{array}{l}
(Q_t\uq)((m_B\ot R)\Sigma_{23}\uq)(Q_k^*\ot Q_l^*\uq)V^{\ot k+l}(Q_kQ_k^*\ot Q_lQ_l^*\uq)=\\ \hspace{3.8cm} (Q_tQ_t^*\uq)V^{\ot t}(Q_t\uq)((m_B\ot R)\Sigma_{23}\uq)(Q_k^*\ot Q_l^*\uq)\vspace{.5cm}\end{array}$

$\begin{array}{l}
(Q_t\uq)((m_B\ot R)\Sigma_{23}\uq)(Q_k^*Q_kQ_k^*\ot Q_l^*Q_lQ_l^*\uq)V^{\ot k+l}=\\ \hspace{5.5cm} V^{\ot t}(Q_tQ_t^*Q_t\uq)((m_B\ot R)\Sigma_{23}\uq)(Q_k^*\ot Q_l^*\uq)\vspace{.5cm}\end{array}$

$\begin{array}{l}
(Q_t\uq)((m_B\ot R)\Sigma_{23}\uq)(Q_k^*\ot Q_l^*\uq)V^{\ot k+l}=\\ \hspace{6.5cm} V^{\ot t}(Q_t\uq)((m_B\ot R)\Sigma_{23}\uq)(Q_k^*\ot Q_l^*\uq)
\end{array}$

Then, the original condition is equivalent to
$$Q_t(m_B\ot R)\Sigma_{23}(Q_k^*\ot Q_l^*)\in\Hom(V^{\ot k+l},V^{\ot t}).$$
Now, if we replace every $Q_i$ with its definition and write $K:=\delta^{-\frac{k+l+t-3}{2}}$, we get
\begin{eqnarray*}
Q_t(m_B\ot R)\Sigma_{23}(Q_k^*\ot Q_l^*)&=&
K\Sigma_t^*(m_B^{(t)*}\ot \id_{B'}^{\ot t})(\id_B\ot S_t)(m_B\ot R)\Sigma_{23}\circ\ldots\\
&&\ldots\circ[(\id_B\ot S_k^*)(m_B^{(k)}\ot \id_{B'}^{\ot k})\Sigma_k\ot(\id_B\ot S_l^*)(m_B^{(l)}\ot\id_{B'}^{\ot l})\Sigma_l]\\
&=&K\Sigma_t^*(m_B^{(t)*}\ot S_t)(m_B\ot R)\Sigma_{23}(m_B^{(k)}\ot m_B^{(l)}\ot S_k^*\ot S_l^*)(\Sigma_k\ot\Sigma_l)\\
&=&K\Sigma_t^*(m_B^{(t)*}\ot S_t)(m_B\ot R)(m_B^{(k)}\ot m_B^{(l)}\ot S_k^*\ot S_l^*)\widetilde{\Sigma_{23}}(\Sigma_k\ot\Sigma_l)\\
&=&K\Sigma_t^*(m_B^{(t)*}m_B^{(k+l)}\ot S_tR(S_k^*\ot S_l^*))\widetilde{\Sigma_{23}}(\Sigma_k\ot\Sigma_l)\\
&=&K\Sigma_t^*(m_B^{(t)*}m_B^{(k+l)}\ot S_tR(S_k^*\ot S_l^*))\Sigma_{k+l}=K\phi_{k+l,t}(S_tR(S_k^*\ot S_l^*)),
\end{eqnarray*}
where $\widetilde{\Sigma_{23}}:B^{\ot k}\ot B'^{\ot k}\ot B^{\ot l}\ot B'^{\ot l}\longrightarrow B^{\ot k+l}\ot B'^{\ot k+l}$  is the map that exchanges $B'^{\ot k}$ and $B^{\ot l}$. It is easy to check that $\widetilde{\Sigma_{23}}(\Sigma_k\ot\Sigma_l)=\Sigma_{k+l}$. In the third equality we used that $m_B(m_B^{(k)}\ot m_B^{(l)})=m_B^{(k+l)}$; this is due to the associativity of the multiplication.
Since $S_t,S_k,S_l$ and $R$ are intertwiners of $\G^{aut}(B',\psi')$, we have that $S_tR(S_k^*\ot S_l^*)\in\Hom(u^{\ot k}\ot u^{\ot l},u^{\ot t})$.
We can then apply Lemma \ref{lemphi} $(v)$ and find that $\phi_{k+l,t}(S_tR(S_k^*\ot S_l^*))\in\Hom(V^{\ot k+l},V^{\ot t})$. This completes the proof of $(3)$.

Let us prove $(4)$. By using the same method used in the prove of $(3)$, we have that $d\in\End(\widetilde{A}_0)$ is equivalent to $Q_0dQ_0^*=d\ot \eta_{B'}\eta_{B'}^*\in\End(V)$. This last condition ifollows from the definition of $I$.

To conclude the proof we check that the morphisms $\widetilde{\Psi}$ and $\Phi$ are inverse to each other. We have:
\begin{eqnarray*}
(\id\ot \Phi\widetilde{\Psi})(V)&=&
(\id\ot\Phi)(\id\ot \Psi)(U)=
(\id\ot\Phi)(v)\\&=&(Q_0\uq)\widetilde{A}_0(Q_0^*\uq)+(Q_1\uq)\widetilde{A}_1(Q_1^*\uq)\\ &=&
(Q_0Q_0^*\ot 1_{M/I})V(Q_0Q_0^*\ot 1_{M/I})+(Q_1Q_1^*\ot 1_{M/I})V(Q_1Q_1^*\ot 1_{M/I})\\ &=&
V(Q_0Q_0^*Q_0Q_0^*\ot 1_{M/I})+V(Q_1Q_1^*Q_1Q_1^*\ot 1_{M/I})\\ &=&
V((Q_0Q_0^*+Q_1Q_1^*)\ot 1_{M/I})=V
\end{eqnarray*}
since, for $s=0,1$, $Q_s$ is an isometry such that $Q_sQ_s\in\End(V)$ and $Q_0Q_0^*+Q_1Q_1^*=\idbb$. Similarly,
\begin{eqnarray*}
(\id\ot \widetilde{\Psi})(\id\ot\Phi)(a(u_k)) &=&
(\id\ot \widetilde{\Psi})(\widetilde{A}_k)=
(\id\ot \widetilde{\Psi})(\id \ot \pi)(A_k)=
(\id \ot \Psi)(A_k)\\ &=&
(Q_k^*\ot 1_N)(\id\ot\Psi)(U^{\ot k})(Q_k\ot 1_N)=
(Q_k^*\ot 1_N)v^{\ot k}(Q_k\ot 1_N)\\ &=&
a(u_k).
\end{eqnarray*}
The last equality requires particular attention. It is verified if and only if $Q_k\in\Hom(a(u_k),a(u)^{\ot k})$, therefore, in order to complete the proof, we have to check that the map $Q_k$ defined during the proof is in $\Hom(a(u_k),a(u)^{\ot k})$. If $k=0,1$, it is clear. In the general case, for $k\geq 2$, recall that 
$$Q^*_k=\delta^{-\frac{k-1}{2}}(\id_B\ot S_k)\circ T_k=\delta^{-\frac{k-1}{2}}(\id_B\ot S_k)(m_B^{(k)}\ot \id_{B'}^{\ot k})\Sigma_k.$$
We claim that

$\begin{array}{l}
Q^*_k=(\id_B\ot S_k)\circ \delta'^{-\um}(m_B\ot\id_{B'}^{\ot k})\Sigma_{23} \circ 
(\id_{B\ot B'}\ot \delta'^{-\um}(m_B\ot\id_{B'}^{\ot k-1})\Sigma_{23})\circ ...\\ \hspace{10.2cm} ... \circ (\id_{B\ot B'}^{\ot k-2}\ot \delta'^{-\um}(m_B\ot\id_{B'}^{\ot 2})\Sigma_{23}).
\end{array}$

This can be easily verified by evaluating the two formulations of $Q^*_k$ on a general element of $(B\ot B')^{\ot k}$. The equality depends essentially on the associativity of the multiplication.
Moreover, we observe that $\delta'^{-\um}(m_B\ot\id_{B'}^{\ot k})\Sigma_{23}\in\Hom(a(u)\ot a(u^{\ot k-1}),a(u^{\ot k}))$ by Proposition \ref{PropFunctor} $(4)$. Therefore, the linear map $Q_k^*$ can be obtained as composition and tensor product of morphisms. It follows that $Q_k^*\in\Hom(a(u)^{\ot k},a(u_k))$ and $Q_k\in \Hom(a(u_k),a(u)^{\ot k})$. This concludes the proof.\end{proof}

\begin{remark}
We observe that this theorem generalizes the results of Banica and Bichon. In \cite{bb07}, they investigated the free wreath product of two quantum permutation groups and, in the particular case of two quantum symmetric groups, they proved that
$$C(S_{mn}^+)/I\cong C(S_m^+)\ast_w C(S_n^+)$$
where $I\subset C(S^+_{mn})$ is the closed two-sided $\ast$-ideal generated by the relations corresponding to the condition $\id_{\C^n}\otimes \eta_{\C^m}\eta^*_{\C^m}\in\End(U)$ and $U$ is the fundamental representation of $S_{mn}^+$.
\end{remark}

\bibliographystyle{alpha}
\bibliography{biblio}

\noindent
{\sc Pierre FIMA} \\ \nopagebreak
  {Univ Paris Diderot, Sorbonne Paris Cit\'e, IMJ-PRG, UMR 7586, F-75013, Paris, France \\
  Sorbonne Universit\'es, UPMC Paris 06, UMR 7586, IMJ-PRG, F-75005, Paris, France \\
  CNRS, UMR 7586, IMJ-PRG, F-75005, Paris, France \\
\em E-mail address: \tt pierre.fima@imj-prg.fr}

\vspace{0.2cm}

\noindent
{\sc Lorenzo PITTAU} \\ \nopagebreak
  {Universit\'e de Cergy-Pontoise, 95000, Cergy-Pontoise, France\\
  Univ Paris Diderot, Sorbonne Paris Cit\'e, IMJ-PRG, UMR 7586, F-75013, Paris, France \\
  Sorbonne Universit\'es, UPMC Paris 06, UMR 7586, IMJ-PRG, F-75005, Paris, France \\
  CNRS, UMR 7586, IMJ-PRG, F-75005, Paris, France \\
\em E-mail address: \tt lorenzo.pittau@u-cergy.fr}

\end{document}